\documentclass[11pt]{article}
\usepackage[T1]{fontenc}
\usepackage[latin9]{inputenc}
\usepackage[letterpaper,margin=2.0cm]{geometry}
\usepackage{color}
\usepackage{verbatim}
\usepackage{float}
\usepackage{amsmath}
\usepackage{amsthm}
\usepackage{amssymb}
\usepackage{graphicx}
\usepackage{esint}
\usepackage{cite}

\makeatletter
\@ifundefined{showcaptionsetup}{}{%
 \PassOptionsToPackage{caption=false}{subfig}}
\usepackage{subfig}
\makeatother

\allowdisplaybreaks

\providecommand{\tabularnewline}{\\}
\floatstyle{ruled}
\newfloat{algorithm}{tbp}{loa}
\providecommand{\algorithmname}{Algorithm}
\floatname{algorithm}{\protect\algorithmname}

\theoremstyle{plain}
\newtheorem{thm}{\protect\theoremname}
\theoremstyle{plain}
\newtheorem{lem}[thm]{\protect\lemmaname}

\numberwithin{thm}{section}
\numberwithin{equation}{section} 

\usepackage{graphicx}

\newcommand{\ddt}[1]{\frac{\partial}{\partial t}#1}
\newcommand{\inp}[2][]{\left(#1,\, #2\right)}
\newcommand{\gnp}[2][]{\langle#1,\, #2\rangle}

\def\Tc{\mathcal{T}}

\def\Qc{\mathcal{Q}}

\def\X{\mathbb{X}}
\def\W{\mathbb{Q}}
\def\R{\mathbb{R}}
\def\M{\mathbb{M}}
\def\S{\mathbb{S}}
\def\N{\mathbb{N}}

\def\a{\alpha}
\def\s{\sigma}
\def\t{\tau}
\def\g{\gamma}
\def\Gi{\Gamma _{i}}

\def\lH{\lambda_H}
\def\LH{\Lambda_H}
\def\LHi{\Lambda_{H,i}}
\def\lHi{\lambda_{H,i}}
\def\lHu{\lambda_{H}^u}
\def\lHp{\lambda_{H}^p}
\def\m{\mu}

\def\dvr{\operatorname{div}}  

\def\dt{\partial_t}

\def\Quh{\Qc^{u}_h}

\def\Qph{\Qc^{p}_h}
\def\Qphi{\Qc^{p}_{h,i}}

\def\Omg{\Omega}
\def\O{\Omega}
\def\Oi{\Omega_i}
\def\G{\Gamma}
\def\Gn{\Gamma_N}
\def\Gd{\Gamma_D}
\def\dO{{\partial\Omega}}
\def\dOi{{\partial\Omega_i}}

\def\K{{K^{-1}}}

\newtheorem{remark}{Remark}[section]

\newtheorem{assump}{Assumption}

\def\sss{\sigma_{h,i}^{*,n+1}}

\def\zs{z_{h,i}^{*,n+1}}
\def\ps{p_{h,i}^{*,n+1}}
\def\usd{\dot{u}_{h,i}^{*,n+1}}
\def\gsd{\dot{\g}_{h,i}^{*,n+1}}

\def\sb{\bar{\sigma}_{h,i}^{n+1}}

\def\zb{\bar{z}_{h,i}^{n+1}}
\def\pb{\bar{p}_{h,i}^{n+1}}
\def\ubd{\bar{\dot{u}}_{h,i}^{n+1}}
\def\gbd{\bar{\dot{\g}}_{h,i}^{n+1}}

\def\fz{ \phi_{z}}
\def\fs{ \phi_{\sigma}}
\def\fg{ \phi_{\gamma}}
\def\fp{ \phi_{p}}
\def\fu{ \phi_{u}}
\def\flu{\phi_{\lambda^u}}
\def\flp{\phi_{\lambda^p}}

\def\sz{ \psi_{z}}
\def\ss{ \psi_{\sigma}}
\def\sg{ \psi_{\gamma}}
\def\sp{ \psi_{p}}
\def\su{ \psi_{u}}

\def\intg{\int_{0}^{t}}

\def\Ahalf{A^{\frac{1}{2}}}

\def\half{\frac{1}{2}}

\def\sumsubd{\sum_{i=1}^{N}}

\def\AHi{\mathcal{A}_{H,i}^{n+1}}
\def\AH{\mathcal{A}_{H}^{n+1}}
\def\GH{G_{H}^{n+1}}

\def\BHi{\mathcal{\beta}_{H,i}^{k}}
\def\PHi{\mathcal{\phi}_{H,i}^{k}}

\providecommand{\lemmaname}{Lemma}
\providecommand{\theoremname}{Theorem}

\begin{document}

\title{Multiscale mortar mixed finite element methods for the Biot system of poroelasticity}

\author{Manu Jayadharan\thanks{Engineering Sciences and Applied Mathematics, Northwestern University, Evanston, IL 60208, USA; {\tt manu.jayadharan@northwestern.edu}.} 
  \and Ivan Yotov\thanks{Department of Mathematics, University of
    Pittsburgh, Pittsburgh, PA 15260, USA; {\tt yotov@math.pitt.edu}.
This work is partially supported by NSF grants DMS 2111129 and DMS 2410686.}
}

\maketitle

\begin{abstract}
  We develop a mixed finite element domain decomposition method on non-matching grids for the Biot system of poroelasticity. A displacement--pressure vector mortar function is introduced on the interfaces and utilized as a Lagrange multiplier to impose weakly continuity of normal stress and normal velocity. The mortar space can be on a coarse scale, resulting in a multiscale approximation. We establish existence, uniqueness, stability, and error estimates for the semi-discrete continuous-in-time formulation under a suitable condition on the richness of the mortar space. We further consider a fully-discrete method based on the backward Euler time discretization and show that the solution of the algebraic system at each time step can be reduced to solving a positive definite interface problem for the composite mortar variable. A multiscale stress--flux basis is constructed, which makes the number of subdomain solves independent of the number of iterations required for the interface problem, and weakly dependent on the number of time steps. We present numerical experiments verifying the theoretical results and illustrating the multiscale capabilities of the method for a heterogeneous benchmark problem.
\end{abstract}

\section{Introduction}

In this paper we develop and study a domain decomposition method for the quasistatic Biot system of poroelasticity \cite{biot1941general} using mixed finite element subdomain discretizations with non-matching grids along the interfaces. The Biot system models flow of viscous fluids through deformable porous media. The system has a wide rage of applications, including in the geosciences, such as earthquakes, landslides, groundwater cleanup, and hydraulic fracturing, as well as in biomedicine, such as arterial flows and biological tissues. The model consists of an equilibrium equation for the solid medium coupled with a mass balance equation for the fluid flow through the medium. Various numerical methods for the Biot system have been developed in the literature, considering two-field displacement--pressure formulations \cite{Gaspar-FD-Biot,Jan-SINUM-Biot,Zik-MINI}, three-field displacement--pressure--Darcy velocity formulations \cite{phillips2007coupling1,Yi-Biot-locking,Zik-nonconf,Lee-Biot-three-field,phillips-DG,Whe-Xue-Yot-Biot}, three-field displacement--pressure--total pressure formulations
\cite{Lee-Mardal-Winther,ORB}, and four-field stress--displacement--pressure--Darcy velocity mixed formulations \cite{Yi-Biot-mixed,Ahmed-apost-CMAME}. In this work we consider the five-field weakly symmetric stress--displacement--rotation--pressure--Darcy velocity formulation \cite{Lee-Biot-five-field,msfmfe-Biot}. The four-field and five-field formulations lead to mixed finite element (MFE) approximations, which exhibit local mass and momentum conservation, accurate normal-continuous Darcy velocity and solid stress, as well as robust and locking-free behavior for a wide range of physical parameters. An additional advantage of the five-field MFE method is that it can be reduced to a positive definite cell-centered scheme for the pressure and displacement only, as it is done in the multipoint
stress--multipoint flux MFE method developed in \cite{msfmfe-Biot}, through the use of a vertex quadrature rule and local elimination of some of the variables. We note that our analysis for the weakly symmetric formulation can be carried over to the strongly symmetric formulation found in \cite{Yi-Biot-mixed}.

Numerical methods for the Biot system of poroelasticity usually lead to
large algebraic systems, due to the coupling of unknowns, as well as the size of the domain and the wide range of scales associated with practical applications. Domain decomposition methods \cite{Toselli-Widlund,QV} are commonly used for solving large systems resulting from discretizations of partial differential equations, as they lead to parallel and efficient solution algorithms. In this work we focus on non-overlapping domain decomposition methods, where the domain is split into non-overlapping subdomains and the continuity of the solution variables at the subdomain interfaces is enforced through a suitable interface Lagrange multiplier. The global problem can be reduced to solving iteratively an interface problem, involving the solution of smaller subdomain systems at each iteration, which can be performed in parallel. Despite the extensive studies of numerical methods for the Biot system of poroelasticity, there have been relatively few results on domain decomposition methods for this problem and they have been mostly based on the two-field displacement--pressure formulation \cite{girault2011domain,HoracioNurbs1,FlorezDD,GosseletMonoDD}. To the best of our knowledge, the only paper on domain decomposition for Biot with a mixed formulation is our previous work \cite{dd-biot}, which is based on the five-field mixed formulation with weak stress symmetry. Two types of methods are developed in \cite{dd-biot}. One is a monolithic domain decomposition method, which involves solving the Biot system on each subdomain. The second is a partitioned method, which splits the Biot system into solving separate elasticity and flow equations, and applies domain decomposition for each of the equations. The developments in \cite{dd-biot} are motivated by earlier works on non-overlapping domain decomposition methods
for MFE discretizations of Darcy flow \cite{glowinski1988domain,cowsar1995balancing,arbogast2000mixed}
and elasticity \cite{eldar_elastdd}.

The domain decomposition methods in \cite{dd-biot} are limited to subdomain grids that match at the interfaces. In this paper we generalize the work in \cite{dd-biot} to enable the use of non-matching subdomain grids through the use of mortar finite elements \cite{arbogast2000mixed,APWY,Fritz-mortar-elast,Kim-elast-BDDC,pencheva2003balancing,eldar_elastdd}. This generality provides the flexibility to use different grid resolution in different subdomains, as well as a coarser mortar space, resulting in a multiscale approximation. We refer to the method as a multiscale mortar mixed finite element (MMMFE) method. The MMMFE method has been studied for mixed formulations of scalar elliptic equations in \cite{arbogast2000mixed,APWY,ganis2009implementation} and for weakly-symmetric mixed elasticity in \cite{eldar_elastdd}. Following the monolithic domain decomposition method from \cite{dd-biot}, we utilize a physically heterogeneous Lagrange multiplier vector consisting of interface displacement and pressure variables to impose weakly the continuity of the normal components of stress and velocity, respectively. In contrast to \cite{dd-biot}, we choose the Lagrange multiplier vector from a space of mortar finite elements defined on a separate interface grid, which allows for handling non-matching subdomain grids through projections from and onto the mortar finite element space. This also allows for the mortar space to be on a coarser scale $H$, see \cite{PWY,APWY,ganis2009implementation}, compared to a finer subdomain
grid size $h$. The multiscale capability adds an extra layer of flexibility over the methods from \cite{dd-biot}. 

The main contributions of this paper are as follows. We first consider the semi-discrete continuous-in-time formulation and establish existence, uniqueness, and stability of the MMMFE method for the Biot system, employing the theory of degenerate evolutionary systems of partial differential equations with monotone operators. For the solvability of the associated resolvent problem, under a condition on the richness of the mortar finite element space, we establish
an inf-sup condition for the mortar space, as well as inf-sup conditions for the stress and velocity spaces with weak interface continuity of normal components. Next, we establish a priori error estimates for the stress, displacement, rotation, pressure, and Darcy velocity, as well as the displacement and pressure mortar variables in their natural norms. We then consider a fully-discrete method based on the backward Euler time discretization. We show that the solution of the algebraic system at each time step can be reduced to solving a positive definite interface problem for the composite displacement--pressure mortar variable. Motivated by the multiscale flux basis from \cite{ganis2009implementation} and the multiscale stress basis from \cite{eldar_elastdd}, we propose the construction and use of a multiscale stress--flux basis, which makes the number of subdomain solves independent of the number of iterations required for the interface problem. Moreover, since the basis can be reused at each time step, the total number of subdomain solves depends weakly on the number of time steps. This illustrates that the multiscale basis results in a significant reduction of computational cost in the case of time-dependent problems. Finally, we present the results of several numerical tests designed to illustrate the well-posedness, stability, and accuracy of the proposed MMMFE method. We also consider a test based on data from the Society of Petroleum Engineers SPE10 benchmark, illustrating the multiscale capabilities of the method and the advantages of using a multiscale basis.

The rest of the paper is organized as follows. Section~\ref{sec:Math-Formulation-mortar}
introduces the model problem and its domain decomposition mortar mixed finite element approximation. Various interpolation and projection operators are presented in Section~\ref{sec:interp}, where discrete inf-sup stability bounds are also obtained. 
Well-posedness of the semi-discrete method is established in Section~\ref{sec:Analysis-of-MMMFE}, followed by error analysis in Section~\ref{sec:error}. In Section \ref{sec:Implementation-Mortar} we discuss the fully discrete method, the non-overlapping domain decomposition algorithm based on a reduction to an interface problem, and the construction of the multiscale stress--flux basis. Numerical results are reported in Section~\ref{sec:Numerical-results-mortar}. The paper ends with some concluding remarks in Section~\ref{sec:concl}.

\section{Formulation of the method}
\label{sec:Math-Formulation-mortar}

In this section, we introduce the mathematical model of interest and its
mixed finite element approximation. We also develop the framework for the multiscale mortar mixed finite element domain decomposition method. Finally, we introduce the weakly continuous
normal stress and velocity spaces and reformulate the MMMFE method in terms of these spaces. 

\subsection{Mathematical formulation of the model problem}

Let $\Omega\subset\mathbb{R}^{d}$, $d=2,3$ be a simply connected
domain. We use the notation
$\M$, $\S$ and $\N$ for the spaces of $\ensuremath{d\times d}$
matrices, symmetric matrices, and skew-symmetric matrices, respectively,
over the field of real numbers. Let $I\in\S$ denote the $d\times d$
identity matrix. The partial derivative operator with respect to time, $\frac{\partial
}{\partial t}$, is often abbreviated to $\partial_t$. $C$ denotes a generic positive
constant that is independent of the discretization parameters
$h$ and $H$. Throughout the paper,
the divergence operator is the usual divergence for vector fields,
which produces vector field when applied to matrix field by taking
the divergence of each row. 

For a set $G\subset\R^{d},$ the $L^{2}(G)$ inner product and norm
are denoted by $ (\cdot,\cdot)_G$ and $\|\cdot\|_G$, respectively,
for scalar, vector, or tensor valued functions. For any real number $r$,
$\|\cdot\|_{r,G}$ denotes the $H^{r}(G)$-norm.
We omit subscript $\ensuremath{G}$ if $\ensuremath{G=\Omega}.$ For
a section of the domain or element boundary $S\subset\R^{d-1}$, we
write $\gnp[\cdot]{\cdot}_{S}$ and $\|\cdot\|_{S}$ for the $L^{2}(S)$
inner product (or duality pairing) and norm, respectively. We will
also use the spaces 
\begin{align*}
 & H(\dvr;G)=\{\zeta\in L^{2}(G,\R^{d}):\dvr\zeta\in L^{2}(G)\},\\
 & H(\dvr;G,\M)=\{\t\in L^{2}(G,\M):\dvr\t\in L^{2}(G,\R^{d})\},
\end{align*}
with the norm $\|\t\|_{\dvr,G}=\left(\|\t\|_G^{2}+\|\dvr\t\|_G^{2}\right)^{1/2}.$

Given a vector field $f$ representing body
forces and a source term $g$, we consider the quasi-static Biot system of poroelasticity \cite{biot1941general}:
\begin{align}
-\dvr\s(u) & =f, & \text{in \ensuremath{\Omega \times (0,T]}},\label{biot-1}\\
\K z+\nabla p & =0, & \text{in \ensuremath{\Omega \times (0,T]}},\label{biot-2}\\
\ddt(c_{0}p+\a\dvr u)+\dvr z & =g, & \text{in \ensuremath{\Omega \times (0,T]}},\label{biot-3}
\end{align}
where $u$ is the displacement, $p$ is the fluid pressure, $z$ is the 
Darcy velocity, and $\sigma$ is the poroelastic stress, defined as
\begin{equation}\label{stress-comb}
\sigma = \sigma_e - \alpha p I.
\end{equation}
Here $0 < \alpha \le 1$ is the
Biot-Willis constant, and $\sigma_e$ is the elastic stress satisfying the
stress-strain relationship
\begin{equation}\label{stress-strain}
A\s_e = \epsilon(u), \quad \epsilon(u):=\frac{1}{2}\left(\nabla u+(\nabla u)^{T}\right),
\end{equation}
where $A$ is the compliance tensor, which is a symmetric,
bounded and uniformly positive definite linear operator acting from
$\ensuremath{\S\to\S}$, extendible to $\M\to\M$. In particular, there exist constants
$0 < a_{\min} \le a_{\max} < \infty$ such that
\begin{equation}\label{eq:coercivity-elast-1}
  \text{for a.e. } x \in \Omg, \quad a_{\min}\, \tau:\tau \le A(x)\tau:\tau \le a_{\max}\, \tau:\tau, \quad
  \forall \tau \in \M.
\end{equation}
In the special case of homogeneous and isotropic body, $A$ is given by,
\begin{equation}
A\sigma_e = \frac{1}{2\mu}\left(\sigma_e - \frac{\lambda}{2\mu + d\lambda}\operatorname{tr}(\sigma_e)I\right),\label{A-defn}
\end{equation}
where $\mu>0$ and $\lambda\ge0$
are the Lam\'e coefficients. In this case, $\sigma_{e}(u)=2\mu\epsilon(u)+\lambda\dvr u\,I$.
Finally, $c_{0} > 0$ is the mass storativity and $K$ stands for the conductivity
tensor, which equals to the permeability of the media divided by the fluid viscosity. It is
spatially dependent, symmetric, and uniformly bounded and
positive definite, i.e, for constants $0 < k_{\min} \le k_{\max}
< \infty$,
\begin{equation}\label{eq:coercivity-flow-1}
\text{for a.e. } x \in \Omg, \quad k_{\min}\,\zeta\cdot\zeta \le K(x) \zeta\cdot\zeta \le k_{\max}\,\zeta\cdot \zeta,
\quad \forall \zeta \in \R^d.
\end{equation}
To close the system, we impose the boundary conditions 
\begin{align}
u & = 0 & \text{on \ensuremath{\Gd^u \times (0,T]}} & , & \s n & =0
& \text{on \ensuremath{\Gn^\sigma  \times (0,T]}},\label{biot-bc-1}\\
p & = 0 & \text{on \ensuremath{\Gd^p \times (0,T]}} & , & z\cdot n
& =0 & \text{on \ensuremath{\Gn^z \times (0,T]}},\label{biot-bc-2}
\end{align}
where $\ensuremath{\Gd^u\cup\Gn^\sigma=\Gd^p\cup\Gn^z=\partial\Omg}$ and 
$n$ is the outward unit normal vector field on $\partial \O$,
along with the initial condition $p(x,0)=p_0(x)$ in $\Omega$.
Compatible initial data for the rest of the variables can be obtained from
($\ref{biot-1}$) and ($\ref{biot-2}$) at $t=0$. 
Well posedness analysis for this system can be found in \cite{SHOWALTER2000310}.

We consider a five-field mixed variational formulation for ($\ref{biot-1}$)--($\ref{biot-bc-2}$)
\cite{Lee-Biot-five-field,msfmfe-Biot}. It uses a rotation
Lagrange multiplier $\gamma:=\frac{1}{2}\left(\nabla u-\nabla u^{T}\right)\in\N$
to impose weakly the symmetry of the stress tensor
$\s$. The formulation reads: find $\ensuremath{(\sigma,u,\gamma,z,p):[0,T]\to\X\times V\times\mathbb{Q}\times Z\times W}$
such that $p(0)=p_{0}$ and for a.e. $t\in(0,T),$
\begin{align}
& \inp[A(\sigma + \a p I)]{\tau} + \inp[u]{\dvr{\tau}}+\inp[\gamma]{\tau}=0, & \forall\tau\in\X,\label{eq:cts1-1}\\
& \inp[\dvr{\sigma}]{v}=-\inp[f]{v}, & \forall v\in V,\label{eq:cts2-1}\\
& \inp[\sigma]{\xi}=0, & \forall\xi\in\W,\label{eq:cts3-1}\\
& \inp[\K z]{q}-\inp[p]{\dvr{q}}=0, & \forall q\in Z,\label{eq:cts4-1}\\
& c_{0}\inp[\dt{p}]{w}+\a\inp[\dt A(\sigma + \a pI)]{wI}
+\inp[\dvr{z}]{w}=\inp[g]{w}, & \forall w\in W,\label{eq:cts5-1}
\end{align}
where
\begin{align*}
&\X = \big\{ \t\in H (\dvr;\Omega,\M) : \t\,n = 0 \text{ on } \Gn^\sigma  \big\},
\quad V = L^2 (\Omega, \R^d), \quad \W = L^2 (\Omega, \N), \\
&Z = \big\{ q\in H (\dvr;\Omega) : q\cdot n = 0 \text{ on } \Gn^z  \big\},
\quad W = L^2 (\Omega).
\end{align*}
It was shown in \cite{msfmfe-Biot} that the system
(\ref{eq:cts1-1})--(\ref{eq:cts5-1}) is well posed.

\subsection{The semi-discrete multiscale mortar mixed finite element method}\label{sec:method}

Let $\Omega=\cup_{i=1}^{N}\Omega_{i}$
be a union of non-overlapping polygonal subdomains.
Let $\G_{i,j}=\dO_{i}\cap\dO_{j}$, $\G=\cup_{i,j=1}^{N}\G_{i,j}$,
and $\G_{i}=\dO_{i}\cap\G=\dO_{i}\setminus\dO$. Let $\Tc_{h,i}$ be a shape regular simplicial or rectangular finite element partition on $\Omg_i$ with maximal
element diameter $h$. The partitions are not required to match along the subdomain interfaces. 
For $1\le i\le N$, let $\mathbb{X}_{h,i}\times V_{h,i}\times\mathbb{\mathbb{Q}}_{h,i}\times Z_{h,i}\times W_{h,i} \subset \mathbb{X}_{i}\times V_{i}\times\mathbb{Q}_{i}\times Z_{i}\times W_{i}$
be a family of suitable mixed finite element spaces defined on subdomain
$\Omega_{i}$, where, for a space $U$ on $\Omg$, $U_i = U|_{\Omg_i}$. The elasticity discretizations $\X_{h,i}\times V_{h,i}\times\mathbb{Q}_{h,i}$
can be chosen from any of the stable finite element triplets for linear elasticity with weakly imposed stress symmetry. Examples of such triplets include \cite{ArnAwaQiu,arnold2007mixed,Awanou-rect-weak,cockburn2010new,BBF-reduced,msmfe-quads,lee2016towards}. These spaces satisfy the inf-sup condition
\begin{align}
\forall v \in V_{h,i}, \, \xi \in \W_{h,i}, \quad
\|v\|_{\Omg_i} + \|\xi\|_{\Omg_i} & \le C \sup_{0\ne\tau\in\mathbb{X}_{h,i}} \frac{\inp[\dvr{\tau}]{v}_{\Omg_i} + \inp[\tau]{\xi}_{\Omg_i}}{\|\tau\|_{\text{div},\Omg_i}}. \label{eq:inf-sup-elast-1}
\end{align}
The flow discretizations $Z_{h}\times W_{h}$ can be chosen from any
of the stable pressure--velocity pairs of MFE spaces such as the Raviart-Thomas
($\mathcal{RT}$) or Brezzi-Douglas-Marini ($\mathcal{BDM}$) spaces \cite{brezzi1991mixed}.
These spaces satisfy the inf-sup condition
\begin{align}
\forall w \in W_h, \quad \|w\| & \le
C\sup_{0\ne q\in Z_{h,i}}\frac{(\dvr q,w)_{\Omg_i}}{\|q\|_{\text{div},\Omg_i}}\label{eq:inf-sup-darcy-1}.
\end{align}
We define the global finite element spaces on $\Omg$ as follows: 
\[
\mathbb{X}_{h}=\bigoplus_{1\le i\le N}\mathbb{X}_{h,i},\quad V_{h}=\bigoplus_{1\le i\le N}V_{h,i},\quad\mathbb{Q}_{h}=\bigoplus_{1\le i\le N}\mathbb{\mathbb{Q}}_{h,i},\quad Z_{h}=\bigoplus_{1\le i\le N}Z_{h,i},\quad W_{h}=\bigoplus_{1\le i\le N}W_{h,i}.
\]
The spaces $V_h$, $\mathbb{Q}_{h}$, and $W_h$ are equipped with $L^2(\Omg)$-norms. The spaces $\mathbb{X}_{h}$ and $Z_h$ are equipped with the norms
$$
\|\tau\|_{\X_h}^2 := \|\tau\|^2 + \|\dvr_h \tau\|^2 \quad \mbox{and} \quad
\|\zeta\|_{Z_h}^2 := \|\zeta\|^2 + \|\dvr_h \zeta\|^2,
$$
where for simplicity we define $\dvr_h \varphi |_{\Omega_i} := \dvr (\varphi|_{\Omega_i})$. We note that functions in $\mathbb{X}_{h}$ and $Z_{h}$ do not have continuity of the normal components across the subdomain interfaces. This discontinuity is addressed using Lagrange multipliers defined
on mortar spaces on the interface $\Gamma$, which approximate the traces of the displacement vector and the pressure. The mortar spaces satisfy suitable coarseness conditions, which will be discussed in the later sections. Let $\Tc_{H,i,j}$ be a shape regular 
finite element partition of $\Gamma_{i,j}$ consisting of simplices or 
quadrilaterals in $d-1$ dimensions with maximal element diameter
$H$. Let $\Lambda_{H,i,j}^{u}\subset L^{2}(\Gamma_{i,j};\R^d)$
and $\Lambda_{H,i,j}^{p}\subset L^{2}(\Gamma_{i,j})$ be mortar finite
element spaces on $\Gamma_{i,j}$ representing the displacement and
pressure Lagrange multipliers, respectively. We assume that these
mortar spaces contain either continuous or discontinuous piecewise polynomials on $\Tc_{H,i,j}$. The global mortar finite element space on the union of subdomain interfaces $\Gamma$ is defined as
\[
\Lambda_{H}^{u}=\bigoplus_{1\le i<j\le N}\Lambda_{H,i,j}^{u}, \quad \Lambda_{H}^{p}=\bigoplus_{1\le i<j\le N}\Lambda_{H,i,j}^{p}.
\]
 
The semi-discrete multiscale mortar mixed finite element method for the 
Biot problem \eqref{eq:cts1-1}--\eqref{eq:cts5-1} is obtained by testing the equations on each subdomain and integrating by parts, which results in interface terms involving the displacement and pressure Lagrange multipliers. The method reads as follows: find $(\sigma_{h},u_{h},\g_{h},z_{h},p_{h},\lambda_{H}^u,\lambda_H^p):[0,T]\to\X_{h}\times V_{h}\times\mathbb{Q}_{h}\times Z_{h}\times W_{h}\times\Lambda_{H}^u\times\Lambda_H^p$
such that for a.e. $t\in(0,T)$,
\begin{align}
  & \inp[A\left(\sigma_{h}+\a p_{h}I\right)]{\tau} + \inp[u_{h}]{\dvr_h{\tau}} + \inp[\gamma_{h}]{\tau}
  - \sum_{i=1}^N \gnp[\lambda_{H}^{u}]{\t\,n_{i}}_{\Gamma_{i}} = 0, &  & \forall\tau\in\X_{h},\label{eq:monodd-mmmfe1}\\
& \inp[\dvr_h{\sigma_{h}}]{v} = -\inp[f]{v}, &  & \forall v\in V_{h},\label{eq:monodd-mmmfe2}\\
& \inp[\sigma_{h}]{\xi} = 0, &  & \forall\xi\in\mathbb{Q}_{h},\label{eq:monodd-mmmfe3}\\
& \inp[\K z_{h}]{\zeta} -\inp[p_{h}]{\dvr_h{\zeta}} +
\sum_{i=1}^N \gnp[\lambda_{H}^{p}]{\zeta\cdot n_{i}}_{\Gamma_{i}} = 0, &  & \forall\zeta\in Z_{h},\label{eq:monodd-mmmfe4}\\
& c_{0}\inp[\dt p_{h}]{w} + \a\inp[\dt A\left(\sigma_{h} + \a p_{h}I\right)]{wI} + \inp[\dvr_h{z_{h}}]{w} = \inp[g]{w}, &  & \forall w \in W_{h},\label{eq:monodd-mmmfe5}\\
& \sum_{i=1}^{N}\gnp[\sigma_{h,i}\,n_{i}]{\mu^{u}}_{\G_{i}}=0, &  & \forall\mu^{u}\in\Lambda_{H}^{u},\label{eq:monodd-mmmfe6}\\
& \sum_{i=1}^{N}\gnp[z_{h,i}\cdot n_{i}]{\mu^{p}}_{\G_{i}}=0, &  & \forall\mu^{p}\in\Lambda_{H}^{p},\label{eq:monodd-mmmfe7}
\end{align}
where $n_{i}$ is the outward unit normal vector field on $\Omega_{i}$.
Note that equations (\ref{eq:monodd-mmmfe6})$\--$(\ref{eq:monodd-mmmfe7})
enforce weak continuity of the normal components of the stress
tensor and velocity vector, respectively, across the interface $\Gamma$.

\begin{remark}
The method requires discrete initial data $p_{h,0}$ and $\sigma_{h,0}$, which is obtained from the continuous initial data using an elliptic projection. Details are provided in Section~\ref{sec:well-posed}.
\end{remark}

\subsection{Weakly continuous normal stress and velocity formulation}

For the purpose of the analysis, we consider an equivalent formulation of \eqref{eq:monodd-mmmfe1}--\eqref{eq:monodd-mmmfe7} in the spaces of stress and velocity with weakly continuous normal components. Let
\[
\X_{h,0}=\left\{ \tau\in\X_{h}:\sum_{i=1}^{N}\gnp[\tau n_{i}]{\mu^{u}}_{\G_{i}}=0,\,\,\,\forall\mu^{u}\in\Lambda_{H}^{u}\right\} 
\]
and 
\[
Z_{h,0}=\left\{ \zeta\in Z_{h}:\sum_{i=1}^{N}\gnp[\zeta\cdot n_{i}]{\mu^{p}}_{\G_{i}}=0,\,\,\,\forall\mu^{p}\in\Lambda_{H}^{p}\right\} .
\]
Then \eqref{eq:monodd-mmmfe1}--\eqref{eq:monodd-mmmfe7} can be restated as follows:
find $\ensuremath{(\sigma_{h},u_{h},\gamma_{h},z_{h},p_{h})}:[0,T]\to \left(\X_{h,0},V_h,\mathbb{Q}_{h},Z_{h,0},W_h \right)$
such that
\begin{align}
& \inp[A\left(\sigma_{h}+\a p_{h}I\right)]{\tau}+\inp[u_{h}]{\dvr_h{\tau}}_{\Omega_{i}}+\inp[\gamma_{h}]{\tau}=0, & \forall\tau\in\X_{h,0},\label{eq:fe-mono-weak-1}\\
& \inp[\dvr_h{\sigma_{h}}]{v}_{\Omega_{i}}=-\inp[f]{v}{}, & \forall v\in V_{h},\label{eq:fe-mono-weak-2}\\
& \inp[\sigma_{h}]{\xi}=0, & \forall\xi\in\mathbb{Q}_{h},\label{eq:fe-mono-weak-3}\\
& \inp[\K z_{h}]{\zeta}-\inp[p_{h}]{\dvr_h{\zeta}}_{\Omega_{i}}=0, & \forall\zeta\in Z_{h,0},\label{eq:fe-mono-weak-4}\\
& c_{0}\inp[\dt p_{h}]{w}+\a\inp[\dt A\left(\sigma_{h}+\a p_{h}I\right)]{wI}+\inp[\dvr_h{z_{h}}]{w}{}_{\Omega_{i}}=\inp[g]{w}, & \forall w\in W_{h}.\label{eq:fe-mono-weak-5}
\end{align}
Note that constructing basis functions for the spaces $\X_{h,0}$
and $Z_{h,0}$ is difficult and we use the above formulation only
for the sake of analysis.

\section{Interpolation and projection operators and discrete inf-sup conditions}
\label{sec:interp}

In this section we discussing various interpolation and projection operators useful in the analysis. We then prove inf-sup stability bounds for the interface jump operators and the weakly continuous stress $\X_{h,0}$ and velocity $Z_{h,0}$ spaces under an appropriate condition on the mortar space $\LH$.

\subsection{Interpolation operators}

Let $\mathcal{Q}_{h,i}^{u}: L^{2}(\dOi,\R^d)to\X_{h,i}n_{i}$
and $\mathcal{Q}_{h,i}^{p}:L^{2}(\dOi)\to Z_{h,i}\cdot n_{i}$
be the $L^2$-projection operators such that for any $\phi_{u}\in L^{2}(\dOi,\R^d)$
and $\phi_{p}\in L^{2}(\dOi)$,

\begin{align}
& \left<\phi_{u}-\mathcal{Q}_{h,i}^{u}\phi_{u},\,\tau n_{i}\right>_{\dOi}=0, & \forall\tau\in\X_{h,i},\label{eq:motor-project-1}\\
& \left<\phi_{p}-\mathcal{Q}_{h,i}^{p}\phi_{p},\,\zeta\cdot n_{i}\right>_{\dOi}=0, & \forall\zeta\in Z_{h,i}.\label{eq:motor-project-2}
\end{align}
For any inf-sup stable pair of finite element spaces $\X_{h,i}\times V_{h,i}$
with $\text{div\,}\X_{h,i}=V_{h,i}$, there exists a mixed canonical
interpolant \cite{brezzi1991mixed}, $\Pi_{i}^{\s}:H^{\epsilon}(\Omega_{i},\M)\cap\X_{i}\to\X_{h,i}$,
for any $\epsilon>0$, such that for any $\tau\in H^{\epsilon}(\Omega_{i},\M)\cap\X_{h,i}$,
\begin{align}
& \left(\text{div\,}(\Pi_{i}^{\s}\tau-\tau),\,v\right)_{\Oi}=0, & \forall v\in V_{h,i},\label{eq:elast-project-1}\\
& \left<(\Pi_{i}^{\s}\tau-\tau)n_{i},\,\hat{\tau} n_{i}\right>_{\Gi}=0, & \forall\hat{\tau}\in\X_{h,i},\label{eq:elast-project-2}\\
& \|\Pi_{i}^{\s}\tau\|_{\dvr,\Oi}\le C\left(\|\tau\|_{H^{\epsilon}(\Oi)}+\|\dvr\tau\|_{\Oi}\right).
\end{align}
Similarly for any inf-sup stable pair $Z_{h,i}\times W_{h,i}$ with
$\text{div }Z_{h,i}=W_{h,i}$, there exists a mixed canonical interpolant
$\Pi_{i}^{z}: H^{\epsilon}(\Omega_{i},\R^d) \cap Z_{i}\to Z_{h,i}$
such that for any $\zeta\in H^{\epsilon}(\Omega_{i},\R^d)\cap Z_{i}$,
\begin{align}
& \left(\text{div\,}(\Pi_{i}^{z}\zeta-\zeta),\,w\right)_{\Oi}=0, & \forall w\in W_{h,i},\label{eq:darcy-project-1}\\
& \langle(\Pi_i^{z}\zeta-\zeta)\cdot n_i,\,\hat{\zeta}\cdot n_i\rangle_{\Gi}=0, & \forall\hat{\zeta}\in Z_{h,i},\label{eq:darcy-project-2}\\
& \|\Pi_i^{z}\zeta\|_{\dvr,\Oi}\le C\left(\|\zeta\|_{H^{\epsilon}(\Oi)}+\|\dvr\zeta\|_{\Oi}\right).\label{eq:darcy-project-3}
\end{align}
Let ${\cal P}_{h,i}^{p}:L^{2}(\Oi)\to W_{h,i}$ denote the $L^2$-orthogonal projection
such that for any $w\in L^{2}(\Oi)$,
\begin{align}
& \left({\cal P}_{h,i}^{p}w-w,\,\hat{w}\right)_{\Oi}=0, & \forall\hat{w}\in W_{h,i}.\label{eq:pressure-project}
\end{align}
Let ${\cal P}_{h,i}^{u}: L^{2}(\Oi,\R^d)\to V_{h,i}$ denote the $L^2$-orthogonal projection such that for any $v\in L^{2}(\Oi,\R^d)$, 
\begin{align}
& \left({\cal P}_{h,i}^{u}v-v,\,\hat{v}\right)_{\Oi}=0, & \forall\hat{v}\in V_{h,i}.\label{eq:disp-project-1}
\end{align}
Let ${\cal R}_{h,i}:L^{2}(\Oi,\mathbb{N})\to\mathbb{Q}_{h,i}$ denote the $L^2$-orthogonal projection such that for any $\xi\in L^{2}(\Oi,\mathbb{N})$,
\begin{align}
& \left({\cal \cal R}_{h,i}\xi-\xi,\,\hat{\xi}\right)_{\Oi}=0, & \forall\hat{\xi}\in\mathbb{Q}_{h,i}.\label{eq:disp-project-1-1}
\end{align}
We will use an elliptic projection operator onto $\X_{h,i}$ \cite{eldar_elastdd}. Define $\hat{\Pi}_{i}^{\s}:H^{\epsilon}(\Omega_{i},\M)\cap\X_{i}\to\X_{h,i}$ as the
operator that takes $\s\in H^{\epsilon}(\Omega_{i},\M)\cap\X_{i}$ to its finite element
approximation $\hat{\sigma}$ via the solution of the following Neumann problem: for any $\sigma\in H^{\epsilon}(\Omega_{i},\M)$, find $(\hat{\sigma},\hat{u},\hat{\g})\in\X_{h,i}\times V_{h,i}\times\mathbb{Q}_{h,i}$
such that 
\begin{align}
  & \inp[\hat{\sigma}]{\tau}_{\Omega_{i}}+\inp[\hat{u}]{\dvr{\tau}}_{\Omega_{i}}
  + \inp[\hat{\gamma}]{\tau}_{\Omega_{i}}=\inp[\s]{\tau}_{\Oi}, &  & \forall\tau\in\X^0_{h,i},
 \label{ell-proj1} \\
& \inp[\dvr{\hat{\sigma}}]{v}_{\Omega_{i}}=\inp[\dvr{\s}]{v}_{\Omega_{i}}, &  & \forall v\in V_{h,i}, \label{ell-proj2} \\
& \inp[\hat{\sigma}]{\xi}_{\Omega_{i}}=\inp[\sigma]{\xi}_{\Omega_{i}}, &  & \forall\xi\in\mathbb{Q}_{h,i}, \label{ell-proj3} \\
& \hat{\sigma}n_{i}=(\Pi_{i}^{\s}\sigma)n_{i}\,\,\,\,\text{on \ensuremath{\dO_{i}}}, \label{ell-proj4} 
\end{align}
where $\X^0_{h,i} = \left\{\tau\in \X_{h,i}: \tau n_i = 0 \mbox{ on } \dO_{i}\right\}$. More details on the well-posedness and properties of $\hat{\Pi}_{i}^{\s}$ can be found in \cite{eldar_elastdd}. In particular, the following bounds hold:
\begin{align*}
& \|\s-\hat{\Pi}_{i}^{\s}\s\|_{\Oi}\le C\|\s-\Pi_{i}\s\|_{\Oi}, & \s\in H^{1}\left(\Oi,\mathbb{M}\right),\\
& \|\hat{\Pi}_{i}^{\s}\s\|_{\dvr,\Oi}\le C\left(\|\s\|_{H^{\epsilon}(\Oi)}+\|\dvr\s\|_{\Oi}\right). & \s\in H^{\epsilon}(\Oi,\mathbb{M})\cap\X_{i},\,0<\epsilon\le1.
\end{align*}

We also use the Scott-Zhang interpolants (see \cite{Scott-Zhang}) ${\cal I}_{H}^{u}: H^1(\Gamma) \to \Lambda_{H}^{u} \cap C(\Gamma)$ and ${\cal I}_{H}^{p}: H^1(\Gamma) \to \Lambda_{H}^{p} \cap C(\Gamma)$, defined to preserve the trace on $\partial \Gamma$ for functions that are zero on $\partial \Gamma$. 

Let the finite element spaces $\mathbb{X}_{h,i},\,V_{h,i},\,\mathbb{Q}_{h,i},\,Z_{h,i}$,
$W_{h,i}$, and $\Lambda_{H,i,j}$ contain polynomials of degree less than or equal to
$k\ge1,\,l\ge0,\,j\ge0,\,r\ge0$, $s\ge0$, and $m \ge 0$, respectively.
The operators defined above satisfy the following approximation bounds:
\begin{align}
  \|\psi-{\cal I}_{H}^{u}\psi\|_{t,\Gamma_{i,j}} & \le CH^{\hat{m}-t}\|\psi\|_{\hat{m},\Gamma_{i,j}}, &
  0\le t\le1,  
  \,\,
t\le\hat{m}\le m+1,
  \label{eq:inter-proj-1}\\
  \|\psi-{\cal I}_{H}^{p}\psi\|_{t,\Gamma_{i,j}} & \le CH^{\hat{m}-t}\|\psi\|_{\hat{m},\Gamma_{i,j}}, &
  0\le t\le1,  
  \,\,
t\le\hat{m}\le m+1,
  \label{eq:inter-proj-1a}\\
\|v-{\cal P}_{h,i}^{u}v\|_{\Omega_{i}} & \le Ch^{\hat{l}}\|v\|_{\hat{l},\Oi}, & 0\le\hat{l}\le l+1,\label{eq:inter-proj-3}\\
\|\zeta-{\cal P}_{h,i}^{p}\zeta\|_{\Omega_{i}} & \le Ch^{\hat{s}}\|\zeta\|_{\hat{s},\Oi}, & 0\le\hat{s}\le s+1,\label{eq:inter-proj-4}\\
\|\xi-{\cal R}_{h,i}\xi\|_{\Oi} & \le Ch^{\hat{j}}\|\xi\|_{\hat{j},\Oi}, & 0\le\hat{j}\le j+1,\label{eq:inter-proj-6}\\
\|\psi-{\cal Q}_{h,i}^{u}\psi\|_{\Gamma_{i,j}} & \le Ch^{\hat{k}+t}\|\psi\|_{\hat{k},\Gamma_{i,j}}, & 0\le\hat{k}\le k+1,
\label{eq:inter-proj-8}\\
\|\psi-{\cal Q}_{h,i}^{p}\psi\|_{\Gamma_{i,j}} & \le Ch^{\hat{r}+t}\|\psi\|_{\hat{r},\Gamma_{i,j}}, & 0\le\hat{r}\le r+1,
\label{eq:inter-proj-9}\\
\|\tau-\hat{\Pi}_{i}^{\s}\tau\|_{\Omega_{i}} & \le Ch^{\hat{k}}\|\tau\|_{\hat{k},\Oi}, & 0<\hat{k}\le k+1,\label{eq:inter-proj-14}\\
\|\zeta-\Pi_{i}^{z}\zeta\|_{\Omega_{i}} & \le Ch^{\hat{r}}\|\zeta\|_{\hat{r},\Oi}, & 0<\hat{r}\le r+1,\label{eq:inter-proj-15}\\
\|\text{div}(\tau-\hat{\Pi}_{i}^{\s}\tau)\|_{\Omega_{i}} & \le Ch^{\hat{l}}\|\text{div\,}\tau\|_{\hat{l},\Oi}, & 0\le\hat{l}\le l+1,\label{eq:inter-proj-5}\\
\|\text{div}(\zeta-\Pi_{i}^{z}\zeta)\|_{\Omega_{i}} & \le Ch^{\hat{s}}\|\text{div\,}\tau\|_{\hat{s},\Oi}, & 0\le\hat{s}\le s+1,\label{eq:inter-proj-11}
\end{align}
where the functions $\psi,\,v,\,\zeta,\,\tau,$ and $\xi$ are taken
from the domains of the operators acting on them. Bound (\ref{eq:inter-proj-1})
can be found in \cite{Scott-Zhang}, bounds (\ref{eq:inter-proj-3})$\--$(\ref{eq:inter-proj-9}) and (\ref{eq:inter-proj-5})$\--$(\ref{eq:inter-proj-11}) are standard $L^2$-projection approximation bounds \cite{ciarlet2002finite},
and bounds (\ref{eq:inter-proj-14})$\--$(\ref{eq:inter-proj-15})
can be found in \cite{brezzi1991mixed,roberts1991mixed,eldar_elastdd}.

We will also use the trace inequalities
\begin{align}
& \|\psi\|_{t,\Gamma_{i,j}} \le C\|\psi\|_{t+\frac{1}{2},\Oi}, \quad t>0,\label{eq:trace-1}\\
& \left\langle \psi,\tau n\right\rangle _{\dOi} \le C\|\psi\|_{\frac{1}{2},\dOi}\|\tau\|_{H(\text{div};\Oi)},
\quad 
\left\langle \psi,\zeta\cdot n\right\rangle _{\dOi} \le C\|\psi\|_{\frac{1}{2},\dOi}\|\zeta\|_{H(\text{div};\Oi)},
\label{eq:trace-2}
\end{align}
which can be found in \cite{grisvard2011elliptic} and \cite{brezzi1991mixed,roberts1991mixed},
respectively. 

Finally, define the projection operators $\hat{\Pi}^{\s},\,\Pi^{z},\,{\cal P}_{h}^{p},\,{\cal P}_{h}^{u},\,{\cal \,R}_{h},{\cal Q}^u_h$, and ${\cal Q}^p_h$
on the respective spaces defined in the global domain $\Omg$ to be the
piece-wise application of $\hat{\Pi}_{i}^{\s},\,\Pi_{i}^{z},\,{\cal P}_{h,i}^{p},\,{\cal P}_{h,i}^{p},\,{\cal R}_{h,i},\,{\cal Q}^u_{h,i}$, and ${\cal Q}^p_{h,i}$,
respectively, on subdomains $\Omega_{i}$ for $i=1,\dots,N$. 

\subsection{Discrete inf-sup conditions \label{subsec:Existence-of-inf-sup}}

In this subsection we give inf-sup stability bounds for the mortar space $\LH$ and the weakly continuous stress $\X_{h,0}$ and velocity $Z_{h,0}$ spaces under a coarseness 
condition on the mortar space $\LH$.

\begin{assump} 
	The mortar space $\Lambda_{H}$ is chosen so that there exists a positive
	constant $C$ independent of $H$ and $h$ such that
	\begin{equation}
	\|\mu^\star\|_{\Gamma_{i,j}}\le C\left(\|\mathcal{Q}_{h,i}^\star\mu^\star\|_{\Gamma_{i,j}} + \|\mathcal{Q}_{h,j}^\star\mu^\star\|_{\Gamma_{i,j}}\right),\quad\forall\mu^\star\in\Lambda_{H}^\star,\quad1\le i<j\le n, \ \ \star \in \{p,u\}.\label{eq:mortar_assumption}
	\end{equation}	
\end{assump}

\begin{remark}\label{rem:mortar_assum} 
Assumption \eqref{eq:mortar_assumption}, which was first introduced in \cite{APWY},
implies that the mortar space $\Lambda_{H}$ cannot be too rich compared to the normal traces of the subdomain stress/velocity spaces. In practice, this condition can be satisfied by taking a coarser mortar mesh, see \cite{APWY,arbogast2000mixed,pencheva2003balancing}.	
\end{remark}

\begin{lem}[Pressure mortar inf-sup condition]\label{lem:mortar-press}
	Under assumption \eqref{eq:mortar_assumption}, there exists a
	constant $\beta_{D}>0$, independent of $h$ and $H$ such that for
	any $\mu^{p}\in\Lambda_{H}^{p}$, 
	\begin{equation}
	\|\mu^{p}\|_{\G}\le\beta_{D}\sup_{0\ne\zeta\in Z_{h}}\frac{\sum_{i=1}^{N}\gnp[\zeta\cdot n_{i}]{\mu^{p}}_{\G_{i}}}{\|\zeta\|_{Z_{h}}}.\label{eq:mortar-p-inf-sup}
	\end{equation}
\end{lem}

\begin{proof}
	We start with any $\mu^{p}\in\Lambda_{H}^{p}$ and extend it by zero
	on $\dO$. Let $\phi_{i}$ be the solution to the following auxiliary
	problem 
	\begin{align}
	& \dvr\nabla\phi_{i}=\overline{\Qphi\mu^{p}}, & \text{in }\Oi,\label{eq:m-inf-sup-1}\\
	& \nabla\phi_{i}\cdot n_{i}=\Qphi\mu^{p}, & \text{on }\dOi,\label{eq:m-inf-sup-2}
	\end{align}
	where $\overline{\Qphi\mu^{p}}$ denotes the mean value of $\Qphi\mu^{p}$
	on $\dOi$. Let $\psi_{i}=\nabla\phi_{i}$. The elliptic
	problem \eqref{eq:m-inf-sup-1}--\eqref{eq:m-inf-sup-2} is well-posed and its solution satisfies the elliptic regularity bound \cite{grisvard2011elliptic}
\begin{equation}
	\|\psi_{i}\|_{1/2,\Oi}+\|\dvr\psi\|_{\Oi}\le C\|\Qphi\mu^{p}\|_{\dOi}.\label{eq:m-inf-sup-3}
\end{equation}
Take $\zeta_{h,i}=\Pi_i^{z}\psi_{i}\in Z_{h,i}$. Using \eqref{eq:darcy-project-2},
\eqref{eq:m-inf-sup-2}, and \eqref{eq:motor-project-2}, we obtain
\begin{align}
  \gnp[\zeta_{h,i}\cdot n_{i}]{\mu^{p}}_{\G_{i}} & = \gnp[\Pi^{z}\psi_{i}\cdot n_{i}]{\mu^{p}}_{\G_{i}}
  = \gnp[\Pi^{z}\psi_{i}\cdot n_{i}]{\Qphi\mu^{p}}_{\G_{i}}\nonumber \\
  & 
= \gnp[\psi_{i}\cdot n_{i}]{\Qphi\mu^{p}}_{\G_{i}}
  =\gnp[\Qphi\mu^{p}]{\Qphi\mu^{p}}_{\G_{i}}\ge C \|\mu^p\|^2_{\Gamma_{i}},\label{eq:m-inf-sup-4}
\end{align}
where we have used the mortar coarseness assumption \eqref{eq:mortar_assumption}. 
Next, we note that 
\begin{equation}
\|\zeta_{h,i}\|_{\dvr,\Omg_i}\le C\|\mu^p\|_{\Gamma_{i}},\label{eq:m-inf-sup-5}
\end{equation}
which follows from the stability of $\Pi_i^{z}$
(\ref{eq:darcy-project-3}) with $\epsilon=1/2$, (\ref{eq:m-inf-sup-3}),
and the stability of $\Qphi$. 
	
Finally, combining (\ref{eq:m-inf-sup-4}) with (\ref{eq:m-inf-sup-5})
and defining $\zeta:=\zeta_{h,i}$ on $\Oi$ completes the proof.
\end{proof}

\begin{lem}[Displacement mortar inf-sup condition]
Under assumption \eqref{eq:mortar_assumption}, there exists a
	constant $\beta_{E}>0$, independent of $h$ and $H$ such that for
	any $\mu^{u}\in\Lambda_{H}^{u}$, the following bound holds 
	\begin{equation}
	\|\mu^{u}\|_{\G}\le\beta_{E}\sup_{0\ne\tau\in\X_{h}}\frac{\sum_{i=1}^{N}\gnp[\tau n_{i}]{\mu^{u}}_{\G_{i}}}{\|\tau\|_{\X_{h}}}.\label{eq:mortar-u-inf-sup}
	\end{equation}
\end{lem}

\begin{proof}
The proof follows similar arguments as in the proof of Lemma~\ref{lem:mortar-press}.
\end{proof}
\begin{lem}
  \label{lem:operator_elast}
  Under assumption \eqref{eq:mortar_assumption},
	there exists a linear operator $\Pi_{0}^{\s}:H^{\frac{1}{2}+\epsilon}(\Omega,\mathbb{M})\cap\X\to\X_{h,0}$
	for any $\epsilon>0$, such that for any $\tau\in H^{\frac{1}{2}+\epsilon}(\Omega,\mathbb{M})\cap\X$,
	\begin{align}
& \left({\rm div} (\Pi_{0}^{\s}\tau-\tau),\,v\right)_{\Omega_{i}}=0, \quad 1 \le i \le N, & \forall v\in V_{h,i},\label{eq:pi-weak-1}\\
	& \left(\Pi_{0}^{\s}\tau-\tau,\,\xi\right)=0, & \forall\xi\in\mathbb{Q}_{h},\label{eq:pi-weak-2}\\
& \|\Pi_{0}^{\s}\tau\|\le C\left(\|\tau\|_{\frac{1}{2}+\epsilon}+\|\text{\emph{div\,}}\tau\|\right),\label{eq:pi-weak-3}\\
& \|\Pi_{0}^{\s}\tau-\tau\|\le C\left(\sum_{i=1}^{N} h^{\tilde{t}}\|\tau\|_{\tilde{t},\Omega_i} + h^{\tilde{k}}H^{\frac{1}{2}}\|\tau\|_{\tilde{k}+\frac{1}{2}}\right), & 0 < \tilde{t}\le k+1,\,0<\tilde{k}\le k+1,\label{eq:pi-weak-5}\\
& \|{\rm div}(\Pi_{0}^{\s}\tau-\tau)\|_{\Omega_{i}}\le
Ch^{\tilde{l}}\|{\rm div}\,\tau\|_{\tilde{l},\Oi}, \quad 1 \le i \le N, & 0\le\tilde{l}\le l+1.         
\label{eq:pi-weak-div}
\end{align}
\end{lem}

\begin{proof}
The proof is based on constructing $ \Pi_{0}^{\s}\tau|_\dOi = \hat{\Pi}^\s_i (\tau + \delta \tau_i)$, where the correction $\delta \tau_i$ is designed to give weak continuity of the normal components.
The proof of \eqref{eq:pi-weak-1}--\eqref{eq:pi-weak-5}
is given in \cite[Lemma 4.6]{eldar_elastdd}. Bound \eqref{eq:pi-weak-div} follows from \eqref{eq:pi-weak-1} and the approximation properties of the $L^2$-projection \cite{ciarlet2002finite}.
\end{proof}
\begin{lem}\label{lem:operator_darcy}
  Under assumption \eqref{eq:mortar_assumption}, there exists a linear operator $\Pi_{0}^{z}: H^{\frac{1}{2}+\epsilon}(\Omega,\R^d)\cap Z\to Z_{h,0}$
such that for any $\zeta\in H^{\frac{1}{2}+\epsilon}(\Omega,\R^d) \cap Z$,
\begin{align}
  & \left(\text{\emph{div}}\left(\Pi_{0}^{z}\zeta-\zeta\right),\,w\right)_{\Omega_{i}}=0,
  \quad 1 \le i \le N, & \forall w\in W_{h,i},\label{eq:pi-weak-d-1}\\
& \|\Pi_{0}^{z}\zeta\|_{Z_{h}}\le C\left(\|\zeta\|_{\frac{1}{2}+\epsilon} + \|\dvr\zeta\|\right),\label{eq:pi-weak-d-2}\\
& \|\Pi_{0}^{z}\zeta-\zeta\| \le C \left( \sum_{i=1}^{N} h^{\tilde{t}}\|\zeta\|_{\tilde{t},\Omega_{i}} + h^{\tilde{r}}H^{\frac{1}{2}}\|\zeta\|_{\tilde{r}+\frac{1}{2}}\right), &
  0 < \tilde{t}\le r+1,\, 0<\tilde{r}\le r+1,\label{eq:pi-weak-d-4}\\
&  \|{\rm div}(\Pi_{0}^{z}\zeta-\zeta)\|_{\Omega_{i}} \le
Ch^{\tilde{s}}\|{\rm div}\,\zeta\|_{\tilde{s},\Oi}, \quad 1 \le i \le N, & 0\le\tilde{s}\le s+1.         
\label{eq:pi-weak-d-div} 
\end{align}
\end{lem}

\begin{proof}
The proof follows from the arguments given in \cite[Section 3]{arbogast2000mixed} and \cite[Section 3]{APWY}. 
\end{proof}

Lemmas \ref{lem:operator_elast} and \ref{lem:operator_darcy} can be used to show
inf-sup stability for to the weakly continuous stress and velocity spaces.

\begin{lem}\label{thm:inf-sup-weak}
Under assumption \eqref{eq:mortar_assumption}, there exist positive constants $C_{E}$ and $C_{D}$ independent of the discretization parameters $h$ and $H$ such that 
\begin{align}
  \forall \, v\in V_{h}, \, \xi\in\mathbb{Q}_{h}, \quad \|v\|+\|\xi\| & \le C_{E}\sup_{0\ne\tau\in\mathbb{X}_{h,0}}\frac{\inp[\dvr_h \tau]{v}
    + \inp[\tau]{\xi}}{\|\tau\|_{\X_{h}}},\label{eq:inf-sup-elast-weak} \\
\forall \, w \in W_h, \quad
  \|w\| & \le C_{D}\sup_{0\ne\zeta\in Z_{h,0}}\frac{\left(\dvr_h{\zeta},\,w\right)}
        {\|\zeta\|_{Z_{h}}}.\label{eq:inf-sup-darcy-weak}
\end{align}

\end{lem}

\begin{proof}
The proof follows the argument in Fortin's Lemma \cite[Proposition~2.8]{brezzi1991mixed}. We present the proof of \eqref{eq:inf-sup-elast-weak}. The proof of \eqref{eq:inf-sup-darcy-weak} is similar. We first note that the following continuous inf-sup condition holds \cite[Section~2.4.3]{Gatica}:
\begin{align}
\forall v \in V, \, \xi \in \W, \quad
\|v\| + \|\xi\| & \le \tilde C_E \sup_{0\ne\tau\in H^1(\Omega, \mathbb{M})} \frac{\inp[\dvr{\tau}]{v} + \inp[\tau]{\xi}}
      {\|\tau\|_{1}}. \label{eq:inf-sup-elast-cont}
\end{align}
Using Lemma~\ref{lem:operator_elast} and \eqref{eq:inf-sup-elast-cont}, we have, $\forall v \in V_h, \, \xi \in \W_h$,
\begin{align*}
& \sup_{0\ne\tau\in\mathbb{X}_{h,0}}\frac{\inp[\dvr_h \tau]{v}
    + \inp[\tau]{\xi}}{\|\tau\|_{\X_{h}}} \\
& \qquad \ge
\sup_{0\ne\tau\in H^1(\Omega, \mathbb{M})} \frac{\inp[\dvr_h \Pi_0^\sigma\tau]{v}
  + \inp[\Pi_0^\sigma\tau]{\xi}}{\|\Pi_0^\sigma\tau\|_{\X_{h}}}
= \sup_{0\ne\tau\in H^1(\Omega, \mathbb{M})} \frac{\inp[\dvr \tau]{v}
  + \inp[\tau]{\xi}}{\|\Pi_0^\sigma\tau\|_{\X_{h}}} \\
& \qquad \ge \frac{1}{C}
\sup_{0\ne\tau\in H^1(\Omega, \mathbb{M})} \frac{\inp[\dvr \tau]{v}
  + \inp[\tau]{\xi}}{\|\tau\|_1} \ge \frac{1}{C\tilde C_E}(\|v\| + \|\xi\|),
\end{align*}
which implies \eqref{eq:inf-sup-elast-weak} with $C_{E} = C\tilde C_E$.
\end{proof}

\section{Well-posedness of the semi-discrete multiscale mortar MFE method} \label{sec:Analysis-of-MMMFE}

In this section we present the well-posedness analysis of the method developed in Section~\ref{sec:method}. We show that the method has a unique solution and establish 
stability bounds.

\subsection{Existence and uniqueness of a solution}\label{sec:well-posed}

We next show the existence of a unique solution
to the system of equations (\ref{eq:monodd-mmmfe1})$\--$(\ref{eq:monodd-mmmfe7})
under the assumption (\ref{eq:mortar_assumption}). We follow closely
the proof for the well-posedness of the multipoint flux method for
the Biot system given in \cite{msfmfe-Biot}. We base our proof on
the theory for showing the existence of solution to a degenerate parabolic
system \cite{Showalter}. In particular, we use \cite[IV, Theorem 6.1(b)]{Showalter}
which is stated as follows.
\begin{thm}
	\label{thm:algebraic-pde}Let the linear, symmetric, and monotone
	operator $\mathcal{N}$ be given for the real vector space $E$ to
	its algebraic dual $E^{*}$, and let $E_{b}^{'}$ be the Hilbert space
	which is the dual of $E$ with the seminorm $|x|_{b}=\sqrt{\mathcal{N}x(x)}$
	for $x\in E.$ Let $\mathcal{M}\subset E\times E_{b}^{'}$ be a relation
	with the domain $D=\left\{ x\in E:\mathcal{M}(x)\ne\emptyset\right\} .$
	Assume that ${\cal M}$ is monotone and $Range({\cal N+M})=E_{b}^{'}$.
	Then for each $x_{0}\in D$ and for each ${\cal F}\in W^{1,1}\left(0,T;E_{b}^{'}\right)$,
	there is a solution $x$ of 
	\begin{eqnarray*}
	\frac{d}{dt}\mathcal{N}x(t)+{\cal M}x(t)\ni{\cal F}(t), &  & \quad a.e.\,\,\,0<t<T,
	\end{eqnarray*}
	with 
	\[
	{\cal N}x\in W^{1,\infty}\left(0,T;E_{b}^{'}\right),x(t)\in D,\,\text{for all }0\le t\le T,\text{ and }{\cal N}x(0)={\cal N}x_{0}.
	\]
\end{thm}

Using the above theorem, we now prove that the semi-discrete system
(\ref{eq:monodd-mmmfe1})--(\ref{eq:monodd-mmmfe7}) is well-posed. We start by reformulating it 
to fit the setting of Theorem \ref{thm:algebraic-pde}. For this purpose, we define operators 
\begin{align*}
& \left(A_{\s\s}\s_{h},\,\tau\right)=\left(A\s_{h},\,\tau\right),\quad\left(A_{\s p}\s_{h},\,w\right)=\alpha\left(A\s_{h},\,wI\right),\quad\left(A_{\s u}\s_{h},\,v\right)=\left(\text{div}_h \, \s_{h},\,v\right),\\
& \left(A_{\s\gamma}\s_{h},\,\xi\right)=\left(\s_{h},\,\xi\right),\quad\left(A_{\s\lambda}\s_{h},\,\mu^{u}\right)=\sum_{i=1}^{N}\left\langle \s_{h}n_{i},\,\mu^{u}\right\rangle _{\G_{i}},\quad\left(A_{zz}z_{h},\,\zeta\right)=\left(K^{-1}z_{h},\,\zeta\right),\\
& \left(A_{zp}z_{h},\,w\right)=-\left(\text{div}_h \, z_{h},\,w\right),\quad\left(A_{z\lambda}z_{h},\,\mu^{p}\right)=\sum_{i=1}^{N}\left\langle z_{h}\cdot n_{i},\,\mu^{p}\right\rangle _{\G_{i}},\\
& \left(A_{pp}p_{h},\,w\right)=c_{0}\left(p_{h},\,w\right)+\alpha^{2}\left(Ap_{h}I,\,wI\right).
\end{align*}

In order to fit in the structure of Theorem~\ref{thm:algebraic-pde}, we consider a modified problem where \eqref{eq:monodd-mmmfe1} is differentiated in time. Introducing the new variables $\dot{u}_h$, $\dot{\gamma}_h$, and $\dot{\lambda}^u_H$ representing $\dt u_h$, $\dt \gamma_h$, and $\dt \lambda^u_H$, respectively, we differentiate (\ref{eq:monodd-mmmfe1}) in time to get
\begin{align}
& \inp[\dt A\left(\sigma_{h} + \a p_{h}I\right)]{\tau} + \inp[\dot{u}_h]{\dvr_h{\tau}} + \inp[\dot{\gamma}_h]{\tau} - \sum_{i=1}^{N}\inp[\dot{\lambda}^u_H]{\tau n_{i}}_{\G_{i}} = 0, & \forall\tau\in\X_{h}.\label{eq:fe-mono-weak-diff}
\end{align}
Using the above definitions of operators we can write the differentiated system \eqref{eq:fe-mono-weak-diff}, \eqref{eq:monodd-mmmfe2}--\eqref{eq:monodd-mmmfe7}
as 
\begin{eqnarray}
\frac{d}{dt}\mathcal{N}\dot{x}(t)+{\cal M}\dot{x}(t) = {\cal F}(t), & 0<t<T,\label{eq:eq:mod_matrix_form}
\end{eqnarray}
where 
\begin{gather*}
\dot{x}=\left(\begin{array}{c}
\s_{h}\\
\dot{u}_h\\
\dot{\gamma}_h\\
z_{h}\\
p_{h}\\
\dot{\lambda}^u_H\\
\lHp
\end{array}\right), \ N=\left(\begin{array}{ccccccc}
	A_{\s\s} & 0 & 0 & 0 & A_{\s p}^{T} & 0 & 0\\
	0 & 0 & 0 & 0 & 0 & 0 & 0\\
	0 & 0 & 0 & 0 & 0 & 0 & 0\\
	0 & 0 & 0 & 0 & 0 & 0 & 0\\
	A_{\s p} & 0 & 0 & 0 & A_{pp} & 0 & 0\\
	0 & 0 & 0 & 0 & 0 & 0 & 0\\
	0 & 0 & 0 & 0 & 0 & 0 & 0
	\end{array}\right), \\
	M=\left(\begin{array}{ccccccc}
	0 & A_{\s u}^{T} & A_{\s\gamma}^{T} & 0 & 0 & -A_{\s\lambda}^{T} & 0\\
	-A_{\s u} & 0 & 0 & 0 & 0 & 0 & 0\\
	-A_{\s\gamma} & 0 & 0 & 0 & 0 & 0 & 0\\
	0 & 0 & 0 & A_{zz} & A_{zp}^{T} & 0 & A_{z\lambda}^{T}\\
	0 & 0 & 0 & -A_{zp} & 0 & 0 & 0\\
	A_{\s\lambda} & 0 & 0 & 0 & 0 & 0 & 0\\
	0 & 0 & 0 & A_{z\lambda} & 0 & 0 & 0
  \end{array}\right), \
  \mathcal{F}=\left(\begin{array}{c}
	0\\
	-f\\
	0\\
	0\\
	g\\
	0\\
	0
	\end{array}\right).
\end{gather*}
The space $E$ is $\X_{h}\times V_{h}\times\mathbb{Q}_{h}\times Z_{h}\times W_{h}\times\Lambda_{H}^u\times\Lambda_H^p$. The dual space $E_{b}^{'}$ is given by $L^{2}(\Omega,\mathbb{M})\times0\times0\times0\times L^{2}(\Omega)\times0\times0$
	and the condition ${\cal F}\in W^{1,1}\left(0,T;E_{b}^{'}\right)$
	implies that non-zero source terms can appear only in equations with
	time derivatives. This means we have to take $f=0$ in our case. We
	can fix this issue by considering an auxiliary problem that, for each
	$t\in(0,T]$, solves the system 
	\begin{equation}
	\left(\begin{array}{cccc}
	A_{\s\s} & A_{\s u}^{T} & A_{\s\gamma}^{T} & -A_{\s\lambda}^{T}\\
	-A_{\s u} & 0 & 0 & 0\\
	-A_{\s\gamma} & 0 & 0 & 0\\
	A_{\s\lambda} & 0 & 0 & 0
	\end{array}\right)\left(\begin{array}{c}
	\s_{h}^{f}\\
	\dt u_{h}^{f}\\
	\dt\gamma_{h}^{f}\\
	\dt\lambda_{H}^{u,f}
	\end{array}\right)=\left(\begin{array}{c}
	0\\
	-f\\
	0\\
	0
	\end{array}\right).\label{eq:aux-elast-wellpos}
	\end{equation}
	Such an auxiliary system (\ref{eq:aux-elast-wellpos}) is well-posed
	and the proof can be found in \cite{eldar_elastdd}. Now we can subtract
	the solution to (\ref{eq:aux-elast-wellpos}) from the original system
	of equations (\ref{eq:monodd-mmmfe1})$\--$(\ref{eq:monodd-mmmfe7})
	to obtain the modified right hand side ${\cal F}=\left(A_{\s\s}\left(\s_{h}^{f}-\dt\s_{h}^{f}\right),0,0,0,q-A_{\s p}\dt\s_{h}^{f},0,0\right)^{T}$. Thus, it is enough to analyze \eqref{eq:eq:mod_matrix_form} with $f=0$.

In order to apply Theorem \ref{thm:algebraic-pde} for system \eqref{eq:eq:mod_matrix_form}, we need to prove the range condition $Range({\cal N+M}) = E_{b}^{'}$ and construct compatible initial data $\dot{x}_{0}\in D$, i.e., ${\cal M}\dot{x}_{0}\in E_{b}^{'}$. This is done in the following two lemmas.

\begin{lem}\label{lem:range} If assumption \eqref{eq:mortar_assumption} holds, for the system \eqref{eq:eq:mod_matrix_form} it holds that 
$Range({\cal N+M})=E_{b}^{'}$.
\end{lem}  

\begin{proof}
The statement of the lemma can be established by proving that the following homogeneous system has only the zero solution: $\ensuremath{(\hat{\sigma}_{h},\hat{u}_{h},\hat{\gamma}_{h},\hat{z}_{h},\hat{p}_{h},\hat{\lambda}_{H})}\in\X_{h}\times V_{h}\times\mathbb{Q}_{h}\times Z_{h}\times W_{h}\times\Lambda_{H}$ such that
\begin{align}
	& \inp[A\left(\hat{\sigma}_{h}+\a\hat{p}_{h}I\right)]{\tau} + \inp[\hat{u}_{h}]{\dvr_h{\tau}} + \inp[\hat{\gamma}_{h}]{\tau} -\sum_{i=1}^{N}\gnp[\hat{\lambda}_{H}^{u}]{\t\,n_{i}}_{\Gamma_{i}}=0, & \forall\tau\in\X_{h},\label{eq:well-pose-aux-1}\\
	& \inp[\dvr_h{\hat{\sigma}_{h}}]{v} = 0, & \forall v\in V_{h}, \\
	& \inp[\hat{\sigma}_{h}]{\xi}=0, & \forall\xi\in\mathbb{Q}_{h}, \\
	& \inp[\K\hat{z}_{h}]{\zeta} - \inp[\hat{p}_{h}]{\dvr_h{\zeta}} + \sum_{i=1}^{N}\gnp[\hat{\lambda}_{H}^{p}]{\zeta\cdot n_{i}}_{\Gamma_{i}}=0, & \forall\zeta\in Z_{h},\label{eq:well-pose-aux-4}\\
	& c_{0}\inp[\dt{\hat{p}_{h}}]{w} + \a\inp[ A\left(\hat{\sigma}_{h} + \a\hat{p}_{h}I\right)]{wI} + \inp[\dvr_h{\hat{z}_{h}}]{w} = 0, & \forall w\in W_{h}, \\
	& \sum_{i=1}^{N}\gnp[\hat{\sigma}_{h}n_{i}]{\mu^{u}}_{\G_{i}}=0, & \forall\mu^{u}\in\Lambda_{H}^{u}, \\
	& \sum_{i=1}^{N}\gnp[\hat{z}_{h}\cdot n_{i}]{\mu^{p}}_{\G_{i}}=0, & \forall\mu^{p}\in\Lambda_{H}^{p}.
\end{align}
Taking test functions $(\tau,v,\xi,\zeta,w,\mu^{u},\mu^{p})=(\hat{\sigma}_{h},\hat{u}_{h},\hat{\gamma}_{h},\hat{z}_{h},\hat{p}_{h},\hat{\lambda}_{H}^{u},\hat{\lambda}_{H}^{p})$
in the above system and adding the equations together gives $\|A^{\frac{1}{2}}\left(\hat{\sigma}_{h}+\a\hat{p}_{h}I\right)\|^{2}+c_{0}\|\hat{p}_{h}\|^{2}+\|K^{-\frac{1}{2}}\hat{z}_{h}\|^{2}=0$.
The coercivity of $A$, (\ref{eq:coercivity-elast-1}), and $K$, (\ref{eq:coercivity-flow-1}),
give $\hat{\sigma}_{h} = 0$, $\hat{p}_{h}=0$, and $\hat{z}_{h}=0$.
The inf-sup condition with respect to the weakly continuous space $\X_{h,0}$ \eqref{eq:inf-sup-elast-weak} along with \eqref{eq:well-pose-aux-1} implies $\hat{u}_{h}=0$ and $\hat{\gamma}_{h}=0$. Finally, \eqref{eq:mortar-p-inf-sup}
combined with (\ref{eq:well-pose-aux-1}) implies $\hat{\lambda}_{H}^{u}=0$,
and (\ref{eq:mortar-u-inf-sup}) combined with (\ref{eq:well-pose-aux-4})
implies $\hat{\lambda}_{H}^{p}=0$.
\end{proof}

\begin{lem}\label{lem:IC}
Let the assumption \eqref{eq:mortar_assumption} hold. Given initial data $p_0 \in H^1(\Omega)$ with $K\nabla p_0 \in H(\dvr;\Omega)$, there exists initial data $\dot{x}_{0}$ for the system \eqref{eq:eq:mod_matrix_form} such that ${\cal M}\dot{x}_{0}\in E_{b}^{'}$.
\end{lem}

\begin{proof}
We first construct compatible initial data $\ensuremath{(\sigma_{0},u_{0},\gamma_{0},z_{0},p_{0})}$ to the continuous system (\ref{eq:cts1-1})--(\ref{eq:cts5-1}) from the initial data $p_{0}$ as follows:
\begin{enumerate}
\item Solve equations (\ref{eq:cts1-1})$\--$(\ref{eq:cts3-1}) using $p=p_{0}$
as given data to obtain $\sigma_{0},u_{0},\gamma_{0}$.
\item Set $z_0=-K\nabla p_0$ and use integration by parts to show that \eqref{eq:cts4-1} holds.
\end{enumerate}
Next, define $\tilde{x}_{0}=\ensuremath{(\sigma_{0},u_{0},\gamma_{0},z_{0},p_{0},\lambda_{0}^{u},\lambda_{0}^{p})}$, where $\lambda_{0}^{u}=u_{0}|_{\G}$ and $\lambda_{0}^{p}=p_{0}|_{\G}$.
Take the initial data $x_{0} = (\sigma_{h,0},u_{h,0},\gamma_{h,0},z_{h,0},p_{h,0},\lambda_{H,0}^{u},\lambda_{H,0}^{p})$ for the system (\ref{eq:monodd-mmmfe1})$\--$(\ref{eq:monodd-mmmfe7}) to be the elliptic projection of $\tilde{x}_{0}$:
\begin{equation}
  \left(\mathcal{N+M}\right)x_{0}=\left(\mathcal{N+M}\right)\tilde{x}_{0},
  \label{eq:well-pose-initial-data}
\end{equation}
The above problem has a unique solution, due to the argument in the proof of Lemma~\ref{lem:range}.
With the reduction of the problem to the case with $f=0$, we have $\left(\mathcal{N+M}\right)\tilde{x}_{0}\in E_{b}^{'}$ and, due to \eqref{eq:well-pose-initial-data},
${\cal M}x_{0}=\left(\mathcal{N+M}\right)\tilde{x}_{0}-{\cal N}x_{0}\in E_{b}^{'}$. For the differentiated system (\ref{eq:eq:mod_matrix_form}) we
take the initial data $\dot{x}_{0}$ to be $(\sigma_{h,0},0,0,z_{h,0},p_{h,0},0,\lambda_{H,0}^p)$,
which also satisfies ${\cal M}\dot{x}_{0}\in E_{b}^{'}$. We note that
the initial data $u_{h,0}$, $\gamma_{h,0}$, and $\lambda_{H,0}^{u}$ are
not required for solving (\ref{eq:eq:mod_matrix_form}), but are used to recover the solution to the original problem.
\end{proof}

We are now ready to establish that the system \eqref{eq:eq:mod_matrix_form} has a solution using Theorem~\ref{thm:algebraic-pde}.

\begin{lem}\label{lem:well-posed}
Let assumption \eqref{eq:mortar_assumption} hold. For each $(f,g)\in W^{1,\infty}\left(0,T;L^{2}(\Omega;{\mathbb R}^d)\right)\times W^{1,\infty}\left(0,T;L^{2}(\Omega)\right)$ and $p_0 \in H^1(\Omega)$ with $K\nabla p_0 \in H(\dvr;\Omega)$, the system \eqref{eq:eq:mod_matrix_form} has a solution such that $\s_{h}(0)=\s_{h,0}$ and $p_{h}(0)=p_{h,0}$, with the initial data constructed in Lemma~\ref{lem:IC}. In addition, $z_{h}(0)=z_{h,0}$ and $\lambda_H^p(0) = \lambda_{H,0}^p$.
\end{lem}
  
\begin{proof}
  We first note that the arguments in the proof of Lemma~\ref{lem:range} can be used to show that ${\cal N}$ and ${\cal M}$ are non-negative and therefore, due to linearity, monotone. Using Lemmas~\ref{lem:range} and \ref{lem:IC}, an application of Theorem~\ref{thm:algebraic-pde} implies the existence of a solution $\dot{x}=\ensuremath{(\sigma_{h},\dot{u}_{h},\dot{\gamma}_{h},z_{h},p_{h},\dot\lambda_H^u,\lambda_H^p)}$ to \eqref{eq:eq:mod_matrix_form} such that $\s_{h}(0)=\s_{h,0}$ and $p_{h}(0)=p_{h,0}$. Next, it
is easy to see that $z_{h}(0)=z_{h,0}$ by taking $t\to0$ in (\ref{eq:fe-mono-weak-4})
and using the fact that $z_{h,0}$ and $p_{h,0}$ satisfy (\ref{eq:fe-mono-weak-4}). Finally, taking $t\to0$ in \eqref{eq:monodd-mmmfe4}, using that $z_{h,0}$, $p_{h,0}$, and $\lambda_{H,0}^p$ satisfy \eqref{eq:monodd-mmmfe4}, and employing the inf-sup condition \eqref{eq:mortar-p-inf-sup}, we conclude that $\lambda_H^p(0) = \lambda_{H,0}^p$.
\end{proof}

Next, we prove the solvability of the original system \eqref{eq:monodd-mmmfe1}--\eqref{eq:monodd-mmmfe7}.

\begin{thm}\label{thm:well_posedness-mono-mortar}
  Let assumption \eqref{eq:mortar_assumption} hold. For each
  $(f,g)\in W^{1,\infty}\left(0,T;L^{2}(\Omega;{\mathbb R}^d)\right)\times W^{1,\infty}\left(0,T;L^{2}(\Omega)\right)$
    and $p_0 \in H^1(\Omega)$ with $K\nabla p_0 \in H(\dvr;\Omega)$, the system \eqref{eq:monodd-mmmfe1}--\eqref{eq:monodd-mmmfe7}
has a unique solution such that $\s_{h}(0)=\s_{h,0}$ and $p_{h}(0)=p_{h,0}$, where the initial data is constructed in Lemma~\ref{lem:IC}. In addition, $u_h(0) = u_{h,0}$, $\gamma_h(0) = \gamma_{h,0}$,
$z_{h}(0)=z_{h,0}$, $\lambda_H^u(0) = \lambda_{H,0}^u$, and $\lambda_H^p(0) = \lambda_{H,0}^p$.
\end{thm}

\begin{proof}
Let $(\sigma_{h},\dot u_{h},\dot\gamma_{h},z_{h},p_{h},\dot\lambda_{H}^{u},\lambda_{H}^{p})$ be a solution to the differentiated system \eqref{eq:fe-mono-weak-diff}, \eqref{eq:monodd-mmmfe2}--\eqref{eq:monodd-mmmfe7} obtained in Lemma~\ref{lem:well-posed}.
For each $t\in[0,T]$, define 
\begin{equation}\label{orig-var}
u_{h}(t)=u_{h,0}+\int_{0}^{t}\dot{u}_{h}(s)ds, \quad
\gamma_{h}(t)=\gamma_{h,0}+\int_{0}^{t}\dot{\gamma}_{h}(s)ds, \quad
\lHu(t)=\lambda_{H,0}^{u}+\int_{0}^{t}\dot{\lambda}_{H}^{u}(s)ds.
\end{equation}
Consider $x = (\sigma_{h},u_{h},\gamma_{h},z_{h},p_{h},\lambda_{H}^{u},\lambda_{H}^{p})$.
Since equations \eqref{eq:monodd-mmmfe2}--\eqref{eq:monodd-mmmfe7} do not involve $u_h$, $\gamma_h$ or $\lambda_{H}^{u}$, they still hold for $x$. It remains to show that \eqref{eq:monodd-mmmfe1} holds. This follows by integrating (\ref{eq:fe-mono-weak-diff}) with respect to time from $0$
to any $t\in(0,T]$ and using \eqref{orig-var} and the fact that $\s_{h,0},\,u_{h,0},\,\gamma_{h,0}$, and $\lambda_{H,0}^{u}$ are constructed to satisfy (\ref{eq:monodd-mmmfe1}). This completes the proof that the system \eqref{eq:monodd-mmmfe1}--\eqref{eq:monodd-mmmfe7} has a solution. Uniqueness of the solution follows from the stability bound presented in the next section. Finally, Lemma~\ref{lem:well-posed} gives that $\s_{h}(0)=\s_{h,0}$, $p_{h}(0)=p_{h,0}$, $z_{h}(0)=z_{h,0}$, and $\lambda_H^p(0) = \lambda_{H,0}^p$, while \eqref{orig-var} gives that $u_h(0) = u_{h,0}$, $\gamma_h(0) = \gamma_{h,0}$, and $\lambda_H^u(0) = \lambda_{H,0}^u$.
\end{proof}

\subsection{Stability analysis}

In this subsection we give a stability bound for the system
\eqref{eq:monodd-mmmfe1}--\eqref{eq:monodd-mmmfe7}. 

\begin{thm}
\label{thm:stability-mono-mortar} Under the assumption \eqref{eq:mortar_assumption}, there exists a constant $C>0$, independent of $c_{0}$ and the discretization parameters $h$ and $H$, such that for the solution of \eqref{eq:monodd-mmmfe1}--\eqref{eq:monodd-mmmfe7}, 
\begin{align}
& \|\sigma_{h}\|_{L^{\infty}\left(0,T;\X_h\right)}+\|u_{h}\|_{L^{\infty}\left(0,T;L^{2}\left(\Omega\right)\right)}+\|\gamma_{h}\|_{L^{\infty}\left(0,T;L^{2}\left(\Omega\right)\right)}+\|z_{h}\|_{L^{\infty}\left(0,T;L^{2}\left(\Omega\right)\right)}+\|p_{h}\|_{L^{\infty}\left(0,T;L^{2}\left(\Omega\right)\right)} \nonumber \\[1ex]
& \qquad +\|\lHu\|_{L^{\infty}\left(0,T;L^{2}\left(\G\right)\right)}+\|\lHp\|_{L^{\infty}\left(0,T;L^{2}\left(\G\right)\right)}+\|\sigma_{h}\|_{L^{2}\left(0,T;\X_{h}\right)}+\|u_{h}\|_{L^{2}\left(0,T;L^{2}\left(\Omega\right)\right)}+\|\gamma_{h}\|_{L^{2}\left(0,T;L^{2}\left(\Omega\right)\right)} \nonumber \\[1ex]
& \qquad +\|z_{h}\|_{L^{2}\left(0,T;Z_h\right)}+\|p_{h}\|_{L^{2}\left(0,T;L^{2}\left(\Omega\right)\right)}+\|\lHu\|_{L^{2}\left(0,T;L^{2}\left(\G\right)\right)}+\|\lHp\|_{L^{2}\left(0,T;L^{2}\left(\G\right)\right)} \nonumber \\[1ex]
&  \quad \le C\big(\|f\|_{H^{1}\left(0,T;L^{2}\left(\Omega\right)\right)}+\|g\|_{H^{1}\left(0,T;L^{2}\left(\Omega\right)\right)}+\|p_{0}\|_{H^{1}(\Omega)}+\|\nabla Kp_{0}\|_{H\left(\text{\emph{div}};\Omega\right)}\big). \label{eq:stability-mono-mortar}
\end{align}
\end{thm}

\begin{proof}
It is convenient to use the weakly continuous normal stress and velocity formulation 
\eqref{eq:fe-mono-weak-1}--\eqref{eq:fe-mono-weak-5}. We differentiate \eqref{eq:fe-mono-weak-1} in time, combine it with \eqref{eq:fe-mono-weak-2}--\eqref{eq:fe-mono-weak-5}, and take test functions $(\tau,v,\xi,\zeta,w) = (\sigma_{h},\dt u_{h},\dt\gamma_{h},z_{h},p_{h})$ to get
\begin{align*}
& \inp[\dt A\left(\sigma_{h}+\a p_{h}I\right)]{\sigma_{h}+\a p_{h}I}+c_{0}\inp[\dt p_{h}]{p_{h}}+\inp[\K z_{h}]{z_{h}}=\left(f,\dt u_{h}\right)+\left(q,p_{h}\right).
\end{align*}
The above equation can be rewritten as
\begin{equation}
\frac{1}{2}\dt\left(\|A^{\frac{1}{2}}\left(\sigma_{h}+\a p_{h}I\right)\|^{2}+c_{0}\|p_{h}\|^{2}\right)+\|K^{-\frac{1}{2}}z_{h}\|^{2}=\dt\left(f,u_{h}\right)-\left(\dt f,u_{h}\right)+\left(g,p_{h}\right).\label{eq:stab-bnd-1}
\end{equation}
For any $t\in(0,T]$, we integrate equation (\ref{eq:stab-bnd-1}) with respect to time from $0$ to $t$ to get
\begin{align*}
& \frac{1}{2}\left(\|A^{\frac{1}{2}}\left(\sigma_{h}+\a p_{h}I\right)(t)\|^{2}+c_{0}\|p_{h}(t)\|^{2}\right)+\int_{0}^{t}\|K^{-\frac{1}{2}}z_{h}\|^{2}ds\\
& =\frac{1}{2}\left(\|A^{\frac{1}{2}}\left(\sigma_{h}+\a p_{h}I\right)(0)\|^{2}+c_{0}\|p_{h}(0)\|^{2}\right)+\int_{0}^{t}\left(\left(g,p_{h}\right)-\left(\dt f,u_{h}\right)\right)ds+\left(f,u_{h}\right)(t)-\left(f,u_{h}\right)(0).
\end{align*}
Applying the Cauchy-Schwartz and Young's inequalities, we get, for any $\epsilon > 0$,
\begin{align}
& \|A^{\frac{1}{2}}\left(\sigma_{h}+\a p_{h}I\right)(t)\|^{2}+c_{0}\|p_{h}(t)\|^{2}+2\int_{0}^{t}\|K^{-\frac{1}{2}}z_{h}\|^{2}ds\nonumber \\
	& \quad \le\|A^{\frac{1}{2}}\left(\sigma_{h}+\a p_{h}I\right)(0)\|^{2}+c_{0}\|p_{h}(0)\|^{2}+\epsilon\left(\int_{0}^{t}\left(\|p_{h}\|^{2}+\|u_{h}\|^{2}\right)ds+\|u_{h}(t)\|^{2}\right)\label{eq:stab-bnd-2}\\
	& \qquad +\frac{1}{\epsilon}\left(\int_{0}^{t}\left(\|g\|^{2}+\|\dt f\|^{2}\right)ds+\|f(t)\|^{2}\right)+\|f(0)\|^{2}+\|u_{h}(0)\|^{2}.\nonumber 
	\end{align}
A bound for $\|u_{h}\|$ and $\|\gamma_{h}\|$ follows from the inf-sup
condition (\ref{eq:inf-sup-elast-weak}) and \eqref{eq:fe-mono-weak-1}:
\begin{align}
& \|u_{h}\|+\|\gamma_{h}\| \le C_{E}\sup_{0\ne\tau\in\mathbb{X}_{h,0}}\frac{\inp[u_{h}]{\dvr_h{\tau}} + \inp[\gamma_{h}]{\tau}}{\|\tau\|_{\X_{h}}}
= C_{E}\sup_{0\ne\tau\in\mathbb{X}_{h,0}}\frac{\inp[A(\sigma_{h} + \a p_{h}I)]{\tau}}{\|\tau\|_{\X_{h}}}\le C\|\sigma_{h}+\a p_{h}I\|.\label{eq:stab-bnd-3}
\end{align}
Next, choose test functions $(\tau,v,\xi)=(\sigma_{h},u_{h},\gamma_{h})$
in \eqref{eq:fe-mono-weak-1}--\eqref{eq:fe-mono-weak-3}, 
combine the equations and use the Cauchy-Schwartz and Young's inequalities to get 
\[
\|\sigma_{h}\|^{2}\le C\left(\|p_{h}^{2}\|+\epsilon\|u_{h}\|^{2} + \frac{1}{\epsilon}\|f\|^{2}\right).
\]
Combining the above inequality with \eqref{eq:stab-bnd-3} and taking $\epsilon$ small enough
yields 
\begin{equation}
\int_{0}^{t}\left(\|u_{h}\|^{2}+\|\gamma_{h}\|^{2}\right)ds\le C\int_{0}^{t}\left(\|p_{h}^{2}\| + \|f\|^{2}\right)ds.
\label{eq:stab-bnd-4-1}
\end{equation}
Bound for $\|p_{h}\|$ can be obtained from the inf-sup condition
(\ref{eq:inf-sup-darcy-weak}) and equation \eqref{eq:fe-mono-weak-4} as follows:
\begin{align}
||p_{h}|| & \le C_{D}\sup_{0\ne\zeta_{h}\in Z_{h,0}}\frac{\sum_{i=1}^{N}\left(\dvr\zeta_{h},p_{h}\right)_{\Omega_{i}}}{||\zeta_{h}||_{Z_{h}}} = C_{D}\sup_{0\ne\zeta_{h}\in Z_{h,0}}\frac{\inp[\K z_{h}]{\zeta_h}}{||\zeta_{h}||_{Z_{h}}}\le C \|z_h\|. \label{eq:inf-stab-bnd-5}
\end{align}
Further, taking test function $v = \text{div}_h \sigma_{h}$ in \eqref{eq:fe-mono-weak-2} yields
\begin{equation}
\|\text{div}_h\,\sigma_{h}\|^2\le\|f\|^2.\label{eq:stab-bnd-6}
\end{equation}
Combining inequalities \eqref{eq:stab-bnd-2}--\eqref{eq:stab-bnd-6} and taking
$\epsilon$ small enough, we obtain
\begin{align}
 & \|\left(\sigma_{h} + \a p_{h}I\right)(t)\|^2 +
\|\text{div}_h\sigma_{h}(t)\|^2 + \|u_{h}(t)\|^{2}+\|\gamma_{h}(t)\|^{2}+c_{0}\|p_{h}(t)\|^{2}\nonumber \\
& \qquad + \int_{0}^{t}\left(\|\sigma_{h}\|^2 + \|\text{div}_h\sigma_{h}\|^2 + \|u_{h}\|^{2}+\|\gamma_{h}\|^{2} + \|z_{h}\|^{2}+\|p_{h}\|^{2}\right)ds\nonumber \\
& \quad \le C\bigg(\int_{0}^{t}\left(\|g\|^{2} + \|\dt\|^{2} + \|f\|^{2}\right) ds + \|f(t)\|^{2}+\|\sigma_{h}(0)\|^{2} + \|u_{h}(0)\|^{2}+\|p_{h}(0)\|^{2}+\|f(0)\|^{2}\bigg).
\label{eq:stab-bnd-6-0-1}
\end{align}

\medskip
\noindent
{\bf Bound on $\|\dvr_h z_{h}\|$.}

We continue with deriving a bound for $\displaystyle \int_0^t\|\text{div}_h \, z_{h}\|^2 ds$. In the process we also obtain bounds on $\|z_{h}(t)\|$ and $\|p_{h}(t)\|$ for all $t\in(0,t]$, which are independent of $c_{0}$. We start by choosing test function $w=\text{div}_hz_{h}$ in \eqref{eq:fe-mono-weak-5} to obtain
\begin{equation}
  \|\text{div}_h \, z_{h}\| \le C\left(c_{0}\|\dt p_{h}\|
  + \|\dt \left(\sigma_{h} + \a p_{h}I\right)\| + \|q\|\right).\label{eq:stab-bond-6-1}
\end{equation}
To bound the time-derivative terms on the right hand side of (\ref{eq:stab-bond-6-1}), differentiate
equations \eqref{eq:fe-mono-weak-1}--\eqref{eq:fe-mono-weak-4} with respect to time and take 
$(\tau,v,\xi,\zeta,w)=(\dt\sigma_{h},\dt u_{h},\dt\gamma_{h},z_{h},\dt p_{h})$ in the differentiated equations and equation \eqref{eq:fe-mono-weak-5}. Combining the resulting equations and integrating in time from $0$ to $t\in(0,T]$, we obtain, similarly to equations (\ref{eq:stab-bnd-1})$\--$(\ref{eq:stab-bnd-2}),
\begin{align}
& 2\int_{0}^{t}\left(\|\dt A^{\frac{1}{2}}\left(\sigma_{h}+\a p_{h}I\right)\|^{2}+c_{0}\dt\|p_{h}\|^{2}\right)ds+\|K^{-\frac{1}{2}}z_{h}(t)\|^{2}\nonumber \\
& \qquad \le\epsilon\left(\int_{0}^{t}\|\dt u_{h}\|^{2}ds+\|p_{h}(t)\|^{2}\right) + \frac{1}{\epsilon}\left(\int_{0}^{t}\|\dt f\|^{2}ds + \|g(t)\|^{2}\right)\nonumber \\
& \qquad \quad +\int_{0}^{t}\left(\|p_{h}\|^{2}+\|\dt g\|^{2}\right)ds+\|K^{-\frac{1}{2}}z_{h}(0)\|^{2}+\|p_{h}(0)\|^{2}+\|g(0)\|^{2}.\label{eq:stab-bnd-7}
\end{align}
To bound $\|\dt u_{h}\|$, we use the inf-sup condition (\ref{eq:inf-sup-elast-weak}) and the time-differentiated \eqref{eq:fe-mono-weak-1} to obtain
\begin{equation}
\|\dt u_{h}\|+\|\dt\gamma_{h}\|\le C\|\dt \left(\sigma_{h}+\a p_{h}I\right)\|.\label{eq:stab-bnd-8}
\end{equation}
Combining inequalities (\ref{eq:inf-stab-bnd-5}), (\ref{eq:stab-bnd-7}),
and (\ref{eq:stab-bnd-8}) and taking $\epsilon$ small enough gives 
\begin{align}
& \int_{0}^{t}\left(\|\dt \left(\sigma_{h}+\a p_{h}I\right)\|^{2}+c_{0}\|\dt p_{h}\|^{2} + \|\dt u_{h}\|^{2}+\|\dt\gamma_{h}\|^{2}\right)ds + \|z_{h}(t)\|^{2}+\|p_{h}(t)\|^{2}\nonumber \\
	& \quad\le C\left(\int_{0}^{t}\left(\|\dt f\|^{2}+\|p_{h}\|^{2}+\|\dt g\|^{2}\right)ds+\|g(t)\|^{2}+\|p_{h}(0)\|^{2}+\|g(0)\|^{2}+\|z_{h}(0)\|^{2}\right).\label{eq:stab-bnd-9}
\end{align}
Integrating (\ref{eq:stab-bond-6-1}) with respect to time
from $0$ to $t\in(0,T]$ and combining it with (\ref{eq:stab-bnd-9}) and \eqref{eq:stab-bnd-6-0-1} gives
  \begin{align}
    & \|p_{h}(t)\|^{2}+\|z_{h}(t)\|^{2} +
    \int_{0}^{t}\|\text{div}_h\,z_{h}\|^{2}ds \le C \bigg(\int_{0}^{t}\big(\|f\|^{2} +
    \|\dt f\|^{2} + \|g\|^{2} + \|\dt g\|^{2}\big)ds  \nonumber \\
    & \qquad 
    + \|f(t)\|^{2}+\|g(t)\|^{2} +\|\sigma_{h}(0)\|^{2} + \|u_h(0)\|^2 +\|p_{h}(0)\|^{2} + \|z_{h}(0)\|^{2} +\|f(0)\|^{2}+\|g(0)\|^{2} \bigg). \label{eq:stab-bnd-10}
\end{align}
We further note that \eqref{eq:stab-bnd-6-0-1} and \eqref{eq:stab-bnd-10} provide bounds on
$\|\left(\sigma_{h}+\a p_{h}I\right)(t)\|$ and $\|p_{h}(t)\|$, which also gives a bound on $\|\sigma_{h}(t)\|$, using
\begin{align}
\|\sigma_{h}\|\le C \left(\|\left(\sigma_{h}+\a p_{h}I\right)(t)\| +\|p_{h}(t)\|\right).\label{eq:stab-bnd-12}
\end{align}

\medskip
\noindent
{\bf Bounds on } $\|\lHu\|$ and $\|\lHp\|$.
    
It remains to bound the Lagrange multipliers $\|\lHu\|$ and $\|\lHp\|$. For this purpose we utilize equations \eqref{eq:monodd-mmmfe1} and \eqref{eq:monodd-mmmfe4}. Combining the inf-sup bound (\ref{eq:mortar-u-inf-sup}) and (\ref{eq:monodd-mmmfe1}) gives 
\begin{align}
\|\lHu\|_{\G} & \le \beta_E \sup_{0\ne\tau\in\X_{h}}\frac{\sum_{i=1}^{N}\gnp[\tau n_{i}]{\lHu}_{\G_{i}}}{\|\tau\|_{\X_{h}}}
= \beta_E \sup_{0\ne\tau\in\X_{h}}\frac{1}{\|\tau\|_{\X_{h}}}\left(\inp[A\left(\sigma_{h}+\a p_{h}I\right)]{\tau} + \inp[u_{h}]{\dvr_h\,{\tau}} + \inp[\gamma_{h}]{\tau}\right) \nonumber \\
& \le C\left(\|\sigma_{h}\| + \|p_{h}\| + \|u_{h}\| + \|\g_{h}\|\right).
\label{lhu-bound}
\end{align}
Similarly, we bound $\|\lHp\|$ by combining the inf-sup bound (\ref{eq:mortar-p-inf-sup})
and (\ref{eq:monodd-mmmfe4}):
\begin{align}
\|\lHp\|_{\G} & \le \beta_D\sup_{0\ne\zeta\in Z_{h}}\frac{\sum_{i=1}^{N}\gnp[\zeta\cdot n_{i}]{\lHp}_{\G_{i}}}{\|\zeta\|_{Z_{h}}}
= \beta_D \sup_{0\ne\zeta\in Z_{h}}\frac{1}{\|\zeta\|_{Z_{h}}}\left(-\inp[\K z_{h}]{\zeta} + \inp[p_{h}]{\dvr_h \,{\zeta}}\right) \nonumber \\
& \le C\left(\|z_{h}\|+\|p_{h}\|\right).
\label{lhp-bound}
\end{align}

\medskip
\noindent
{\bf Bound on the initial data.}
	
In order to bound the initial data $\sigma_{h}(0),\,u_{h}(0),\,z_{h}(0)$,
and $p_{h}(0)$, recall that the discrete initial data is obtained by
taking elliptic projection of the continuous initial data,
cf. (\ref{eq:well-pose-initial-data}). Further note that the continuous
initial data is constructed from the original pressure initial data
$p_{0}$ using the procedure described in the proof of Lemma~\ref{lem:IC}.
Employing the steady-state version of the arguments used in the proof above for the weakly continuous normal stress and velocity formulation of \eqref{eq:well-pose-initial-data}
gives 
\begin{align}
& \|\sigma_{h}(0)\|+\|u_{h}(0)\|+\|\gamma_{h}(0)\|+\|z_{h}(0)\|+\|p_{h}(0)\|\le C\left(\|\s_{0}\|+\|u_{0}\|+\|\g_{0}\|+\|z_{0}\|+\|p_{0}\|\right)\nonumber \\
& \quad\quad\quad\le C\left(\|p_{0}\|_{H^{1}(\Omega)}+\|K\nabla p_{0}\|_{H(\text{div};\Omega)}\right).\label{eq:stab-bnd-13}
\end{align}
	
Finally, bound \eqref{eq:stability-mono-mortar} follows by combining inequalities (\ref{eq:stab-bnd-6-0-1}) and  (\ref{eq:stab-bnd-10})$\--$(\ref{eq:stab-bnd-13}).
\end{proof}

\section{Error analysis}
\label{sec:error}

In this section we establish a combined a priori error estimate
for all unknowns in the method.

\begin{thm}
\label{thm:Error-theorem}Let $\ensuremath{(\sigma_{h}(t),u_{h}(t),\gamma_{h}(t),z_{h}(t),p_{h}(t),\lHu(t),\lHp(t))}\in\X_{h}\times V_{h}\times\mathbb{Q}_{h}\times Z_{h}\times W_{h}\times\Lambda_{H}^u\times\Lambda_H^p$
be the solution to the system of equations \eqref{eq:monodd-mmmfe1}$\--$\eqref{eq:monodd-mmmfe7}
under the assumption \eqref{eq:mortar_assumption} for $t\in[0,T]$,
and suppose that the solution of \eqref{eq:cts1-1}$\--$\eqref{eq:cts5-1}
is sufficiently smooth. Then there exists a positive constant $C$,
independent of $h$, $H$, and $c_{0}$ such that
\begin{align*}
& \|\s-\s_{h}\|_{L^{\infty}\left(0,T;\X_h\right)}+\|u-u_{h}\|_{L^{\infty}\left(0,T;L^{2}\left(\Omg\right)\right)}+\|\g-\g_{h}\|_{L^{\infty}\left(0,T;L^{2}\left(\Omg\right)\right)}+\|z-z_{h}\|_{L^{\infty}\left(0,T;L^{2}\left(\Omg\right)\right)}\\
&\qquad +\|p-p_{h}\|_{L^{\infty}\left(0,T;L^{2}\left(\Omg\right)\right)}+\|u-\lHu\|_{L^{\infty}\left(0,T;L^{2}\left(\G\right)\right)}+\|p-\lHp\|_{L^{\infty}\left(0,T;L^{2}\left(\G\right)\right)}+\|\s-\s_{h}\|_{L^{2}\left(0,T;\X_h\right)}\\
&\qquad +\|u-u_{h}\|_{L^{2}\left(0,T;L^{2}\left(\Omg\right)\right)}+\|\g-\g_{h}\|_{L^{2}\left(0,T;L^{2}\left(\Omg\right)\right)}+\|z-z_{h}\|_{L^{2}\left(0,T;Z_h\right)}+\|p-p_{h}\|_{L^{2}\left(0,T;L^{2}\left(\Omg\right)\right)}\\
&\qquad +\|u-\lHu\|_{L^{2}\left(0,T;L^{2}\left(\G\right)\right)}+\|p-\lHp\|_{L^{2}\left(0,T;L^{2}\left(\G\right)\right)}\\
& \quad
  \le C\Big(h^{k_{1}}\|\s\|_{H^{1}\left(0,T;H^{k_{1}}\left(\Omg\right)\right)}	+h^{k_{2}}H^{\half}\|\s\|_{H^{1}\left(0,T;H^{k_{2}+\half}\left(\Omg\right)\right)} +h^{l_{1}}\|\dvr\s\|_{L^{\infty}\left(0,T;H^{l_{1}}\left(\Omg\right)\right)}
  \\
  & \qquad
  +h^{l_{2}}\|\dvr\s\|_{L^{2}\left(0,T;H^{l_{2}}\left(\Omg\right)\right)}  +h^{l_{3}}\|u\|_{L^{2}\left(0,T;H^{l_{3}}\left(\Omg\right)\right)}+h^{l_{4}}\|u\|_{L^{\infty}\left(0,T;H^{l_{4}}\left(\Omg\right)\right)}+h^{j_{1}}\|\g\|_{H^{1}\left(0,T;H^{j_{1}}\left(\Omg\right)\right)}\\
& \qquad +h^{r_{1}}\|z\|_{H^{1}\left(0,T;H^{r_{1}}\left(\Omg\right)\right)}+h^{r_{2}}H^{\half}\|z\|_{H^{1}\left(0,T;H^{r_{2}+\half}\left(\Omg\right)\right)}+h^{s_{1}}\|\dvr z\|_{L^{2}\left(0,T;H^{s_{1}}\left(\Omg\right)\right)}\\
& \qquad +h^{s_{2}}\|p\|_{H^{1}\left(0,T;H^{s_{2}}\left(\Omg\right)\right)}+H^{m_{1}-\half}\|u\|_{H^{2}\left(0,T;H^{m_{1}+\half}\left(\Omg\right)\right)}+H^{m_{2}-\half}\|p\|_{H^{1}\left(0,T;H^{m_{2}+\half}\left(\Omg\right)\right)}\Big),\\
& \quad\quad\quad\qquad 0 < k_{1}, k_2 \le k+1, \quad\,0\le l_{1},l_{2},l_{3},l_{4}\le l+1,\quad\,0\le j_{1}\le j+1,\\
& \quad\quad\quad\quad\quad 0 < r_{1},r_{2}\le r+1,\,\quad0\le s_{1},s_{2}\le s+1,\,\quad0\le m_{1},m_{2}\le m+1.
\end{align*}
\end{thm}

\begin{proof}
First, note that the solution to \eqref{eq:cts1-1}$\--$\eqref{eq:cts5-1} satisfies,
for $1\le i \le N$,
\begin{align}
& \inp[A(\sigma + \a p I)]{\tau}_{\Oi} + \inp[u]{\dvr{\tau}}_{\Oi}+\inp[\gamma]{\tau}_{\Oi}-\gnp[u]{\t\,n_i}_{\Gamma_i}=0, & \forall\tau\in\X_i,\label{eq:cts1-sub}\\
& \inp[\dvr{\sigma}]{v}_{\Oi}=-\inp[f]{v}_{\Oi}, & \forall v\in V_i,\label{eq:cts2-sub}\\
& \inp[\sigma]{\xi}_{\Oi}=0, & \forall\xi\in\W_i,\label{eq:cts3-sub}\\
& \inp[\K z]{\zeta}_{\Oi}-\inp[p]{\dvr{\zeta}}_{\Oi}+\gnp[p]{\zeta\cdot n_i}_{\Gamma_i}=0, & \forall \zeta\in Z_i,\label{eq:cts4-sub}\\
& c_{0}\inp[\dt{p}]{w}_{\Oi}+\a\inp[\dt A(\sigma + \a pI)]{wI}_{\Oi}
+\inp[\dvr{z}]{w}_{\Oi}=\inp[g]{w}_{\Oi}, & \forall w\in W_i,\label{eq:cts5-sub}
\end{align}
Subtracting the weakly continuous normal stress and velocity system (\ref{eq:fe-mono-weak-1})$\--$(\ref{eq:fe-mono-weak-5})
from (\ref{eq:cts1-sub})$\--$(\ref{eq:cts5-sub}) gives
\begin{align}
  & \inp[A\left(\left(\s-\sigma_{h}\right)+\a\left(p-p_{h}\right)I\right)]{\tau}
  +\inp[u-u_{h}]{\dvr_h {\tau}}
  + \inp[\gamma-\gamma_{h}]{\tau} - \sum_{i=1}^{N}\gnp[u]{\tau n_i}_{\Gamma_{i}} = 0, &&  \forall\tau\in\X_{h,0},\label{eq:error-1}\\
& \inp[\dvr_h{\left(\s-\sigma_{h}\right)}]{v} = 0, &&  \forall v\in V_{h},\label{eq:error-2}\\
	& \inp[\left(\s-\s_{h}\right)]{\xi}=0, && \forall\xi\in\mathbb{Q}_{h},\label{eq:error-3}\\
  & \inp[\K\left(z-z_{h}\right)]{\zeta}-\inp[p-p_{h}]{\dvr_h{\zeta}}
  + \sum_{i=1}^{N}\gnp[p]{\zeta\cdot n_{i}}_{\Gamma_{i}} = 0, &&  \forall\zeta\in Z_{h,0},\label{eq:error-4}\\
  & c_{0}\inp[\dt(p-p_{h})]{w}
  +\a(\dt A((\s-\s_{h}) + \a(p-p_{h})I),wI)
  + (\dvr_h(z-z_{h}),w) = 0, && \forall w\in W_{h}. \label{eq:error-5}
\end{align}
Next, rewrite the above error equations in terms of the approximation
errors $\psi_{\star}$ and discretization errors $\phi_{\star}$, for $\star\in\{\s,u,\g,z,p,\lambda^u,\lambda^p\}$
as follows: 
\begin{align*}
\sigma-\sigma_{h} & =(\sigma-\Pi_{0}^{\s}\sigma)+(\Pi_{0}^{\s}\sigma-\sigma_{h}):=\psi_{\s}+\phi_{\s},\\
u-u_{h} & =(u-{\cal P}_{h}^{u}u)+({\cal P}_{h}^{u}u-u_{h}):=\psi_{u}+\phi_{u},\\
\gamma-\gamma_{h} & =(\gamma-{\cal R}_{h}\gamma)+({\cal R}_{h}\gamma-\gamma_{h}):=\psi_{\gamma}+\phi_{\gamma},\\
	z-z_{h} & =(z-\Pi_{0}^{z}z)+(\Pi_{0}^{z}z-z_{h}):=\psi_{z}+\phi_{z},\\
	p-p_{h} & =(p-{\cal P}_{h}^{p}p)+({\cal P}_{h}^{p}p-p_{h}):=\psi_{p}+\phi_{p},\\
	u-\lHu & =(u-\Quh u)+(\Quh u-\lHu):=\psi_{\lambda^u}+\phi_{\lambda^u},\\
	p-\lHp & =(p-\Qph p)+(\Qph p-\lHp):=\psi_{\lambda^p}+\phi_{\lambda^p}.
\end{align*}
Combining (\ref{eq:error-2}) with (\ref{eq:pi-weak-1}) gives 
\begin{equation}
  \dvr_h\phi_{\sigma}=0,
  \label{eq:err-proof-1}
\end{equation}
and (\ref{eq:error-3}) combined with (\ref{eq:pi-weak-2}) gives
\begin{equation}
\inp[\phi_{\s}]{\xi}=0\,\,\,\text{for }\xi\in\mathbb{Q}_{h}.\label{eq:err-proof-2}
\end{equation}
We rewrite error equation (\ref{eq:error-1}) as 
\begin{align}
	  & \inp[A\left(\phi_{\s} + \alpha\phi_{p}I\right)]{\tau}
          + \inp[\phi_u]{\dvr_h{\tau}}
          + \inp[\phi_{\gamma}]{\tau}
          \nonumber \\
          & \qquad
          = -\inp[A\left(\psi_{\s}+\alpha\psi_{p}I\right)]{\tau}-\inp[\psi_{\gamma}]{\tau}
          +\sum_{i=1}^{N}\gnp[u-{\cal I}_{H}^{u}u]{\tau n_i}_{\Gamma_{i}},
          \label{eq:err-proof-2-1}
\end{align}
where we have used that $\sum_{i=1}^{N}\gnp[{\cal I}_{H}^{u}u]{\tau n_i}_{\Gamma_{i}}=0$
for any $\tau\in\X_{h,0}$. Differentiating the above equation with
respect to time gives 
\begin{align}
& \inp[\partial_{t}A\left(\phi_{\s}+\alpha\phi_{p}I\right)]{\tau} + \inp[\partial_{t}\phi_u]{\dvr_h{\tau}} +\inp[\partial_{t}\phi_{\gamma}]{\tau}\nonumber \\
& \quad \quad = -\inp[\partial_{t}A\left(\psi_{\s}+\alpha\psi_{p}I\right)]{\tau} -\inp[\partial_{t}\psi_{\gamma}]{\tau} + \sum_{i=1}^{N}\gnp[\partial_{t}\left(u-{\cal I}_{H}^{u}u\right)]{\tau n_i}_{\Gamma_{i}}.\label{eq:err-proof-3}
\end{align}
Taking $\tau=\phi_{\s}$ in (\ref{eq:err-proof-3}) and using (\ref{eq:err-proof-1})
and (\ref{eq:err-proof-2}) gives 
\begin{align}
	& \inp[\partial_{t}A\left(\phi_{\s}+\alpha\phi_{p}I\right)]{\phi_{\s}}=-\inp[\partial_{t}A\left(\psi_{\s}+\alpha\psi_{p}I\right)]{\phi_{\s}}-\inp[\partial_{t}\psi_{\gamma}]{\phi_{\s}}+\sum_{i=1}^{N}\inp[\partial_{t}\left(u-{\cal I}_{H}^{u}u\right)]{\phi_{\s}n_i}_{\Gamma_{i}}.\label{eq:err-proof-4}
\end{align}
Error equation (\ref{eq:error-5}) can be written as 
\begin{align}
  c_{0}\inp[\partial_{t}\phi_{p}]{w} + \a\inp[\partial_{t}A\left(\phi_{\s} + \a\phi_{p}I\right)]{wI} + \inp[\dvr_h{\phi_{z}}]{w} = -\a\inp[\partial_{t}A\left(\psi_{\s} + \a\psi_{p}I\right)]{wI},
 \label{eq:err-proof-5} 
\end{align}
where we have used \eqref{eq:pressure-project} and \eqref{eq:pi-weak-d-1}.	
Taking $w=\phi_{p}$ in (\ref{eq:err-proof-5}) and combining the
resulting equation with (\ref{eq:err-proof-4}), we obtain
\begin{align}
  & \frac{1}{2}\partial_{t}\left(\|A^{\frac{1}{2}}\left(\phi_{\s}+\alpha\phi_{p}I\right)\|^{2}+c_{0}\|\phi_{p}\|^{2}\right) + \inp[\dvr_h{\phi_{z}}]{\phi_{p}} \nonumber \\
 & \qquad =-\inp[\partial_{t}A\left(\psi_{\s}+\alpha\psi_{p}I\right)]{\phi_{\s}+\alpha\phi_{p}I}
-\inp[\partial_{t}\psi_{\gamma}]{\phi_{\s}}+\sum_{i=1}^{N}\gnp[\partial_{t}\left(u-{\cal I}_{H}^{u}u\right)]{\phi_{\s}n_i}_{\Gamma_{i}}.\label{eq:err-proof-6}
\end{align}
Error equation (\ref{eq:error-4}) can be written as 
\begin{align}
& \inp[\K\phi_{z}]{\zeta}- \inp[\phi_{p}]{\dvr_h{\zeta}} =-\inp[\K\psi_{z}]{\zeta}+\sum_{i=1}^{N}\gnp[{\cal I}_{H}^{p}p-p]{\zeta\cdot n_{i}}_{\Gamma_{i}},\label{eq:err-proof-8}
\end{align}
where we have used \eqref{eq:pressure-project} and
\begin{equation}
\sum_{i=1}^{N}\gnp[{\cal I}_{H}^{p}p]{\zeta\cdot n_{i}}_{\Gamma_{i}}=0 \quad \forall \zeta \in Z_{h,0}. \label{eq:cont_lagrang_p}
\end{equation}
Taking test function $\zeta=\phi_{z}$ in equation (\ref{eq:err-proof-8}) 
and combining the resulting equation with (\ref{eq:err-proof-6}) gives
\begin{align}
& \frac{1}{2}\partial_{t}\left(\|A^{\frac{1}{2}}\left(\phi_{\s}+\alpha\phi_{p}I\right)\|^{2}+c_{0}\|\phi_{p}\|^{2}\right)+\|K^{-\half}\phi_{z}\|^{2} = -\inp[\partial_{t}A\left(\psi_{\s} + \alpha\psi_{p}I\right)]{\phi_{\s} + \alpha\phi_{p}I}\nonumber \\
  & \quad\quad -\inp[\partial_{t}\psi_{\gamma}]{\phi_{\s}} - \inp[\K\psi_{z}]{\phi_{z}}
  -\sum_{i=1}^{N}\gnp[\partial_{t}({\cal I}_{H}^{u}u-u)]{\phi_{\s}n_i}_{\Gamma_{i}}
+\sum_{i=1}^{N}\gnp[{\cal I}_{H}^{p}p-p]{\phi_{z}\cdot n_{i}}_{\Gamma_{i}}.\label{eq:err-proof-9}
\end{align}
We bound the first three terms on the right hand side of (\ref{eq:err-proof-9}) as follows:
\begin{align}
& |\inp[\partial_{t}A\left(\psi_{\s}+\alpha\psi_{p}I\right)]{\phi_{\s}+\alpha\phi_{p}I}|
  + |\inp[\partial_{t}\psi_{\gamma}]{\phi_{\s}}| + |\inp[\K\psi_{z}]{\phi_{z}}| \nonumber \\
& \quad\quad \le\|\partial_{t}A\left(\psi_{\s}+\alpha\psi_{p}I\right)\|\|\phi_{\s}+\alpha\phi_{p}I\|+\|\partial_{t}\psi_{\gamma}\|\|\phi_{\s}\|+\|\K\psi_{z}\|\|\phi_{z}\|\nonumber \\
	& \quad\quad\quad\quad\le \frac{C}{\epsilon}\left(\|\partial_{t}\psi_{\s}\|^{2}+\|\partial_{t}\psi_{p}\|^{2}+\|\partial_{t}\psi_{\gamma}\|^{2}+\|\psi_{z}\|^{2}\right)+\epsilon\left(\|\phi_{\s}\|^{2}+\|\phi_{p}\|^{2}+\|\phi_{z}\|^{2}\right),\label{eq:err-proof-10}
\end{align}
where we have used Young's inequality for a some $\epsilon > 0$. 
	
Next, we give a bound on the last two boundary terms in the right
hand side of equation (\ref{eq:err-proof-9}). For this, we note that
the following bounds hold for any $(\tau,v)\in H(\text{div};\Omega,\mathbb M) \times H^1_0(\Omega,\mathbb R^d)$ and $(\zeta,w)\in H(\text{div};\Omega) \times H^1_0(\Omega)$:
\begin{align}
\gnp[{\cal I}_{H}^{u}v-v]{\tau n_i}_{\Gamma_{i}} & = \gnp[E_{i}({\cal I}_{H}^{u}v-v)]{\tau n_{i}}_{\dOi} \nonumber\\
& \le C\|E_{i}({\cal I}_{H}^{u}v-v)\|_{\frac{1}{2},\dOi}\|\tau\|_{H(\text{div};\Oi)} 
\le C\|{\cal I}_{H}^{u}v-v\|_{\frac{1}{2},\Gi}\|\tau\|_{H(\text{div};\Oi)},
\label{eq:err-proof-10-1} \\
\gnp[{\cal I}_{H}^{p}w-w]{\zeta \cdot n_i}_{\Gamma_{i}} &
= \gnp[E_{i}({\cal I}_{H}^{p}\zeta-\zeta)]{\zeta \cdot n_{i}}_{\dOi} \nonumber \\
& \le C\|E_{i}({\cal I}_{H}^{p}\zeta-\zeta)\|_{\frac{1}{2},\dOi}\|\zeta\|_{H(\text{div};\Oi)} 
\le C\|{\cal I}_{H}^{p}\zeta-\zeta\|_{\frac{1}{2},\Gi}\|\zeta\|_{H(\text{div};\Oi)}, \label{eq:err-proof-10-2}
\end{align}
where $E_{i}$ denotes the extension by zero from $\Gamma_i$ to 
$\dOi$, which is continuous in the $H^{\frac12}$-norm for functions that are zero on $\partial\Gamma$, and we have used the trace inequalities in (\ref{eq:trace-2}). Taking $(\tau,v)=(\fs,\dt u)$ and $(\zeta,w)=(\fz,p)$ in (\ref{eq:err-proof-10-1}) and (\ref{eq:err-proof-10-2}), respectively, and using Young's inequality gives
\begin{align}        
  & \gnp[\partial_{t}({\cal I}_{H}^{u}u-u)]{\phi_{\s}n_i}_{\Gamma_{i}}
  \le \frac{C}{\epsilon}\|\partial_{t}({\cal I}_{H}^{u}u-u)\|_{\frac{1}{2},\Gi}^2 + \epsilon \|\phi_{\s}\|_{\Omega_i}^2, \label{bound-gamma-1}\\
  & \gnp[{\cal I}_{H}^{p}p-p]{\phi_{z}\cdot n_{i}}_{\Gamma_{i}}
  \le \frac{C}{\epsilon}\|{\cal I}_{H}^{p}p-p\|_{\frac{1}{2},\Gi}^2
  + \epsilon(\|\phi_{z}\|_{\Omega_i}^2 + \|\dvr \phi_z\|_{\Omega_i}^2),
  \label{bound-gamma-2}
\end{align}
where we also used \eqref{eq:err-proof-1}. Combining inequalities \eqref{eq:err-proof-9}--\eqref{bound-gamma-1} and integrating with respect to time from $0$ to $t\in(0,T]$ gives 
\begin{align}
& \|A^{\frac{1}{2}}\left(\phi_{\s} +\alpha\phi_{p}I\right)(t)\|^{2} + c_{0}\|\phi_{p}(t)\|^{2} + \int_{0}^{t}\|K^{-\frac12}\phi_{z}\|^{2}\nonumber \\
  & \quad \le C\int_{0}^{t}\left(\|\partial_{t}\psi_{\s}\|^{2} +\|\partial_{t}\psi_{p}\|^{2} +\|\partial_{t}\psi_{\gamma}\|^{2} +\|\psi_{z}\|^{2} + \|{\cal I}_{H}^{u}\dt u-\dt u\|_{\frac{1}{2},\G}^{2}
  + \|{\cal I}_{H}^{p}p-p\|_{\frac{1}{2},\G}^{2} \right)ds   \nonumber \\
  & \qquad
  +\epsilon\int_{0}^{t}\left(\|\phi_{\s}\|^{2} +\|\phi_{p}\|^{2} +\|\phi_{z}\|^{2}\right)ds
  + C\int_{0}^{t}\|\text{div}_h\phi_{z}\|^2 ds 
+\|A^{\frac{1}{2}}\left(\phi_{\s}+\alpha\phi_{p}I\right)(0)\|^{2}+c_{0}\|\phi_{p}(0)\|^{2}.
\label{eq:err-proof-13}
\end{align}
	
Next, we bound the errors of the form $\phi_{\star}$ for $\star\in\{\s,\gamma,u,p\}$. 
Using the inf-sup condition (\ref{eq:inf-sup-elast-weak}), the error equation \eqref{eq:err-proof-2-1}, and
(\ref{eq:err-proof-1}) gives
\begin{align}
& \|\phi_{u}\|+\|\phi_{\gamma}\| \le C_E \sup_{0\ne\tau\in\mathbb{X}_{h,0}}\frac{\inp[\phi_{u}]{\dvr_h{\tau}} +\inp[\phi_{\gamma}]{\tau}}{\|\tau\|_{\X_{h}}}\nonumber \\
  & \quad = C_E \sup_{0\ne\tau\in\mathbb{X}_{h,0}}\frac{1}{\|\tau\|_{\X_{h}}}\Big(\inp[A\left(\phi_{\s} +\alpha\phi_{p}I\right)]{\tau} +\inp[A\left(\psi_{\s}+\alpha\psi_{p}I\right)]{\tau}+\inp[\psi_{\gamma}]{\tau}
  -\sum_{i=1}^{N}\inp[\left({\cal I}_{H}^{u}u-u\right)]{\tau n_i}_{\Gamma_{i}}\Big)\nonumber \\
  & \quad\le C\Big(\|\phi_{\s}+\alpha\phi_{p}I\|+\|\psi_{\s}\|+\|\psi_{\gamma}\| +\|\psi_{p}\|+\|{\cal I}_{H}^{u}u-u\|_{\frac{1}{2},\G}\Big),
  \label{eq:err-proof-14}
\end{align}
where we have used
(\ref{eq:err-proof-10-1}) with $v=u$ in the last inequality. The above inequality implies
\begin{align}
  & \int_{0}^{t}\left(\|\phi_{u}\|^{2} +\|\phi_{\gamma}\|^{2}\right)ds \le C\int_{0}^{t}\Big(\|\phi_{\s}\|^{2} + \|\phi_{p}\|^{2}
+  \|\psi_{\s}\|^{2}+\|\psi_{\gamma}\|^{2}+\|\psi_{p}\|^{2}+\|{\cal I}_{H}^{u}u-u\|_{\frac{1}{2},\G}^{2}\Big)ds.\label{eq:err-proof-15}
\end{align}
To bound $\|\phi_{p}\|$, we use the inf-sup condition (\ref{eq:inf-sup-darcy-weak}) and the error equation \eqref{eq:err-proof-8} to get 
\begin{align}
& \|\phi_{p}\|\le C_D\sup_{0\ne\zeta\in Z_{h,0}}\frac{\sum_{i=1}^{N}\left(\dvr\zeta,\phi_{p}\right)_{\Omega_{i}}}{\|\zeta\|_{Z_{h}}}\nonumber \\
	& \quad\quad = C_D \sup_{0\ne\zeta\in Z_{h,0}}\frac{\inp[\K\phi_{z}]{\zeta}+\inp[\K\psi_{z}]{\zeta}-\sum_{i=1}^{N}\gnp[{\cal I}_{H}^{p}p-p]{\zeta\cdot n_{i}}_{\Gamma_{i}}}{||\zeta||_{Z_{h}}}\nonumber \\
& \quad\quad\le C\Big(\|\psi_{z}\| + \|\phi_{z}\|+\|{\cal I}_{H}^{p}p-p\|_{\frac{1}{2},\G}\Big), \label{eq:err-proof-16}
\end{align}
where we have used (\ref{eq:err-proof-10-2}) with $w=p$
to obtain the last inequality. The above inequality yields
\begin{equation}
\int_{0}^{t}\|\phi_{p}\|^{2}ds\le C\int_{0}^{t}\Big(\|\psi_{z}\|^{2}+\|\phi_{z}\|^{2} +\|{\cal I}_{H}^{p}p-p\|_{\frac{1}{2},\G}^{2}\Big)ds.\label{eq:err-proof-17}
\end{equation}
	
To bound the term $\displaystyle\intg\|\fs\|^{2}ds$, which appears on the right-hand side of \eqref{eq:err-proof-15}, we take $\tau=\fs$ in (\ref{eq:err-proof-2-1})
and $\xi=\fg$ in (\ref{eq:error-3}), and use \eqref{eq:err-proof-1}$\--$\eqref{eq:err-proof-2} to get
\begin{align}
  & \|\Ahalf\fs\|^{2}=-\inp[\Ahalf\alpha\phi_{p}I]{\fs}-\inp[A\left(\psi_{\s}+\alpha\psi_{p}I\right)]{\fs}-\inp[\psi_{\gamma}]{\fs}-\sum_{i=1}^{N}\gnp[{\cal I}_{H}^{u}u-u]{\fs n_i}_{\Gamma_{i}}
+\inp[\psi_{\s}]{\fg}
  \nonumber \\
  & \qquad \le C \Big( \big(\|\fp\|+\|\ss\|+\|\sp\|+\| \sg \|
  + \|{\cal I}_{H}^{u}u-u\|_{\frac{1}{2},\G} \big)\| \fs \|
    +\|\ss\|\|\fg\|\Big)
  \nonumber \\
  & \qquad
  \le \frac{C}{\epsilon}
  \Big(
  \|\fp\|^{2}
  + \|\ss\|^{2}+\|\sp\|^{2}+\|\sg\|^{2}
  +\|{\cal I}_{H}^{u}u-u\|_{\frac{1}{2},\G}^{2}\Big)
  +\epsilon\left(\|\fs\|^{2}+\|\fg\|^{2}\right),\label{eq:err-proof-18}
\end{align}
where we have used (\ref{eq:err-proof-10-1}), \eqref{eq:err-proof-1}, and Young's inequality. Integrating (\ref{eq:err-proof-18}) with respect to time from $0$ to $t\in(0,T]$, and taking $\epsilon$ small enough, we get 
\begin{align}
  & \intg\|\fs\|^{2}ds\le C\intg\Big(
  \|\fp\|^{2} +
  \|\ss\|^{2}+\|\sp\|^{2}+\|\sg\|^{2}+\|{\cal I}_{H}^{u}u-u\|_{\frac{1}{2},\G}^{2}
  \Big)ds
    +\epsilon\intg\|\fg\|^{2}ds.\label{eq:err-proof-19}
\end{align}
	
Combining (\ref{eq:err-proof-13})$\--$(\ref{eq:err-proof-19}) and
(\ref{eq:err-proof-1}), and taking $\epsilon$ small enough gives
\begin{align}
& \|\fs+\a\fp I\|^{2}+\|\fu\|^{2}+\|\fg\|^{2}+ c_0\|\fp\|^{2}
+\|\dvr_h\fs\|^{2}\nonumber \\
& \qquad +\intg\left(\|\fs\|^{2}+\|\fu\|^{2}+\|\fg\|^{2}+\|\fz\|^{2}
+\|\fp\|^{2}+\|\dvr_h\fs\|^{2}\right) ds \nonumber \\
& \quad \le C \bigg( \intg\left(\|\dt\ss\|^{2} +\|\dt\sp\|^{2}+\|\dt\sg\|^{2}+\|\ss\|^{2}+\|\sp\|^{2}+\|\sg\|^{2}+\|\sz\|^{2}\right)ds
\nonumber \\
& \qquad
+ \|\ss\|^{2} +\|\sp\|^{2}+\|\sg\|^{2}
+ \|({\cal I}_{H}^{u}u-u)(t)\|_{\frac{1}{2},\G}^{2} + \|({\cal I}_{H}^{p}p-p)(t)\|_{\frac{1}{2},\Gi}^{2}
\nonumber \\
& \qquad + \intg\left(\|{\cal I}_{H}^{u}\dt u-\dt u\|_{\frac{1}{2},\G}^{2}
+\|{\cal I}_{H}^{u}u-u\|_{\frac{1}{2},\G}^{2} +\|{\cal I}_{H}^{p}p-p\|_{\frac{1}{2},\G}^{2}\right)ds
+ \|\fs(0)\|^{2}+\|\fp(0)\|^{2} \bigg).\label{eq:err-proof-19-1}
\end{align}
\textbf{Bound on $\|\dvr_h\phi_{z}\|$.}

\smallskip
Next, we obtain a $L^2(0,T)$ bound on the error in $\dvr_h z_h$, as well as bounds on the error in $\|z_{h}(t)\|$ and $\|p_{h}(t)\|$ for all $t\in(0,t]$, which are independent of $c_{0}$. We start by taking $w=\phi_{z}$ in (\ref{eq:error-5}) to get
\begin{align*}
& \|\dvr_h\phi_{z}\|^{2} = -\inp[c_{0}\dt\fp]{\dvr_h\fz} -\inp[c_{0}\dt\sp]{\dvr_h\fz} -\a\inp[\dt A\left(\fs+\a\fp I\right)]{(\dvr\fz) I}\\
  & \qquad -\a\inp[\dt A\left(\ss+\a\sp I\right)]{(\dvr_h\fz) I} - \inp[\sz]{\dvr_h\fz}\\
& \quad
  = -\inp[c_{0}\dt\fp]{\dvr_h\fz} -\a\inp[\dt A\left(\fs+\a\fp I\right)]{(\dvr_h\fz) I}
  - \a\inp[\dt A\left(\ss+\a\sp I\right)]{(\dvr_h\fz) I},
\end{align*}
where the last equality follows from (\ref{eq:pressure-project})
and (\ref{eq:pi-weak-d-1}). The above inequality implies
\begin{align}
  \int_0^T\|\dvr_h\phi_{z}\|^{2} ds \le C\int_0^T\left(c_0\|\dt\fp\|^{2} + \|\dt(\fs+\a\fp I)\|^{2} + \|\sp\|^{2}
  +\|\sz\|^2 +\|\ss\|^{2}\right) ds.\label{eq:err-proof-21}
\end{align}
In order to bound $c_0\|\dt\fp\|^{2}$ and $\|\dt(\fs+\a\fp I)\|^{2}$, we 
differentiate in time (\ref{eq:err-proof-2}) and (\ref{eq:err-proof-8}), combine them with (\ref{eq:err-proof-3}) and \eqref{eq:err-proof-5}, and take test functions $\t=\dt\fs$, $\xi=\dt\fg$, $\zeta=\phi_{z}$, and $w=\dt\fp$ to get the following time differentiated version of
(\ref{eq:err-proof-9}): 
\begin{align}
  & \|\partial_{t}A^{\frac{1}{2}}\left(\phi_{\s}+\alpha\phi_{p}I\right)\|^{2} +c_{0}\|\dt\phi_{p}\|^{2}+\half\dt\|K^{-\half}\phi_{z}\|^{2}
\nonumber \\  & \quad =-\inp[\partial_{t}A\left(\psi_{\s}+\alpha\psi_{p}I\right)]{\dt\left(\phi_{\s}+\alpha\phi_{p}I\right)}
-\inp[\partial_{t}\psi_{\gamma}]{\dt\left(\phi_{\s}+\alpha\phi_{p}I\right)}-\inp[\dt\K\psi_{z}]{\phi_{z}} \nonumber \\
& \qquad -\sum_{i=1}^{N}\gnp[\partial_{t}\left({\cal I}_{H}^{u}u-u\right)]{\dt\phi_{\s}n_i}_{\Gamma_{i}}
+\sum_{i=1}^{N}\gnp[{\cal I}_{H}^{p}\dt p-\dt p]{\phi_{z}\cdot n_{i}}_{\Gamma_{i}},
\label{eq:err-proof-22}
\end{align}
where we have used the fact that $\inp[\partial_{t}\psi_{\gamma}]{\dt\alpha\phi_{p}I}=0$
to write
\begin{equation*}	\inp[\partial_{t}\psi_{\gamma}]{\dt\phi_{\s}}=\inp[\partial_{t}\psi_{\gamma}]{\dt\left(\phi_{\s}+\alpha\phi_{p}I\right)}.
\end{equation*}
Using the Cauchy-Schwarz and Young's inequalities for the first three terms on the right in \eqref{eq:err-proof-22} with $\epsilon > 0$ and taking $\epsilon$ small enough results in
\begin{align}
  & \|\partial_{t}(\phi_{\s}+\alpha\phi_{p}I)\|^{2} + c_{0}\|\dt\phi_{p}\|^{2} + \dt\|\phi_{z}\|^{2}
  \le C
\bigg( \|\dt\psi_{\s}\|^2 + \|\dt\psi_p\|^2 + \|\dt\psi_\gamma\|^2 + \|\dt\psi_z\|^2
+ \epsilon\|\phi_z\|^2
\nonumber \\
& \qquad
  + \bigg|\sum_{i=1}^{N}\gnp[\partial_{t}\left({\cal I}_{H}^{u}u-u\right)]{\dt\phi_{\s}n_i}_{\Gamma_{i}}\bigg|
+ \bigg|\sum_{i=1}^{N}\gnp[{\cal I}_{H}^{p}\dt p-\dt p]{\phi_{z}\cdot n_{i}}_{\Gamma_{i}}\bigg| \bigg).
\label{eq:err-proof-22a}
\end{align}
To bound $\gnp[\partial_{t}\left({\cal I}_{H}^{u}u-u\right)]{\dt\phi_{\s}n_i}_{\Gamma_{i}}$,
we use integration by parts to rewrite it as
\begin{align}
  & \gnp[\partial_{t}\left({\cal I}_{H}^{u}u-u\right)]{\dt\phi_{\s}n_i}_{\Gamma_{i}}=\ddt\left(\gnp[\partial_{t}\left({\cal I}_{H}^{u}u-u\right)]{\phi_{\s}n_i}_{\Gamma_{i}}\right)
  -\gnp[\partial_{t}^{2}\left({\cal I}_{H}^{u}u-u\right)]{\phi_{\s}n_i}_{\Gamma_{i}}.\label{eq:err-proof-23}
\end{align}
To bound the last term on the right in \eqref{eq:err-proof-23}
we take $(\tau,v)=(\fs,\dt^{2}u)$ in (\ref{eq:err-proof-10-1}) and use \eqref{eq:err-proof-1} to get 
\begin{equation}
  \Big|\gnp[\partial_{t}^{2}\left({\cal I}_{H}^{u}u-u\right)]{\phi_{\s}n_i}_{\Gamma_{i}}\Big|
  \le C \|{\cal I}_{H}^{u}\dt^{2}u-\dt^{2}u\|_{\frac{1}{2},\Gi} \|\fs\|_{L^2(\Oi)}
\le \frac{C}{\epsilon}\|{\cal I}_{H}^{u}\dt^{2}u-\dt^{2}u\|_{\frac{1}{2},\Gi}^2 + \epsilon\|\fs\|_{L^2(\Oi)}^2.\label{eq:err-proof-24}
\end{equation}
To bound the term $\gnp[{\cal I}_{H}^{p}\dt p-\dt p]{\phi_{z}\cdot n_{i}}_{\Gamma_{i}}$,
in \eqref{eq:err-proof-22a} we take $(\zeta,w)=(\fz,\dt p)$ in (\ref{eq:err-proof-10-2}) to get 
\begin{align}
  & \Big|\gnp[{\cal I}_{H}^{p}\dt p-\dt p]{\fz\cdot n_{i}}_{\Gamma_{i}}\Big|
  \le C \|{\cal I}_{H}^{p}\dt p-\dt p\|_{\frac{1}{2},\Gi} \|\fz\|_{H(\dvr;\Oi)}
\le \frac{C}{\epsilon}\|{\cal I}_{H}^{p}\dt p-\dt p\|_{\frac{1}{2},\Gi}^2 + \epsilon \|\fz\|_{H(\dvr;\Oi)}^2.\label{eq:err-proof-25}
\end{align}
Combining (\ref{eq:err-proof-22a})$\--$(\ref{eq:err-proof-25}),
integrating with respect to time from $0$ to $t\in(0,T]$, and using \eqref{bound-gamma-2} for the first term on the right in \eqref{eq:err-proof-23}, we obtain
\begin{align}
& \|\phi_{z}\|^{2} +\intg\left(\|\partial_{t}(\phi_{\s}+\alpha\phi_{p}I)\|^{2} +c_0\|\dt\phi_{p}\|^{2}\right)ds \nonumber \\
  & \ \ \le C\intg\Big(\|\dt\ss\|^{2}+\|\dt\sp\|^{2}+\|\dt\sg\|^{2}+\|\dt\sz\|^{2}+\|{\cal I}_{H}^{u}\dt^{2}u-\dt^{2}u\|_{\frac{1}{2},\G}^{2}
  +\|{\cal I}_{H}^{p}\dt p-\dt p\|_{\frac{1}{2},\G}^{2}\Big)ds
  \nonumber \\
  & \qquad
  +C\|({\cal I}_{H}^{u}\dt u -\dt u)(t)\|_{\frac{1}{2},\G}^{2}
+\epsilon\left(\intg\left(\|\fs\|^{2}+\|\fz\|^{2}+\|\dvr_h\fz\|^{2}\right)ds+\|\fs(t)\|^{2}\right)\nonumber \\
& \qquad +C\bigg(\|\phi_{z}(0)\|^{2} + \|\fs(0)\|^{2}
+\|({\cal I}_{H}^{u}\dt u -\dt u)(0)\|_{\half,\G}^{2}\bigg).\label{eq:err-proof-26}
\end{align}
Combining (\ref{eq:err-proof-21}) and (\ref{eq:err-proof-26}) and taking $\epsilon$ small enough implies
\begin{align}
  & \|\phi_{z}\|^{2} + \intg\|\dvr_h\phi_{z}\|^{2}ds\nonumber \\
  & \ \ \le C\intg\Big(\|\dt\ss\|^{2}+\|\dt\sp\|^{2}+\|\dt\sg\|^{2}+\|\dt\sz\|^{2}
  + \|\psi_p\|^2 + \|\psi_\sigma\|^2
  \nonumber \\
  & \ \ \
  +\|{\cal I}_{H}^{u}\dt^{2}u-\dt^{2}u\|_{\frac{1}{2},\G}^{2}
  +\|{\cal I}_{H}^{p}\dt p-\dt p\|_{\frac{1}{2},\G}^{2}\Big)ds
  +C\|({\cal I}_{H}^{u}\dt u -\dt u)(t)\|_{\frac{1}{2},\G}^{2}
  \nonumber \\
  & \ \ \
  +\epsilon\Big(\intg\left(\|\fs\|^{2}+\|\fz\|^{2}\right)ds +\|\fs(t)\|^{2}\Big)
  +C\Big(\|\phi_{z}(0)\|^{2} + \|\fs(0)\|^{2}
+\|({\cal I}_{H}^{u}\dt u-\dt u)(0)\|_{\half,\G}^{2}\Big).\label{eq:err-proof-27}
\end{align}
Finally, combining \eqref{eq:err-proof-27} with \eqref{eq:err-proof-19-1} taking $\epsilon$ small enough, and using \eqref{eq:err-proof-16} and the inequality
\begin{equation*}
  \|\fs\|\le C\left(\|\phi_{\s} +\alpha\phi_{p}I\| + \|\fp\|\right),
\end{equation*}
we arrive at
\begin{align}
& 
\|\fs(t)\|_{\X_{h}}^{2}+\|\fu(t)\|^{2}+\|\fg(t)\|^{2} +\|\fz(t)\|^{2} + \|\fp(t)\|^{2}
\nonumber \\
&\qquad\qquad
+\intg \left(\|\fs\|_{\X_{h}}^{2}
+\|\fu\|^{2}+\|\fg\|^{2}+\|\fz\|_{Z_h}^{2}
+\|\fp\|^{2} \right) ds 
\nonumber \\
& \quad \le C \bigg(\intg\Big(\|\dt\ss\|^{2}+\|\dt\sp\|^{2}+\|\dt\sg\|^{2}+\|\dt\sz\|^{2}
+\|\ss\|^{2} +\|\sp\|^{2}+\|\sg\|^{2}+\|\sz\|^{2}\Big)ds 
\nonumber \\
&\qquad
+\intg\Big(\|{\cal I}_{H}^{p}u-u\|_{\frac{1}{2},\Gamma}^{2}
          +\|\dt({\cal I}_{H}^{u} u-u)\|_{\frac{1}{2},\Gamma}^{2}
	  +\|\dt^{2}({\cal I}_{H}^{u}u-u)\|_{\frac{1}{2},\Gamma}^{2}\nonumber \\
          &\qquad +\|{\cal I}_{H}^{p}p-p\|_{\frac{1}{2},\Gamma}^{2}
          +\|\dt({\cal I}_{H}^{p}p-p)\|_{\frac{1}{2},\Gamma}^{2}\Big)ds
	+ \|\ss(t)\|^{2}+\|\sp(t)\|^{2}+\|\sg(t)\|^{2} + \|\sz(t)\|^{2}\nonumber \\
	  & \qquad + \|({\cal I}_{H}^{u}u-u)(t)\|_{\frac{1}{2},\Gamma}^{2}
          +\|\dt({\cal I}_{H}^{u} u- u)(t)\|_{\frac{1}{2},\Gamma}^{2}
          +\|({\cal I}_{H}^{p}p-p)(t)\|_{\frac{1}{2},\Gamma}^{2}\nonumber \\
	  & \qquad +\|\fs(0)\|^{2}+\|\fp(0)\|^{2}+\|\phi_{z}(0)\|^{2}
          +\|\dt({\cal I}_{H}^{u} u-u)(0)\|_{\frac{1}{2},\Gamma}^{2}\bigg).
          \label{eq:err-proof-38}
\end{align}
\textbf{Bound on $\|\phi_{\lambda^u}\|_{\G}$ and $\|\phi_{\lambda^p}\|_{\G}$.}
	
In order to bound the error in $\|\lHu\|_{\G}$, we take the difference between
equations (\ref{eq:cts1-sub}) and (\ref{eq:monodd-mmmfe1}) to get
\begin{align*}
  & \inp[A\left(\left(\s-\sigma_{h}\right)+\a\left(p-p_{h}\right)I\right)]{\tau}
  +\inp[u-u_{h}]{\dvr_h{\tau}} +\inp[\gamma-\gamma_{h}]{\tau}
  =\sum_{i=1}^{N}\gnp[u-\lambda_{H}^{u}]{\t\,n_{i}}_{\Gamma_{i}}, \quad \forall\tau\in\X_{h}.
\end{align*}
We can split the error terms in the above equation and use (\ref{eq:motor-project-1})
to rewrite it as
\begin{align*}
  \sum_{i=1}^{N}\gnp[\flu]{\t\,n_{i}}_{\Gamma_{i}} & =\inp[A\left(\fs+\a\fp\right)]{\tau}+\inp[A\left(\ss+\a\sp\right)]{\tau}
  +\inp[\fu]{\dvr_h{\tau}}\nonumber \\
  & \quad +\inp[\su]{\dvr_h{\tau}} +\inp[\fg]{\tau}+\inp[\sg]{\tau}, \quad \forall\tau\in\X_{h}.
\end{align*}
The inf-sup stability bound (\ref{eq:mortar-u-inf-sup}) combined with the above equation implies 
\begin{align}
 \|\flu\|_{\G} & \le \beta_E \sup_{0\ne\tau\in\X_{h}}\frac{\sum_{i=1}^{N}\gnp[\tau n_{i}]{\flu}_{\G_{i}}}{\|\tau\|_{\X_{h}}}
 = \beta_E\sup_{0\ne\tau\in\X_{h}}\frac{1}{\|\tau\|_{\X_{h}}}\Big(\inp[A\left(\fs+\a\fp\right)]{\tau} +\inp[A\left(\ss+\a\sp\right)]{\tau} \nonumber\\
 & \qquad + \inp[\fu]{\dvr_h{\tau}} +\inp[\su]{\dvr_h{\tau}} +\inp[\fg]{\tau}+\inp[\sg]{\tau} \Big) \nonumber\\
 & 
  \le C\big(\|\phi_{\s}\|+ \|\phi_{p}\|+\|\fu\|+\|\fg\| + \|\ss\|+\|\sp\|+\|\su\|+\|\sg\|  \big).
\end{align}
To bound the error in $\|\lHp\|_{\G}$ we take the difference between (\ref{eq:cts4-sub})
and (\ref{eq:monodd-mmmfe4}) and use (\ref{eq:motor-project-2}) to obtain
\begin{align*}
  & \inp[\K\fz]{\zeta}+\inp[\K\sz]{\zeta}-\inp[\fp]{\dvr_h{\zeta}}
  - \inp[\sp]{\dvr_h{\zeta}}
  =\sum_{i=1}^{N}-\gnp[\flp]{\zeta\cdot n_{i}}_{\Gamma_{i}}, \quad \forall\zeta\in Z_{h}.
\end{align*}
The inf-sup stability bound (\ref{eq:mortar-p-inf-sup}) combined with the above equation implies 
\begin{align}
\|\flp\|_{\G}& \le \beta_D \sup_{0\ne\zeta\in Z_{h}}\frac{\sum_{i=1}^{N}\gnp[\zeta\cdot n_{i}]{\flp}_{\G_{i}}}{\|\zeta\|_{Z_{h}}}\nonumber \\
& = \beta_D\sup_{0\ne\zeta\in Z_{h}} \frac{\inp[\K\fz]{\zeta}+\inp[\K\sz]{\zeta}-\sum_{i=1}^{N}\inp[\fp]{\dvr{\zeta}}_{\Omega_{i}}-\sum_{i=1}^{N}\inp[\sp]{\dvr{\zeta}}_{\Omega_{i}}}{\|\zeta\|_{Z_{h}}}\nonumber \\
& \le C\left(\|\fz\|+\|\fp\| +\|\sz\|+\|\sp\|\right). \label{eq:err-proof-35}
\end{align}

\noindent
{\bf Bound on the initial errors.}

In order to bound the initial errors $\|\fs(0)\|,\,\|\fp(0)\|$,
and $\|\phi_{z}(0)\|$ that appear in \eqref{eq:err-proof-38}, we recall that we obtain the discrete initial data from the elliptic projection of the continuous initial data, cf.
(\ref{eq:well-pose-initial-data}). Following the arguments similar
to the ones used to arrive at (\ref{eq:stab-bnd-13}), we get 
\begin{equation}
\|\fs(0)\|+\|\fp(0)\|+\|\phi_{\g}(0)\|+\|\phi_{z}(0)\| + \|\phi_u(0)\| \le C\left(\|\ss(0)\|+\|\sp(0)\|+\|\sg(0)\|+\|\sz(0)\|+\|\su(0)\|\right).\label{eq:err-proof-41}
\end{equation}

To bound terms $\|{\cal I}_{H}^{u}v-v\|_{\frac{1}{2},\Gamma}$ and $\|{\cal I}_{H}^{p}w-w\|_{\frac{1}{2},\Gamma}$ that appear in \eqref{eq:err-proof-38}, we use (\ref{eq:inter-proj-1})--\eqref{eq:inter-proj-1a} and (\ref{eq:trace-1})
	to obtain 
\begin{align}
& \|{\cal I}_{H}^{u}v-v\|_{\frac{1}{2},\Gamma}\le CH^{\hat{m}-\half}\|v\|_{\hat{m}+\half,\Omega}, & \frac12\le\hat{m}\le m+1, \label{eq:err-proof-39}\\
& \|{\cal I}_{H}^{p}w-w\|_{\frac{1}{2},\Gamma}\le CH^{\hat{m}-\half}\|w\|_{\hat{m}+\half,\Omega}, & \frac12\le\hat{m}\le m+1. \label{eq:err-proof-40}
	\end{align}

Finally, the assertion of the theorem follows by combining bounds (\ref{eq:err-proof-38})$\--$(\ref{eq:err-proof-40}) with the approximation results (\ref{eq:inter-proj-3})$\--$(\ref{eq:inter-proj-9}),
(\ref{eq:pi-weak-5})--\eqref{eq:pi-weak-div} and (\ref{eq:pi-weak-d-4})--\eqref{eq:pi-weak-d-div}.
\end{proof}

\begin{remark}\label{rem:err-order} 	
	The above theorem implies that for sufficiently smooth solution variables,
	the error in using our method is of ${\cal O}\left(h^{k+1}+h^{l+1}+h^{j+1}+h^{r+1}+h^{s+1}+H^{m+\half}\right)$.
	Assuming we use inf-sup stable pairs of FE spaces containing polynomials
	of degree $l=j=s$, and $k=r$, and $l\le k$, we could choose $H={\cal O}\left(h^{\frac{l+1}{m+1/2}}\right)$
	to get a total error bound of order ${\cal O}\left(h^{l+1}\right)$.
	For example, for the choice of $l=0$ and $m=1$, we could choose
	$H={\cal O}\left(h^{\frac{2}{3}}\right)$ and for $l=0$ and $m=2$,
	we could choose $H={\cal O}\left(h^{\frac{2}{5}}\right)$ to obtain
	a total convergence rate of ${\cal O}(h).$ We will demonstrate the
	results for different choices of $H(h)$ in the numerical results
	section.	
\end{remark}

\section{Non-overlapping domain decomposition algorithm}\label{sec:Implementation-Mortar}

In this section, we discuss the implementation of the multiscale mortar mixed finite element method
using a non-overlapping domain decomposition method. First, we present a fully discrete version of the system (\ref{eq:monodd-mmmfe1})$\--$(\ref{eq:monodd-mmmfe7})
using backward Euler time discretization. Then we describe the reduction of the algebraic system at each time step to a mortar interface problem, which can be solved using an iterative
solver like GMRES. Finally, we discuss the use of a multiscale basis to increase the efficiency of the method. 

\subsection{Time discretization}

For time discretization, we use the backward Euler method. Let $\{t_{n}\}_{n=0}^{N_{T}}$,
$t_{n}=n\Delta t$, $\Delta t=T/N_{T}$, be a uniform partition of
$(0,T)$. We discretize a related formulation to the system
\eqref{eq:monodd-mmmfe1}--\eqref{eq:monodd-mmmfe7}, in which the constitutive elasticity
equation \eqref{eq:monodd-mmmfe1} is differentiated in time. The reason for this is that this approach results in a positive definite interface problem; details can be found in \cite{dd-biot}.
We introduce the variables $\dot{u}_h=\dt u_h$, $\dot{\g}_h=\dt\g_h$, and $\dot\lambda_{H}^{u}=\dt\lHu$ representing the time derivatives of the displacement, rotation, and displacement-Lagrange multiplier, respectively. In addition, in order to make more clear the incorporation of boundary conditions in the domain decomposition algorithm, we present the method for non-homogeneous Dirichlet boundary conditions
$$
u = g_u \ \ \mbox{on} \ \Gamma^u_D, \quad p = g_p \ \ \mbox{on} \ \Gamma^p_D.
$$
The fully discrete multiscale mortar MFE method reads as follows: for $0\le n\le N_{T}-1$ and $1\le i\le N$, find
$(\sigma_{h,i}^{n+1},\dot{u}_{h,i}^{n+1},\dot{\g}_{h,i}^{n+1},z_{h,i}^{n+1},p_{h,i}^{n+1},
\dot\lambda_{H}^{u,n+1},\lambda_{H}^{p,n+1})\in\X_{h,i}\times V_{h,i}\times\W_{h,i}\times Z_{h,i}\times W_{h,i}\times\Lambda_{H}^u\times \Lambda_H^p$ such that:
\begin{align}
& \inp[A(\sigma_{h,i}^{n+1}+\a p_{h,i}^{n+1}I)]{\tau}_{\Oi}+\Delta t\inp[\dot{u}_{h,i}^{n+1}]{\dvr{\tau}}_{\Oi}+\Delta t\inp[\dot{\gamma}_{h,i}^{n+1}]{\tau}_{\Oi}\nonumber \\
  & \quad\quad= 
  \Delta t\gnp[\lambda_{H}^{\dot{u},n+1}]{\t\,n_{i}}_{\Gamma_{i}}
+ \Delta t\gnp[\dt g_{u}^{n+1}]{\t\,n_{i}}_{\dO_{i}\cap\Gd^{u}}
  +\inp[A(\sigma_{h,i}^{n}+\a p_{h,i}^{n}I)]{\tau}_{\Oi}, && \forall\tau\in\X_{h,i},\label{eq:mortar-time-diff-elast}\\  
& \inp[\dvr{\sigma_{h,i}^{n+1}}]{v}_{\Oi}=-\inp[f^{n+1}]{v}_{\Oi}, &  & \forall v\in V_{h,i},\label{eq:mortar-dd1-mfe2-dsc}\\
& \inp[\sigma_{h,i}^{n+1}]{\xi}_{\Oi}=0, &  & \forall\xi\in\W_{h,i},\label{eq:mortar-dd1-mfe3-dsc}\\
    & \inp[\K z_{h,i}^{n+1}]{\zeta}_{\Oi}-\inp[p_{h,i}^{n+1}]{\dvr{\zeta}}_{\Oi}=
  -\gnp[\lambda_{H}^{p,n+1}]{\zeta\cdot n_i}_{\Gamma_{i}}
-\gnp[g_{p}^{n+1}]{\zeta\cdot n_i}_{\dOi\cap\Gd^{p}}, &  & \forall \zeta\in Z_{h,i},\label{eq:mortar-dd1-mfe4-dsc}\\
    & c_{0}\inp[p_{h,i}^{n+1}]{w}_{\Oi} +\a\inp[A(\sigma_{h,i}^{n+1}+\a p_{h,i}^{n+1}I)]{wI}_{\Oi}
    +\Delta t\inp[\dvr{z_{h,i}^{n+1}}]{w}_{\Oi} \nonumber \\
& \qquad
    = c_{0}\inp[p_{h,i}^{n}]{w}_{\Oi} +\a\inp[A(\sigma_{h,i}^{n}+\a p_{h,i}^{n}I)]{wI}_{\Oi}
    + \Delta t\inp[g^{n+1}]{w}_{\Oi}, &  & \forall w\in W_{h,i},\label{eq:mortar-dd1-mfe5-dsc}\\
& \sum_{i=1}^{N}\gnp[\sigma_{h,i}^{n+1}\,n_{i}]{\mu^{u}}_{\G_{i}}=0, &  & \forall\mu^{u}\in\Lambda_{H}^{u},\label{eq:mortar-dd1-mfe6-dsc}\\
& \sum_{i=1}^{N}\gnp[z_{h,i}^{n+1}\cdot n_{i}]{\mu^{p}}_{\G_{i}}=0, &  & \forall\mu^{p}\in\Lambda_{H}^{p}.\label{eq:mortar-dd1-mfe7-dsc}
\end{align}
The original variables can be recovered using\begin{equation}\label{init-recover} u_h^n = u_h^0 + \Delta t \sum_{k=1}^n \dot{u}_h^k, \quad \gamma_h^n = \gamma_h^0 + \Delta t \sum_{k=1}^n \dot{\gamma}_h^k, \quad \lambda_H^{u,n} = \lambda_H^{u,0} + \Delta t\sum_{k=1}^n \dot \lambda_H^{u,k}. \end{equation}

\subsection{Reduction to an interface problem}

We solve the system resulting from \eqref{eq:mortar-time-diff-elast}--\eqref{eq:mortar-dd1-mfe7-dsc}
at each time step by reducing it to an interface problem for the mortar variables. To simplify the notation, define
$$
\lambda_{H}=\begin{pmatrix}\lambda_{H}^{u}\\
\lambda_{H}^{p}
\end{pmatrix}, \quad
\Lambda_{H}=\begin{pmatrix}\Lambda_{H}^{u}\\
\Lambda_{H}^{p}
\end{pmatrix}.
$$
and let $\lHi$ and $\Lambda_{H,i}$ denote the restrictions of $\lH$
and $\Lambda_{H}$ to $\Gi$, respectively. We introduce two sets of complementary subdomain problems. 

The first set of problems reads as follows: for $1\le i\le N$, find
$(\sb,\ubd,\gbd,\zb,\pb)\in\X_{h,i}\times V_{h,i}\times\W_{h,i}\times Z_{h,i}\times W_{h,i}$
such that 
\begin{align}
& \inp[A(\sb+\a\pb I)]{\tau}_{\Oi}+\Delta t\inp[\ubd]{\dvr{\tau}}_{\Oi}+\Delta t\inp[\gbd]{\tau}_{\Oi}\nonumber \\
  & \quad\quad
  = \Delta t\gnp[\dt g_{u}^{n+1}]{\t\,n_{i}}_{\dO_{i}\cap\Gd^{u}} + 
  \inp[A(\sigma_{h,i}^{n}+\a p_{h,i}^{n}I)]{\tau}_{\Oi}, &  & \forall\tau\in\X_{h,i},\label{eq:mortar-bar-1-1}\\
& \inp[\dvr{\sb}]{v}_{\Oi}=-\inp[f^{n+1}]{v}_{\Oi}, &  & \forall v\in V_{h,i},\label{eq:mortar-bar-2-1}\\
& \inp[\sb]{\xi}_{\Oi}=0, &  & \forall\xi\in\W_{h,i},\label{eq:mortar-bar-3-1}\\
  & \inp[\K\zb]{\zeta}_{\Oi}-\inp[\pb]{\dvr{\zeta}}_{\Oi}= -\gnp[g_{p}^{n+1}]{\zeta\cdot n_i}_{\dOi\cap\Gd^{p}},
  &  & \forall \zeta\in Z_{h,i},\label{eq:mortar-bar-4-1}\\
& c_{0}\inp[\pb]{w}_{\Oi}+\a\inp[A(\sb+\a\pb I)]{wI}_{\Oi}+\Delta t\inp[\dvr{\zb}]{w}_{\Oi}\nonumber \\
& \quad\quad=\Delta t\inp[g^{n+1}]{w}_{\Oi}+c_{0}\inp[p_{h,i}^{n}]{w}_{\Oi}+\a\inp[A(\sigma_{h,i}^{n}+\a p_{h,i}^{n}I)]{wI}_{\Oi}, &  & \forall w\in W_{h,i}.\label{eq:mortar-bar-5-1}
\end{align}
Note that these subdomain problems have zero Dirichlet data on the
subdomain interfaces, the true source terms $f$ and $g$ and outside
boundary conditions $g_{u}$ and $g_{p}$, and previous time step data $\sigma_{h,i}^{n}$
and $p_{h,i}^{n}$.

The second set of equations reads as follows: given $\lambda_{H}\in\Lambda_{H}$,
for $1\le i\le N$, find $(\sss(\lambda_{H,i}),$ $\usd(\lambda_{H,i}),\gsd(\lambda_{H,i}),\zs(\lambda_{H,i}),\ps(\lambda_{H,i})) \in\X_{h,i}\times V_{h,i}\times\W_{h,i}\times Z_{h,i}\times W_{h,i}$
such that:
\begin{align}
& \inp[A\big(\sss(\lambda_{H,i})+\a\ps(\lambda_{H,i})I\big)]{\tau}_{\Oi}+\Delta t\inp[\usd(\lambda_{H,i})]{\dvr{\tau}}_{\Oi}\nonumber \\
& \qquad\quad+\Delta t\inp[\gsd(\lambda_{H,i})]{\tau}_{\Oi}=\Delta t\left\langle \lambda_{H,i}^{\dot{u}},\t\,n_{i}\right\rangle _{\Gamma_{i}}, & \forall\tau\in\X_{h,i},\label{eq:mortar-star-1-1}\\
& \left(\dvr\sss(\lambda_{H,i}),v\right)_{\Oi}=0, & \forall v\in V_{h,i},\label{eq:mortar-star-2-1}\\
& \left(\sss(\lambda_{H,i}),\xi\right)_{\Oi}=0, & \forall\xi\in\W_{h,i},\label{eq:mortar-star-3-1}\\
  & \left(\K\zs(\lambda_{H,i}),\zeta\right)_{\Oi}-\left(\ps(\lambda_{H,i}),\dvr{\zeta}\right)_{\Oi}
  =-\gnp[\lambda_{H,i}^{p}]{\zeta\cdot n_i}_{\Gamma_{i}}, & \forall \zeta\in Z_{h,i},\label{eq:mortar-star-4-1}\\
& c_{0}\left(\ps(\lambda_{H,i}),w\right)+\a\left(A\big(\sss(\lambda_{H,i})+\a\ps(\lambda_{H,i})I\big),wI\right)_{\Oi}\nonumber \\
& \quad\qquad+\Delta t\left(\dvr\zs(\lH)\right)=0, & \forall w\in W_{h,i}.\label{eq:mortar-star-5-1}
\end{align}
Note that these problems have $\lambda_{H,i}$ as Dirichlet
boundary data on the interfaces $\G$, zero source terms, zero
boundary data on the outside boundary $\dO$, and zero data
from the previous time step.

Define the bilinear forms $a_{H,i}^{n+1}:\Lambda_{H,i}\times\Lambda_{H,i}\to\R$, $1\le i\le N$,
$a_{H}^{n+1}:\LH\times\LH\to\R$, and the linear functional $g_{H}^{n+1}:\LH\to\R$
for all $0\le n\le N_{T}-1$ by
\begin{gather*} a_{H,i}^{n+1}(\lambda_{H,i},\mu_i)=\gnp[\sss(\lambda_{H,i})\,n_{i}]{\mu_i^{u}}_{\Gamma_{i}}-\gnp[\zs(\lambda_{H,i})\cdot n_{i}]{\mu_i^{p}}_{\Gamma_{i}},\quad a_{H}^{n+1}(\lambda_{H},\mu)=\sum_{i=1}^{N}a_{H,i}^{n+1}(\lambda_{H,i},\mu_i),\\
g_{H}^{n+1}(\mu)=\sum_{i=1}^{N}\left(-\gnp[\sb\,n_{i}]{\m^{u}}_{\Gamma_{i}}+\gnp[\zb\cdot n_{i}]{\m^{p}}_{\Gamma_{i}}\right).
\end{gather*}
It follows from (\ref{eq:mortar-dd1-mfe6-dsc})$\--$(\ref{eq:mortar-dd1-mfe7-dsc})
that the solution to the global problem
(\ref{eq:mortar-time-diff-elast})--(\ref{eq:mortar-dd1-mfe7-dsc})
is equivalent to solving the following interface problem for $\lambda_{H}^{n+1}\in\Lambda_{H}$:
\begin{equation}
a_{H}^{n+1}(\lambda_{H}^{n+1},\mu)=g_{H}^{n+1}(\m),\quad\forall\m\in\Lambda_{H},\label{eq:mortar-int-problm-1}
\end{equation}
and setting 
\begin{align*}
& \s_{h,i}^{n+1}=\sss(\lambda_{H}^{n+1})+\sb,\quad\dot{u}_{h,i}^{n+1}=\usd(\lambda_{H}^{n+1})+\ubd,\quad\dot{\g}_{h,i}^{n+1}=\gsd(\lambda_{H}^{n+1})+\gbd,\\
& z_{h,i}^{n+1}=\zs(\lambda_{H}^{n+1})+\zb,\quad p_{h,i}^{n+1}=\ps(\lambda_{H}^{n+1})+\pb.
\end{align*}

\subsection{Solution of the interface problem}

We introduce the linear operators $\AHi:\Lambda_{H,i}\to\Lambda_{H,i}'$, for $1\le i\le N$, and $\AH:\Lambda_{H}\to\Lambda_{H}'$ such that
for any $\lH\in\LH$,
$$
\langle\AHi\lHi,\,\mu_i\rangle = a_{H,i}^{n+1}(\lambda_{H,i},\mu_i) \quad \forall\mu_i\in\Lambda_{H,i}, \qquad
\langle\AH\lH,\mu\rangle = \sumsubd\langle\AHi\lHi,\mu_i\rangle \quad \forall\mu\in\Lambda_{H}.
$$
We also define the functional $\GH\in\LH'$ such that
\begin{equation*}
  \gnp[\GH]{\mu} = g_{H}^{n+1}(\mu)
  \qquad\forall\mu\in\Lambda_{H,i}.
\end{equation*}
The interface problem (\ref{eq:mortar-int-problm-1}) can now be reformulated
as finding $\lH^{n+1}\in\LH$ such that
\begin{equation}
\AH\lH^{n+1}=\GH.\label{eq:mortar-Matrix-int-problm}
\end{equation}
Consider the $L^{2}$-orthogonal projections $\mathcal{Q}_{h,i}^{u,T}:\X_{h,i}n_{i}\to\Lambda_{H}^{u}$
and $\mathcal{Q}_{h,i}^{p,T}:Z_{h,i}\cdot n_{i}:\to\Lambda_{H}^{p}$, which are the adjoint operators of $\mathcal{Q}_{h,i}^{u}$ and $\mathcal{Q}_{h,i}^{p}$, respectively, introduced in \eqref{eq:motor-project-1}--\eqref{eq:motor-project-2}.
Using this notation, we have
\begin{align}
& \AHi\lHi=\begin{pmatrix}\mathcal{Q}_{h,i}^{u,T}\sss(\lambda_{H,i})\,n_{i}\\
-\mathcal{Q}_{h,i}^{p,T}\zs(\lambda_{H,i})\cdot n_{i}
  \end{pmatrix}, \quad i=1,\ldots N.
\label{eq:Matrix-projection-def}
\end{align}

It is shown in \cite[Lemma~3.1]{dd-biot} that in the case of matching grids the interface bilinear form $a_H^{n+1}(\cdot,\cdot)$ is positive definite. The proof can be easily extended to the current setting using mortar variable. Consequently, we use GMRES to solve the interface problem (\ref{eq:mortar-Matrix-int-problm}). The action of the interface operator $\AH$ required at each GMRES iteration is computed using the steps described in Algorithm \ref{alg:Solving-interface-problem}.
\\
\\
\begin{algorithm}[h]
  \caption{\label{alg:Solving-interface-problem}
Computation of $\AH\lH$ at each GMRES iteration.}
\begin{enumerate}
\item Project the mortar data $\lH$ onto the subdomain boundary spaces: $\lambda_{H,i}^u\rightarrow\mathcal{Q}_{h,i}^u\lambda_{H,i}^u$,
$\lambda_{H,i}^p\rightarrow\mathcal{Q}_{h,i}^p\lambda_{H,i}^p$.
\item Solve the second set of subdomain problems (\ref{eq:mortar-star-1-1})$\--$(\ref{eq:mortar-star-5-1}) using the projected functions $\mathcal{Q}_{h,i}^u\lambda_{H,i}^u,\mathcal{Q}_{h,i}^p\lambda_{H,i}^p$ 
as Dirichlet boundary data on $\Gamma_i$ to obtain $\sss(\lambda_{H,i})$ and $\zs(\lambda_{H,i})$.
\item Project the subdomain solutions to the mortar space: $\sss(\lambda_{H,i})\,n_{i}\rightarrow\mathcal{Q}_{h,i}^{u,T}\sss(\lambda_{H,i})\,n_{i}$
and $\zs(\lambda_{H,i})\cdot n_{i}\longrightarrow\mathcal{Q}_{h,i}^{p,T}\zs(\lambda_{H,i})\cdot n_{i}$.
\item Compute the action $\AH\lambda_{H}$ using (\ref{eq:Matrix-projection-def}).
\end{enumerate}
\end{algorithm}

\begin{remark}\label{rem:DD-cost}
The solution algorithm for the multiscale mortar MFE method has the performance advantage over the similar method for matching grids discussed in \cite{dd-biot} that a coarse mortar mesh could be used to obtain a smaller interface problem due to the reduction in the mortar degrees of freedom. Moreover, as discussed in Theorem~\ref{thm:Error-theorem} and Remark~\ref{rem:err-order}, optimal order accuracy on the fine scale can be maintained with a suitable choice of the mortar space polynomial degree.
\end{remark}

\subsection{Implementation with multiscale stress--flux basis} \label{subsec:Implementation-of-multiscale}

As noted in Remark~\ref{rem:DD-cost}, a coarser mortar mesh can lead to a smaller interface problem, but even in that case the number of subdomain solves of the type (\ref{eq:mortar-star-1-1})$\--$(\ref{eq:mortar-star-5-1})
is directly proportional to both the number of time steps and the number of GMRES iterations at each time step. Following \cite{ganis2009implementation,eldar_elastdd}, we propose the construction
and use of a multiscale stress--flux basis (MSB), which makes the number
of subdomain solves independent of the number of GMRES iterations required
for the interface problem and the number of time steps.

Let $\left\{ \BHi\right\} _{k=0}^{N_{H,i}}$ be a basis for $\LHi$,
where $N_{H,i}$ denotes the number of degrees of freedom associated
with the finite element space $\LHi$. We calculate and store the
action of the interface operator of the form 
\begin{align}
& \mathcal{A}_{H,i}\BHi=\mathcal{Q}_{h,i}^{T}\begin{pmatrix}\s_{h,i}^{*}(\BHi)\,n_{i}\\
-z_{h,i}^{*}(\BHi)\cdot n_{i}
\end{pmatrix}, \quad k=1,\ldots N_{H,i}, \label{eq:int-op-matrix-1-generic}
\end{align}
where $\s_{h,i}^{*}(\BHi)$ and $z_{h,i}^{*}(\BHi)$ are obtained by solving (\ref{eq:mortar-star-1-1})$\--$(\ref{eq:mortar-star-5-1})
with $\BHi$ as the Dirichlet boundary data. A detailed description
of the construction of the multiscale basis elements $\left\{ \PHi\right\} _{k=0}^{N_{H}}$,
where $\PHi=\mathcal{A}_{H,i}\BHi$ is given in Algorithm \ref{alg:Construction-of-multiscale};
see \cite{ganis2009implementation,eldar_elastdd} for similar
constructions. We use the notation $Q_{h,i} = \begin{pmatrix}Q_{h,i}^u\\Q_{h,i}^p\end{pmatrix}$.

\begin{algorithm}[h]
  \caption{\label{alg:Construction-of-multiscale}Construction of a multiscale stress--flux basis
$\PHi = \mathcal{A}_{H,i}\BHi$
  }
	
\textbf{for} $k=1,\ldots,N_{H,i}$: 
\begin{enumerate}
\item Project $\BHi$ onto the subdomain boundary space: $\BHi\rightarrow\mathcal{Q}_{h,i}\BHi$.
\item Solve the system (\ref{eq:mortar-star-1-1})$\--$(\ref{eq:mortar-star-5-1})
using the projected function $\mathcal{Q}_{h,i}\BHi$
as Dirichlet boundary data, to obtain $\s_{h,i}^{*}(\BHi)$ and $z_{h,i}^{*}(\BHi)$. 
		\item Project the solution variables to the mortar space to obtain $\PHi=\begin{pmatrix}\mathcal{Q}_{h,i}^{u,T}\s_{h,i}^{*}(\BHi)\,n_{i}\\
		-\mathcal{Q}_{h,i}^{p,T}z_{h,i}^{*}(\BHi)\cdot n_{i}
		\end{pmatrix}$.
	\end{enumerate}
	\textbf{end for}
\end{algorithm}

For any $\lambda_{H,i} \in\LHi$, consider the mortar basis decomposition, $\lambda_{H,i}=\sum_{k=0}^{N_{H,i}}\lambda_{k,i}\BHi$. Using the multiscale basis, Algorithm~\ref{alg:Solving-interface-problem} for computing the action
of the interface operator $\AH\lH$ can be replaced by computing a linear combination of the multiscale basis as follows:
\begin{equation}
  \mathcal{A}_{H,i}\lambda_{H,i}
  =\sum_{k=0}^{N_{H,i}}\lambda_{k,i}\mathcal{A}_{H,i}\BHi
  =\sum_{k=0}^{N_{H,i}}\lambda_{k,i}\PHi.
  \label{eq:multi-basis-use}
\end{equation}

\begin{remark}
The multiscale stress--flux basis is computed and saved once and can be reused over
all time steps and all GMRES iterations, which gains a significant performance advantage in
the case of time-dependent problems like the one we consider. We illustrate the efficiency of using the multiscale stress--flux basis in Example 2 in the numerical
section.
\end{remark}

\section{Numerical Results\label{sec:Numerical-results-mortar}}
In this section, we report the results of several numerical tests
designed to illustrate the well-posedness, stability, and
convergence of the multiscale mortar non-overlapping domain decomposition
method for the Biot system of poroelasticity that we have developed.
We further discuss the computational efficiency of the method, including
the advantage of using a multiscale basis. The numerical schemes are implemented using the finite element package deal.II \cite{dealII90,BangerthHartmannKanschat2007}.

We use the finite element triplet $\X_{h}\times V_{h}\times\W_{h}=(\mathcal{BDM}_{1})^{2}\times (Q_{0})^{2}\times Q_{0}$
\cite{Awanou-rect-weak,ArnAwaQiu} for elasticity and the finite element pair $Z_{h}\times W_{h}=\mathcal{BDM}_{1}\times Q_{0}$ 
\cite{brezzi1991mixed} for Darcy on rectangular meshes. Here
$\mathcal{BDM}_{1}$ stands for the lowest-order Brezzi-Douglas-Marini space \cite{brezzi1991mixed} and $Q_{k}$ denotes polynomials of degree $k$ in each variable. For
the mortar spaces, $\Lambda_H^u$ is taken to be $(DQ_{m})^{2}$, and $\Lambda_H^p$
is taken to be $DQ_{m}$ with $m=1\text{ or }2$, where $DQ_{k}$ represents the discontinuous
finite element spaces containing polynomials of degree $k$, which
lives on the subdomain interfaces. The polynomial degrees of the finite element spaces used in the numerical examples are given in Table~\ref{tab:Degree-poly-numerics}. For solving
the interface problem, we use non-restarted unpreconditioned GMRES
with a tolerance $10^{-6}$ on the relative residual $\frac{r_{k}}{r_{0}}$ as
the stopping criteria.

\begin{table}[h]
	
	\caption{\label{tab:Degree-poly-numerics}Degree of polynomials associated
		with FEM spaces used for numerical experiments.}
	
\begin{centering}
\begin{tabular}{|c|c|c|c|c|c|}
\hline 
$\X_{h}:k$ & $V_{h}:l$ & $\W_{h}:j$ & $Z_{h}:r$ & $W_{h}:s$ & $\Lambda_{H}:m$
\tabularnewline
\hline 
1 & 0 & 0 & 1 & 0 & 1 or 2\tabularnewline
\hline 
\end{tabular}
\par\end{centering}
\end{table}

\begin{figure}[h]
  \begin{minipage}{0.5\textwidth}
 	\begin{centering}   
\begin{tabular}{c|c} \hline Parameter & Value
			\tabularnewline \hline Permeability tensor $(K)$ & $I$
			\tabularnewline Lame coefficient $(\mu)$ & $100.0$
			\tabularnewline Lame coefficient $(\lambda)$ & $100.0$
			\tabularnewline Mass storativity $(c_{0})$ & $1.0, 10^{-3}$
			\tabularnewline Biot-Willis constant $(\alpha)$ & $1.0$
			\tabularnewline Time step $(\Delta t)$ & $10^{-3}, 10^{-4}$
			\tabularnewline Number of time steps & $100$
			\tabularnewline 
			\hline 
		\end{tabular}          
\par\end{centering}
\end{minipage}
  \begin{minipage}{0.5\textwidth}
	\begin{centering}
	  \includegraphics[width=.8\columnwidth]{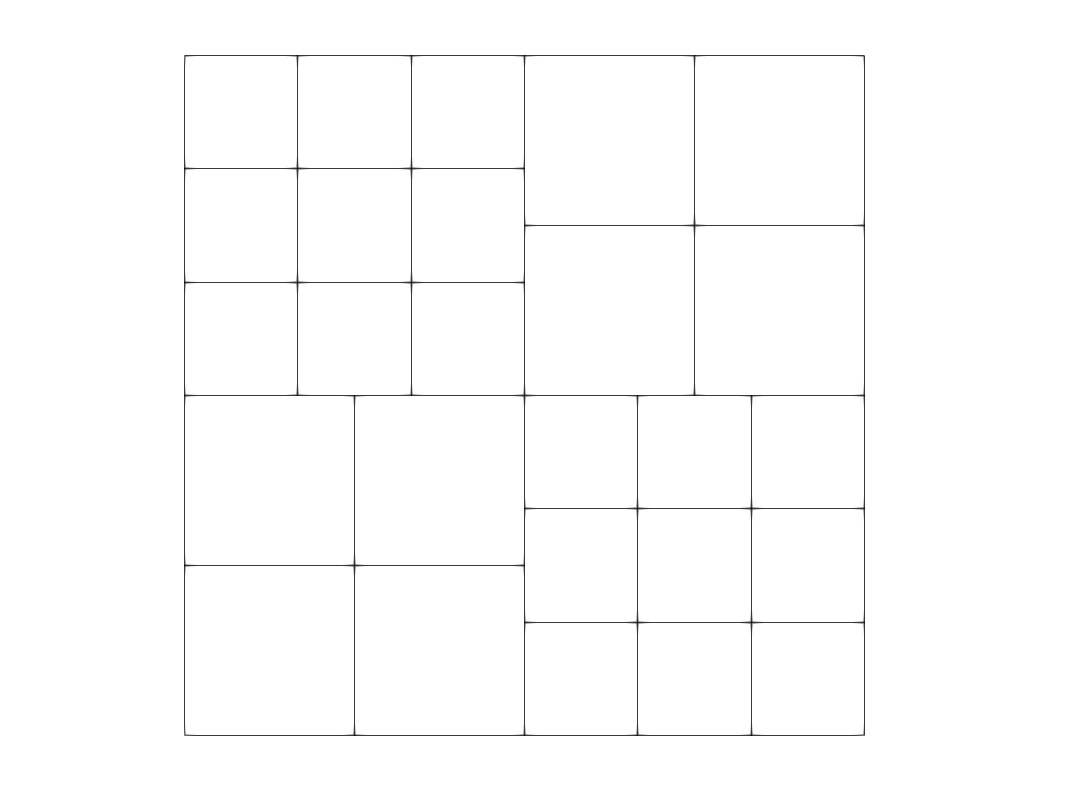}
          \end{centering}
            \end{minipage}
  \caption{\label{fig:multi-grid-mesh}Example 1, left: physical and numerical parameters;
right: coarsest non-matching subdomain grids.}	
\end{figure}

In Example 1, we test the stability, convergence, and
efficiency of the multiscale mortar MFE method using linear ($m=1$)
or quadratic ($m=2$) mortar spaces by solving the problem
with a known solution on successively refined meshes. In Example 2, we apply the multiscale mortar MFE method to solve a benchmark problem with a highly heterogeneous medium. We compare the efficiency of the multiscale versus fine scale methods and study the computational
advantage of constructing the multiscale stress--flux basis (MSB) discussed
in Section~\ref{subsec:Implementation-of-multiscale}.

\subsection{Example 1: convergence rates}

In this example we test the solvability, stability, and convergence
of the multiscale mortar MFE method. The global computational domain $\Omg$ is taken
to be the unit square $(0,1)^{2}$. We consider the following analytical
solution 
\[
p=\exp(t)(\sin(\pi x)\cos(\pi y)+10),\quad u=\exp(t)\begin{pmatrix}x^{3}y^{4}+x^{2}+\sin((1-x)(1-y))\cos(1-y)\\
(1-x)^{4}(1-y)^{3}+(1-y)^{2}+\cos(xy)\sin(x)
\end{pmatrix}.
\]
The physical and numerical parameters are given in Figure~\ref{fig:multi-grid-mesh} (left).
Using this information, we derive the right hand side and initial and boundary
conditions. We partition $\Omg$ into four square subdomains using a checkerboard 
global mesh with non-matching grids on all subdomain interfaces.
In particular, the coarsest mesh has subdomain mesh-sizes $\frac{1}{4}:\frac{1}{6}:\frac{1}{6}:\frac{1}{4}$
as shown in Figure~\ref{fig:multi-grid-mesh} (right). The corresponding coarsest
mortar interface mesh consists of two elements with mesh size $\frac{1}{2}$.

We consider two different cases, with linear $(m=1)$ or quadratic $(m=2)$ mortar
spaces. To test the convergence, we successively
refine the subdomain and mortar meshes. In the linear mortar case,
we maintain a subdomain to mortar mesh ratio $H=2h$, and
in the quadratic mortar case, we maintain the ratio $H=\sqrt{h}$.
The convergence tables for the cases with linear and quadratic mortar
spaces with $\Delta t=10^{-4}$ and $c_{0}=1.0$ are given in Tables
\ref{tab:Example-1-lin-conv} and \ref{tab:Example-1-quad-conv},
respectively. Tables~\ref{tab:Example-1-lin-conv_c0_small} and \ref{tab:Example-1-quad-conv-small-c0}
present the convergence table in the case of linear mortar and quadratic
mortar spaces, respectively with $\Delta t=10^{-4}$ and $c_{0}=10^{-3}$.
We present the number of interface iterations, relative errors, and
convergence rates. Solution plots in the case
of linear mortar with an intermediate level of refinement, $h=1/32$,
and $c_{0}=1.0$ are shown in Figure~\ref{fig:Example-1,-computed_exact_mortar}.
The plots in the case of quadratic mortar space look similar. The plots demonstrate the efficacy of the method in enforcing continuity of the solution variables across non-matching subdomain interfaces, using weakly coarse mortar spaces.

\renewcommand{\tabcolsep}{4pt}
\begin{table}[h]
	\captionsetup{justification=centering}
	\caption{\label{tab:Example-1-lin-conv}Example 1, convergence for
		linear mortar ($m=1$) with $H=2h$, $\Delta t=10^{-4}$ and $c_{0}=1.0$.}
	
	\begin{centering}
		\begin{tabular}{|c|c|c|c|c|c|c|c|c|c|c|}
			\hline 
			$h$ & \multicolumn{2}{c|}{\# GMRES} & \multicolumn{2}{c|}{$\|\sigma-\sigma_{h}\|_{L^{\infty}(L^{2})}$} & \multicolumn{2}{c|}{$\|\dvr\,(\sigma-\sigma_{h})\|_{L^{\infty}(L^{2})}$} & \multicolumn{2}{c|}{$\|\gamma-\gamma_{h}\|_{L^{\infty}(L^{2})}$} & \multicolumn{2}{c|}{$\|u-u_{h}\|_{L^{\infty}(L^{2})}$}\tabularnewline
			\hline 
			$1/4$ & 16  & rate & 1.23e-01  & rate  & 6.09e-01  & rate  & 1.39e+00  & rate  & 5.78e-01  & rate \tabularnewline
			\cline{3-3} \cline{5-5} \cline{7-7} \cline{9-9} \cline{11-11} 
			$1/8$ & 28  & -0.81  & 3.24e-02  & 1.92  & 3.11e-01  & 0.97  & 7.07e-01  & 0.97  & 2.92e-01  & 0.99 \tabularnewline
			$1/16$ & 46  & -0.72  & 8.20e-03  & 1.98  & 1.56e-01  & 0.99  & 3.55e-01  & 0.99  & 1.46e-01  & 1.00 \tabularnewline
			$1/32$ & 73  & -0.67  & 2.08e-03  & 1.98  & 7.82e-02  & 1.00  & 1.78e-01  & 1.00  & 7.31e-02  & 1.00 \tabularnewline
			$1/64$ & 122  & -0.74  & 5.39e-04  & 1.94  & 3.91e-02  & 1.00  & 8.89e-02  & 1.00  & 3.65e-02  & 1.00 \tabularnewline
			\hline 
		\end{tabular}
		\par\end{centering}
	\centering{}%
	\begin{tabular}{|c|c|c|c|c|c|c|c|c|c|c|}
		\hline 
		$h$ & \multicolumn{2}{c|}{$\|z-z_{h}\|_{L^{\infty}(L^{2})}$} & \multicolumn{2}{c|}{$\|\dvr\,(z-z_{h})\|_{L^{2}(L^{2})}$} & \multicolumn{2}{c|}{$\|p-p_{h}\|_{L^{\infty}(L^{2})}$} & \multicolumn{2}{c|}{$\|u-\lambda^{u}{}_{H}\|_{L^{\infty}(L^{2})}$} & \multicolumn{2}{c|}{$\|p-\lambda^{p}{}_{H}\|_{L^{\infty}(L^{2})}$}\tabularnewline
		\hline 
		$1/4$ & 1.04e+00  & rate  & 4.15e-01  & rate  & 5.91e-02  & rate  & 7.50e-01  & rate  & 2.06e-01  & rate \tabularnewline
		\cline{3-3} \cline{5-5} \cline{7-7} \cline{9-9} \cline{11-11} 
		$1/8$ & 3.72e-01  & 1.48  & 1.89e-01  & 1.14  & 2.96e-02  & 1.00  & 1.90e-01  & 1.98  & 5.30e-02  & 1.96\tabularnewline
		$1/16$ & 1.19e-01  & 1.64  & 8.50e-02  & 1.15  & 1.48e-02  & 1.00  & 4.76e-02  & 1.99  & 1.33e-02  & 2.00\tabularnewline
		$1/32$ & 3.56e-02  & 1.74  & 3.97e-02  & 1.10  & 7.39e-03  & 1.00  & 1.19e-02  & 2.00  & 3.33e-03  & 2.00\tabularnewline
		$1/64$ & 1.08e-02  & 1.72  & 1.92e-02  & 1.05  & 3.70e-03  & 1.00  & 3.04e-03  & 1.97  & 8.37e-04  & 1.99\tabularnewline
		\hline 
	\end{tabular}
\end{table}

\begin{table}[h]
	\captionsetup{justification=centering}
	\caption{\label{tab:Example-1-quad-conv}Example 1, convergence for
		quadratic mortar ($m=2$) with $H=\sqrt{h}$, $\Delta t=10^{-4}$
		and $c_{0}=1.0$.}
	
	\begin{centering}
		\begin{tabular}{|c|c|c|c|c|c|c|c|c|c|c|}
			\hline 
			$h$ & \multicolumn{2}{c|}{\# GMRES} & \multicolumn{2}{c|}{$\|\sigma-\sigma_{h}\|_{L^{\infty}(L^{2})}$} & \multicolumn{2}{c|}{$\|\dvr\,(\sigma-\sigma_{h})\|_{L^{\infty}(L^{2})}$} & \multicolumn{2}{c|}{$\|\gamma-\gamma_{h}\|_{L^{\infty}(L^{2})}$} & \multicolumn{2}{c|}{$\|u-u_{h}\|_{L^{\infty}(L^{2})}$}\tabularnewline
			\hline 
			$1/4$ & 22  & rate & 1.26e-01  & rate & 6.09e-01  & rate & 1.39e+00  & rate & 5.79e-01  & rate\tabularnewline
			\cline{3-3} \cline{5-5} \cline{7-7} \cline{11-11} 
			$1/16$ & 40  & -0.43  & 8.25e-03  & 1.97 & 1.56e-01  & 0.98 & 3.55e-01  & 0.99  & 1.46e-01  & 0.99 \tabularnewline
			$1/64$ & 65  & -0.35  & 5.62e-04  & 1.93 & 3.91e-02  & 1.00  & 8.89e-02  & 1.00  & 3.65e-02  & 1.00 \tabularnewline
			\hline 
		\end{tabular}
		\par\end{centering}
	\centering{}%
	\begin{tabular}{|c|c|c|c|c|c|c|c|c|c|c|}
		\hline 
		$h$ & \multicolumn{2}{c|}{$\|z-z_{h}\|_{L^{\infty}(L^{2})}$} & \multicolumn{2}{c|}{$\|\dvr\,(z-z_{h})\|_{L^{2}(L^{2})}$} & \multicolumn{2}{c|}{$\|p-p_{h}\|_{L^{\infty}(L^{2})}$} & \multicolumn{2}{c|}{$\|u-\lambda^{u}{}_{H}\|_{L^{\infty}(L^{2})}$} & \multicolumn{2}{c|}{$\|p-\lambda^{p}{}_{H}\|_{L^{\infty}(L^{2})}$}\tabularnewline
		\hline 
		$1/4$ & 6.72e-01  & rate & 3.92e-01  & rate & 5.92e-02  & rate & 7.55e-01  & rate & 9.70e-02  & rate\tabularnewline
		\cline{3-3} \cline{5-5} \cline{7-7} \cline{9-9} \cline{11-11} 
		$1/16$ & 8.20e-02  & 1.52  & 8.36e-02  & 1.11 & 1.48e-02  & 1.00  & 4.82e-02  & 1.99  & 6.83e-03  & 1.91\tabularnewline
		$1/64$ & 7.03e-03  & 1.77  & 1.92e-02  & 1.06 & 3.70e-03  & 1.00  & 3.31e-03  & 1.93 & 5.91e-04  & 1.77\tabularnewline
		\hline 
	\end{tabular}
\end{table}

\begin{table}[h]
	\captionsetup{justification=centering}
	\caption{\label{tab:Example-1-lin-conv_c0_small}Example 1, convergence
		for linear mortar with $H=2h$, $\Delta t=10^{-4}$ and $c_{0}=10^{-3}$.}
	
	\begin{centering}
		\begin{tabular}{|c|c|c|c|c|c|c|c|c|c|c|}
			\hline 
			$h$ & \multicolumn{2}{c|}{\# GMRES} & \multicolumn{2}{c|}{$\|\sigma-\sigma_{h}\|_{L^{\infty}(L^{2})}$} & \multicolumn{2}{c|}{$\|\dvr\,(\sigma-\sigma_{h})\|_{L^{\infty}(L^{2})}$} & \multicolumn{2}{c|}{$\|\gamma-\gamma_{h}\|_{L^{\infty}(L^{2})}$} & \multicolumn{2}{c|}{$\|u-u_{h}\|_{L^{\infty}(L^{2})}$}\tabularnewline
			\hline 
			$1/4$ & 16  & rate & 1.25e-01  & rate & 6.09e-01  & rate & 1.39e+00  & rate & 5.78e-01  & rate\tabularnewline
			\cline{3-3} \cline{5-5} \cline{7-7} \cline{9-9} \cline{11-11} 
			$1/8$ & 29  & -0.86  & 3.30e-02  & 1.92  & 3.11e-01  & 0.97  & 7.07e-01  & 0.97  & 2.92e-01  & 0.99 \tabularnewline
			$1/16$ & 50  & -0.79  & 8.34e-03  & 1.98  & 1.56e-01  & 0.99  & 3.55e-01  & 0.99  & 1.46e-01  & 1.00 \tabularnewline
			$1/32$ & 87  & -0.80  & 2.09e-03  & 1.99  & 7.82e-02  & 1.00  & 1.78e-01  & 1.00  & 7.31e-02  & 1.00 \tabularnewline
			$1/64$ & 157  & -0.85  & 5.38e-04  & 1.96  & 3.91e-02  & 1.00  & 8.89e-02  & 1.00  & 3.65e-02  & 1.00 \tabularnewline
			\hline 
		\end{tabular}
		\par\end{centering}
	\centering{}%
	\begin{tabular}{|c|c|c|c|c|c|c|c|c|c|c|}
		\hline 
		$h$ & \multicolumn{2}{c|}{$\|z-z_{h}\|_{L^{\infty}(L^{2})}$} & \multicolumn{2}{c|}{$\|\dvr\,(z-z_{h})\|_{L^{2}(L^{2})}$} & \multicolumn{2}{c|}{$\|p-p_{h}\|_{L^{\infty}(L^{2})}$} & \multicolumn{2}{c|}{$\|u-\lambda^{u}{}_{H}\|_{L^{\infty}(L^{2})}$} & \multicolumn{2}{c|}{$\|p-\lambda^{p}{}_{H}\|_{L^{\infty}(L^{2})}$}\tabularnewline
		\hline 
		$1/4$ & 4.18e+01  & rate & 2.31e+00  & rate & 8.81e-01  & rate & 7.52e-01  & rate & 8.48e+00  & rate\tabularnewline
		\cline{3-3} \cline{5-5} \cline{7-7} \cline{9-9} \cline{11-11} 
		$1/8$ & 9.68e+00  & 2.11  & 7.14e-01  & 1.69  & 2.33e-01  & 1.92  & 1.90e-01  & 1.98  & 2.11e+00  & 2.00\tabularnewline
		$1/16$ & 2.31e+00  & 2.07  & 2.00e-01  & 1.84  & 5.93e-02  & 1.98  & 4.77e-02  & 1.99  & 5.08e-01  & 2.06\tabularnewline
		$1/32$ & 5.68e-01  & 2.02  & 6.02e-02  & 1.73  & 1.62e-02  & 1.87  & 1.19e-02  & 2.00  & 1.25e-01  & 2.02\tabularnewline
		$1/64$ & 1.42e-01  & 2.00  & 2.22e-02  & 1.44  & 5.22e-03  & 1.64  & 2.98e-03  & 2.00  & 3.12e-02  & 2.00\tabularnewline
		\hline 
	\end{tabular}
\end{table}

\begin{table}[h]
	\captionsetup{justification=centering}
	\caption{\label{tab:Example-1-quad-conv-small-c0}Example 1, convergence
		for quadratic mortar with $H=\sqrt{h}$, $\Delta t=10^{-4}$ and
		$c_{0}=10^{-3}$.}
	
	\begin{centering}
		\begin{tabular}{|c|c|c|c|c|c|c|c|c|c|c|}
			\hline 
			$h$ & \multicolumn{2}{c|}{\# GMRES} & \multicolumn{2}{c|}{$\|\sigma-\sigma_{h}\|_{L^{\infty}(L^{2})}$} & \multicolumn{2}{c|}{$\|\dvr\,(\sigma-\sigma_{h})\|_{L^{\infty}(L^{2})}$} & \multicolumn{2}{c|}{$\|\gamma-\gamma_{h}\|_{L^{\infty}(L^{2})}$} & \multicolumn{2}{c|}{$\|u-u_{h}\|_{L^{\infty}(L^{2})}$}\tabularnewline
			\hline 
			$1/4$ & 23  & rate & 1.28e-01  & rate  & 6.09e-01  & rate  & 1.39e+00  & rate  & 5.79e-01  & rate \tabularnewline
			\cline{3-3} \cline{5-5} \cline{7-7} \cline{11-11} 
			$1/16$ & 41  & -0.41  & 8.39e-03  & 1.97  & 1.56e-01  & 0.98 & 3.55e-01  & 0.96 & 1.46e-01  & 0.99\tabularnewline
			$1/64$ & 72  & -0.41  & 5.61e-04  & 1.95  & 3.91e-02  & 1.00  & 8.89e-02  & 1.00  & 3.65e-02  & 1.00 \tabularnewline
			\hline 
		\end{tabular}
		\par\end{centering}
	\centering{}%
	\begin{tabular}{|c|c|c|c|c|c|c|c|c|c|c|}
		\hline 
		$h$ & \multicolumn{2}{c|}{$\|z-z_{h}\|_{L^{\infty}(L^{2})}$} & \multicolumn{2}{c|}{$\|\dvr\,(z-z_{h})\|_{L^{2}(L^{2})}$} & \multicolumn{2}{c|}{$\|p-p_{h}\|_{L^{\infty}(L^{2})}$} & \multicolumn{2}{c|}{$\|u-\lambda^{u}{}_{H}\|_{L^{\infty}(L^{2})}$} & \multicolumn{2}{c|}{$\|p-\lambda^{p}{}_{H}\|_{L^{\infty}(L^{2})}$}\tabularnewline
		\hline 
		$1/4$ & 4.24e+01  & rate & 2.42e+00  & rate  & 9.97e-01  & rate  & 7.57e-01  & rate  & 1.07e+01  & rate\tabularnewline
		\cline{3-3} \cline{5-5} \cline{7-7} \cline{9-9} \cline{11-11} 
		$1/16$ & 2.33e+00  & 2.01  & 2.01e-01  & 1.79 & 6.01e-02  & 2.06 & 4.83e-02  & 1.98 & 5.17e-01  & 2.19\tabularnewline
		$1/64$ & 1.50e-01  & 1.97  & 2.25e-02  & 1.58 & 5.40e-03  & 1.74 & 3.26e-03  & 1.95 & 3.38e-02  & 1.97\tabularnewline
		\hline 
	\end{tabular}
\end{table}

\begin{figure}[h]
  \includegraphics[width=0.33\columnwidth]{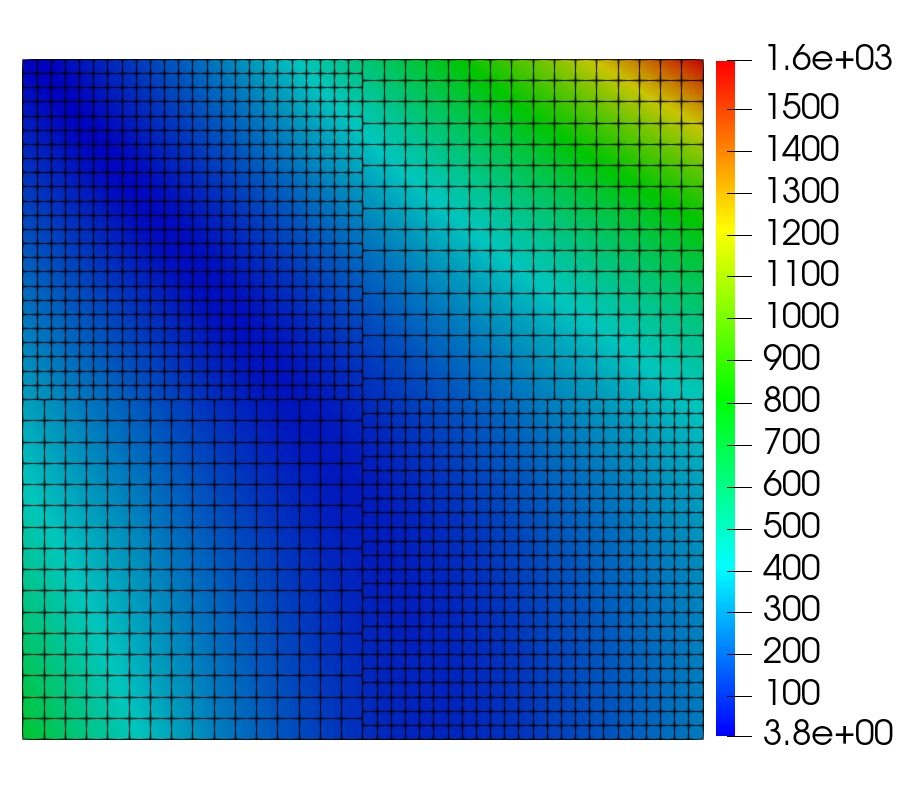}
  \hspace{-.3cm}
	\includegraphics[width=0.33\columnwidth]{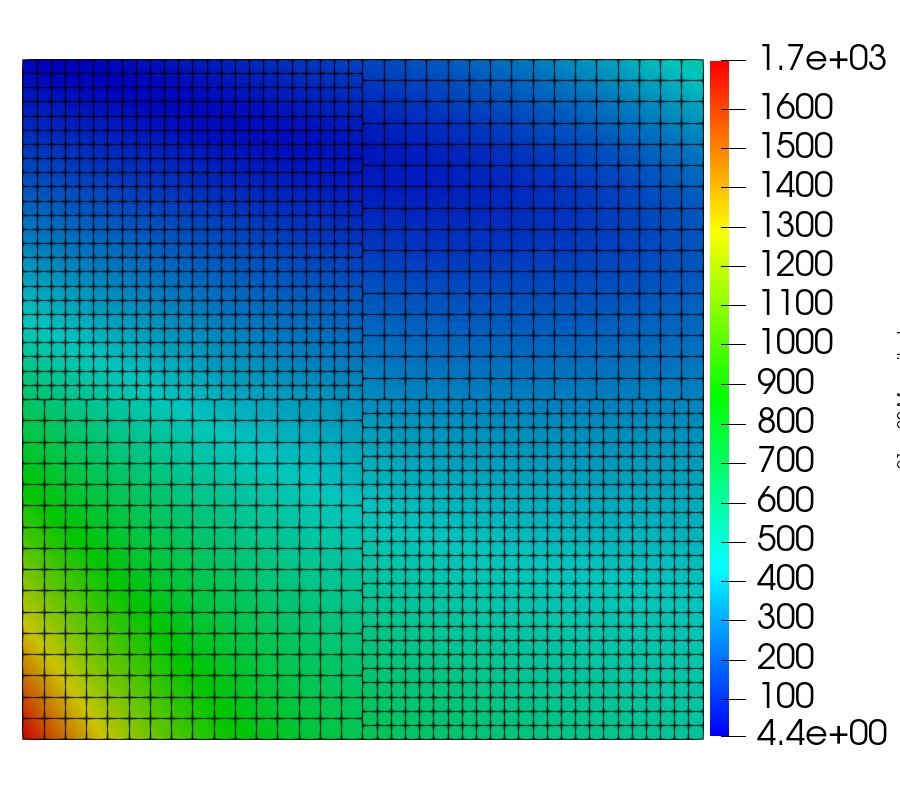}
        \hspace{-.3cm}
	\includegraphics[width=0.33\columnwidth]{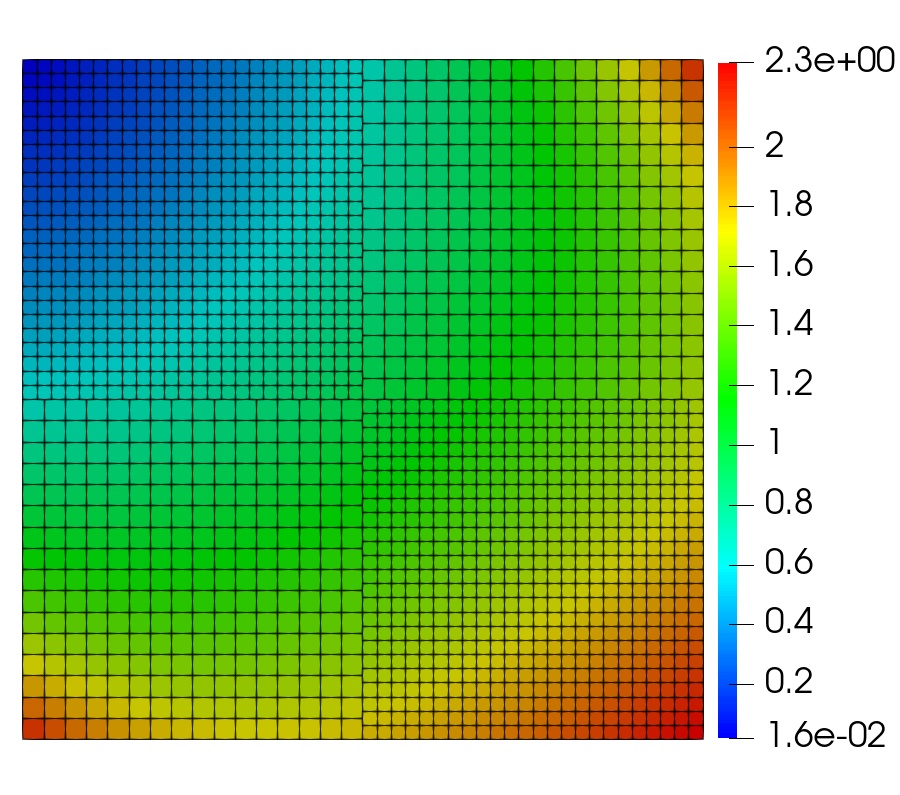}
	
	\includegraphics[width=0.33\columnwidth]{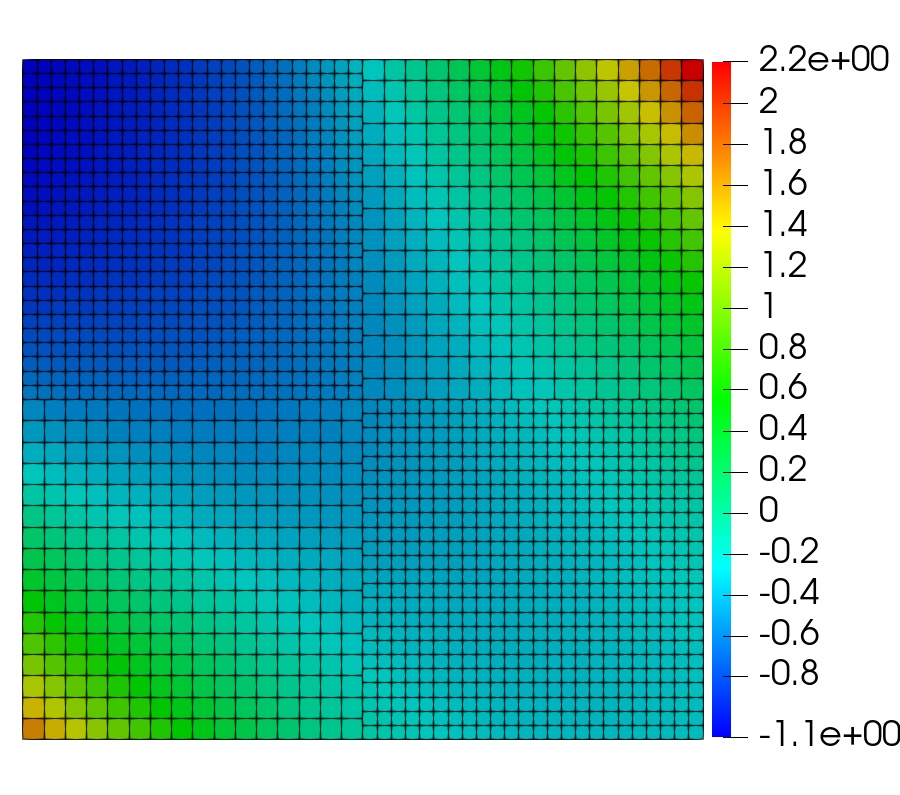}
        \hspace{-.3cm}
	\includegraphics[width=0.33\columnwidth]{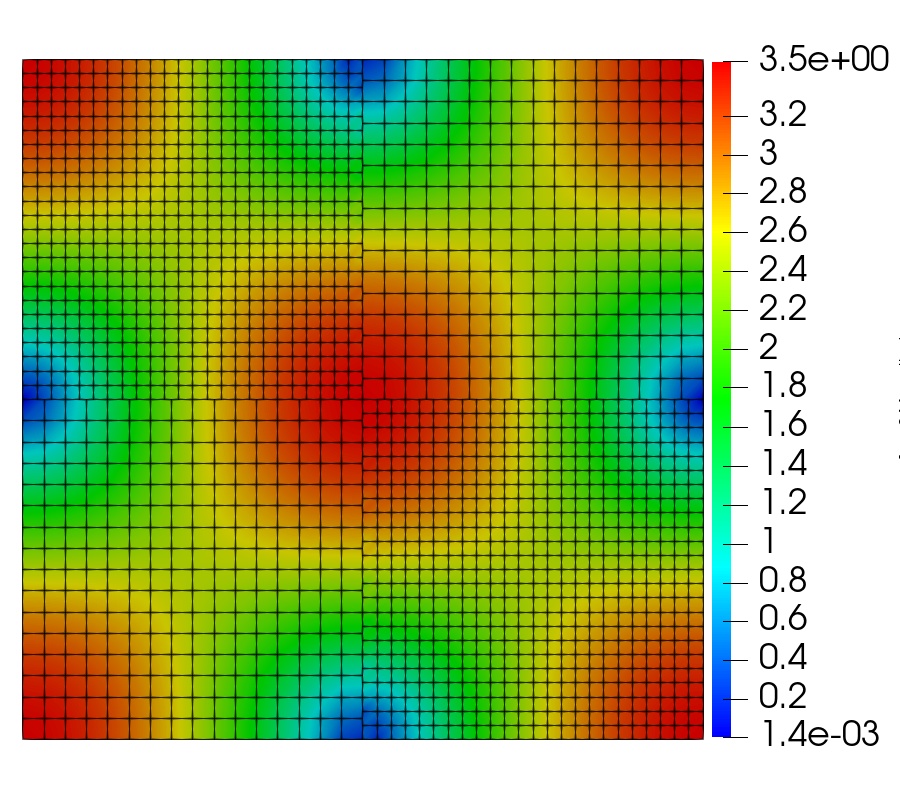}
        \hspace{-.3cm}
	\includegraphics[width=0.33\columnwidth]{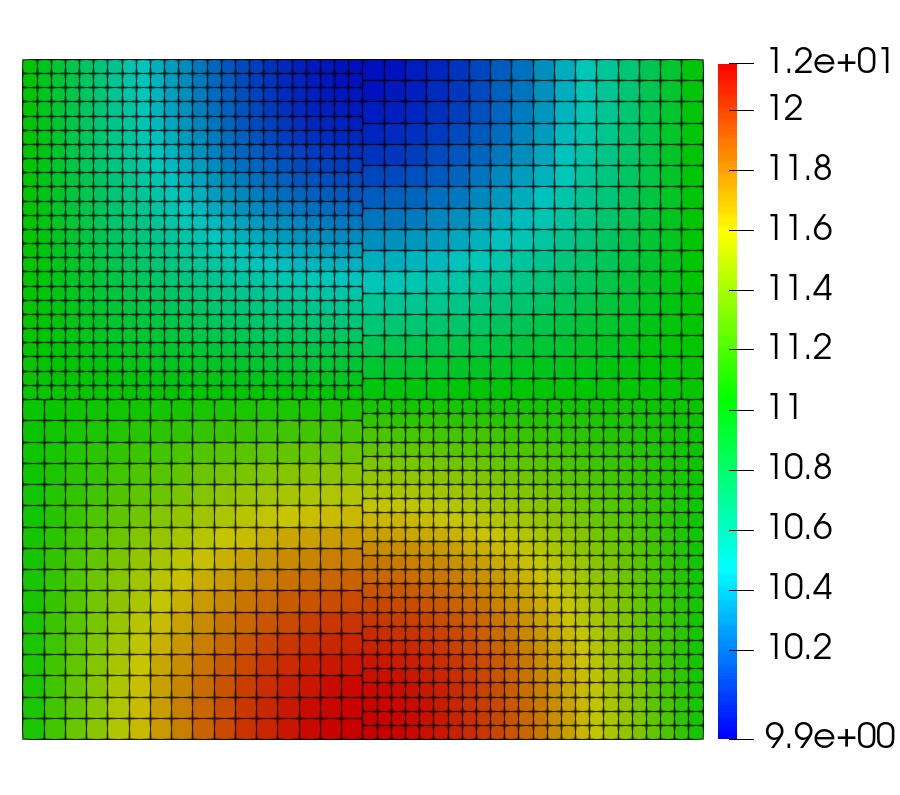}
	\caption{\label{fig:Example-1,-computed_exact_mortar}Example 1, computed solution
		at the final time step using a linear mortar on non-matching subdomain
		grids, $h=1/32$, $\Delta t=10^{-3}$ and $c_{0}=1.0$;
                top: $x$-stress (left), $y$-stress (middle), displacement (right); bottom: rotation (left), velocity (middle), pressure (right).}
\end{figure}

The numerical results observed in the tables are consistent with the theoretical
results from the previous sections. In particular,
we demonstrate the stability of the method over a $100$ time steps,
and Tables \ref{tab:Example-1-lin-conv} and \ref{tab:Example-1-quad-conv}
show convergence rates that follow from Theorem \ref{thm:Error-theorem}
and Table \ref{tab:Degree-poly-numerics}. With linear mortar ($m=1$) and $H=2h$, the interface error is ${\cal O}(h^{\frac{3}{2}})$. With quadratic mortar ($m=2$) and $H=\sqrt{h}$, the interface error is ${\cal O}(h^{\frac{5}{4}})$. In both the cases, it is dominated by the subdomain error, which is ${\cal O}(h)$. As a result, we expect at least ${\cal O}(h)$ convergence
in both cases, which is what we observe. The observed convergence rate for the errors $\|\sigma-\sigma_{h}\|_{L^{\infty}(L^{2})}$ and $\|z-z_{h}\|_{L^{\infty}(L^{2})}$ are close to ${\cal O}(h^2)$, which suggests that it may be possible to establish stress and velocity estimates that are independent of the approximation of the other variables. 

Comparison of the number
of interface iterations required in the case of linear and quadratic
mortars in Tables \ref{tab:Example-1-lin-conv} and \ref{tab:Example-1-quad-conv},
respectively shows that both mortar degrees result in similar accuracy
for the same level of subdomain mesh refinement. At the same time,
the quadratic mortar case requires fewer interface
iterations compared to the linear mortar case with the same level
of subdomain mesh refinement. This is due to the choice of a coarser
mortar mesh in the case of quadratic mortar case. This points towards
a way to decrease the number of interface iterations by using a coarser
mesh and higher mortar space degree, without a loss in accuracy.
Tables \ref{tab:Example-1-lin-conv_c0_small}$\--$\ref{tab:Example-1-quad-conv-small-c0}
indicate that the stability and convergence rates are not affected by smaller values of $c_{0}$, which is consistent with the theoretical bounds established in the previous sections.

\subsection{Example 2: heterogeneous medium}

In this example, we demonstrate the performance of the multiscale mortar method in a
practical application with highly heterogeneous medium. First, we
compare the accuracy and efficiency of the multiscale method with $H>h$
to a fine scale method with $H=h$. We then study the computational advantage
of using a multiscale stress--flux basis. We use
the porosity and the permeability data from the Society of Petroleum Engineers
10th Comparative Solution Project (SPE10)\footnote{https://www.spe.org/web/csp/datasets/set02.htm}.
The data are given on a $60\times220$ grid covering the rectangular region $(0,60)\times(0,220)$. We decompose the global domain into $3\times5$ subdomains consisting of identical rectangular blocks. The Young's modulus is obtained from the porosity field data using the relation
$E=10^{2}\left(1-\frac{\phi}{c}\right)^{2.1}$, where $c=0.5$ refers
to the porosity at which the Young's modulus vanishes, see \cite{kovavcik1999correlation}
for details. The input fields are presented in Figure~\ref{fig:Example-2,-input_fields_biot_mortar}.
The problem parameters and boundary conditions are given in Table~\ref{tab:Physical-Parameters_ex2_mortar} and the source terms are set to zero. These conditions describe a flow from left to right, driven by the gradient in the pressure. We use a compatible initial condition for pressure, $p_{0}=1-x$. To obtain 
discrete initial data, we set $p_{h}^{0}$ to be the $L^{2}$-projection of $p_{0}$ onto
$W_{h}$ and solve a mixed elasticity domain decomposition problem
at $t=0$ to obtain $\sigma_{h}^{0}$, $u_{h}^{0}$,
$\gamma_{h}^{0}$, and $\lambda_{H}^{u,0}$. 

We use a global $60\times220$ grid and solve the problem using both
fine scale ($H=h)$ and coarse $(H>h)$ mortar spaces. For the coarse
mortar case, we use both linear $(m=1)$ and quadratic $(m=2)$ mortars
with one and two mortar elements per subdomain interface. The comparison of
the computed solution using different choices of mortars is given
in Figures~\ref{fig:Example-2,-solution_plots-1}$\--$\ref{fig:Example-2,-solution_plots_last}.
The solution variables are very similar for all five cases, illustrating that the multiscale mortar method obtains comparable accuracy to the fine scale discretization, even in the case of the
coarsest mortar grid with one linear mortar per interface. On the other hand, the computational cost of the multiscale mortar method is smaller than the fine scale method. This is evident from Table \ref{tab:Example-2,-=000023GMRES_table}, where the number of GMRES iterations and subdomain solves are reported, noting that the number of subdomain solves dominates the computational complexity of the method. We observe that the multiscale mortar method requires much fewer number of GMRES iterations and solves compared to fine scale method. As a result, the multiscale mortar method obtains comparable accuracy to the fine scale method at a significantly reduced computational cost.

We further test the effect of using a multiscale stress--flux basis (MSB) on the computational cost of the method. If MSB is not used, the number of subdomain solves equals the total \#GMRES iterations across all time steps + $2\times$number of time steps,
where the last term comes from two extra solves required to solve the
system (\ref{eq:mortar-bar-1-1})$\--$(\ref{eq:mortar-star-5-1}) initially
and recovering the final solution after the interface GMRES converges. On the other hand, in the case of using MSB, total number of subdomain solves equals
the $\text{dim}(\Lambda_{H,i})$+$2\times$number of time steps. Note that the first term in the
cost in the case of no-MSB is directly proportional to the number of GMRES iterations and time steps, while the corresponding term in the case of MSB method is independent of both the number of GMRES iterations and the number of time steps, since the MSB can be reused over all time steps. Therefore the computational efficiency of the MSB for steady-state problems due to the independence of the global number of mortar degrees of freedom is further magnified by the number of time steps in the case of time-dependent problems.
These conclusions are illustrated in the last two columns of Table~\ref{tab:Example-2,-=000023GMRES_table}, where we observe that the number of solves in the case of no MSB is at least an order of magnitude larger than the MSB case. We conclude that 
the construction of MSB is an excellent tool to make the multiscale mortar method even more efficient than it already is compared to the fine scale method discussed in \cite{dd-biot}.

\begin{table}[h]
	\centering{}
	\caption{Example 2, parameters (left) and boundary conditions (right).\label{tab:Physical-Parameters_ex2_mortar}} \begin{tabular}{c|c}
		Parameter & Value\tabularnewline
		\hline 
		Mass storativity $(c_{0})$ & $1.0$\tabularnewline 
		Biot-Willis constant $(\alpha)$ & $1.0$\tabularnewline
		Time step $(\Delta t)$ & $10^{-3}$\tabularnewline 
		Total time $(T)$ & $0.1$\tabularnewline
	\end{tabular} %
	\hspace{5 mm} 
	\begin{tabular}{c|c|c|c|c}
		Boundary & $\sigma$ & $u$ & $p$ & $z$\tabularnewline 
		\hline
		Left & $\sigma n=-\alpha p n$ & - & $1$ & -\tabularnewline
		Bottom & $\sigma n=0$ & - & - & $z\cdot n=0$\tabularnewline
		Right & $-$ & $0$ & $0$ & -\tabularnewline
		Top & $\sigma n=0$ & - & - & $z\cdot n=0$\tabularnewline
	\end{tabular}
\end{table}

\begin{table}[h]
	\caption{Example 2, \#GMRES iterations and maximum number of subdomain solves.\label{tab:Example-2,-=000023GMRES_table}}
	
	\centering{}%
	\begin{tabular}{|c|c|c|c|c|}
		\hline 
		mortar & Average \#GMRES & Total \#GMRES & \multicolumn{2}{c|}{Total \#Solves}\tabularnewline
		\hline 
		&  &  & No MSB & MSB\tabularnewline
		\cline{4-5} \cline{5-5} 
		linear fine scale & 343 & 34375 & 34575 & 968\tabularnewline
		1 linear per interface & 41 & 4149 & 4349 & 224\tabularnewline
		1 quadratic per interface & 61 & 6184 & 6384 & 236\tabularnewline
		2 linear per interface & 80 & 8010 & 8210 & 248\tabularnewline
		2 quadratic per interface & 123 & 12302 & 12502 & 272\tabularnewline
		\hline 
	\end{tabular}
\end{table}

\begin{figure}[h]
  \centering{}
  \includegraphics[width=0.31\columnwidth]{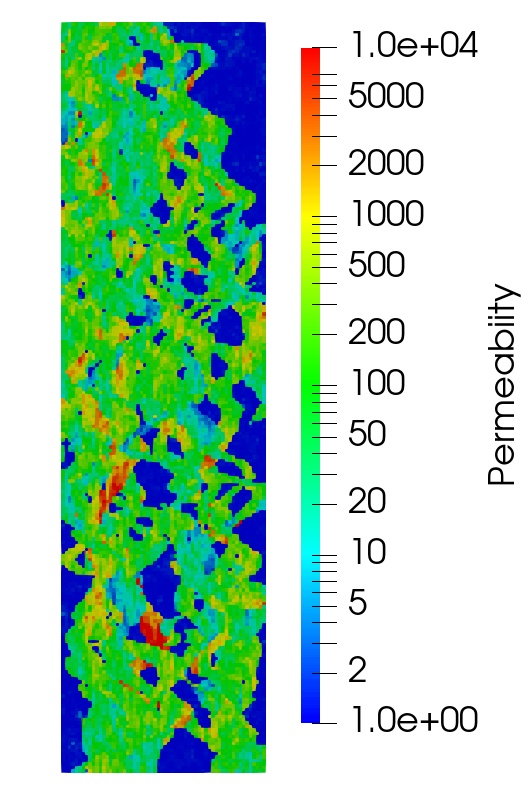}
  \includegraphics[width=0.31\columnwidth]{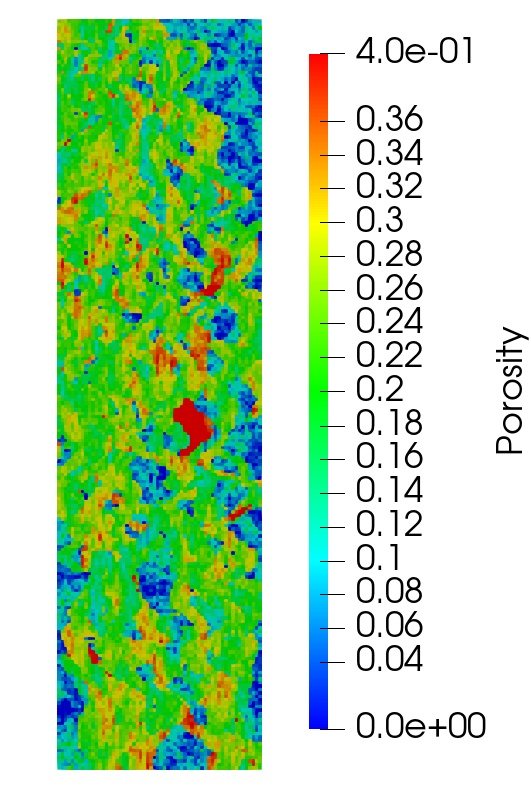}
  \includegraphics[width=0.31\columnwidth]{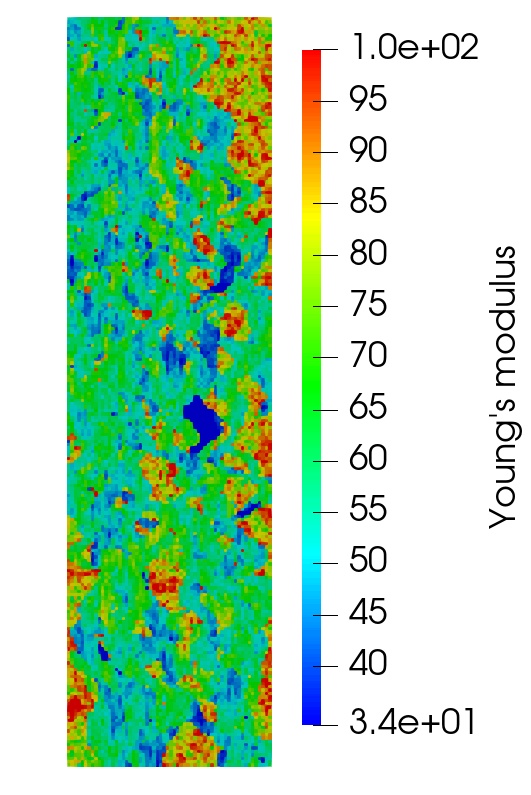}
  \caption{Example 2, permeability, porosity, and Young's modulus. \label{fig:Example-2,-input_fields_biot_mortar}}
\end{figure}

\begin{figure}[h]
  \centering{}
  \includegraphics[width=0.16\columnwidth]{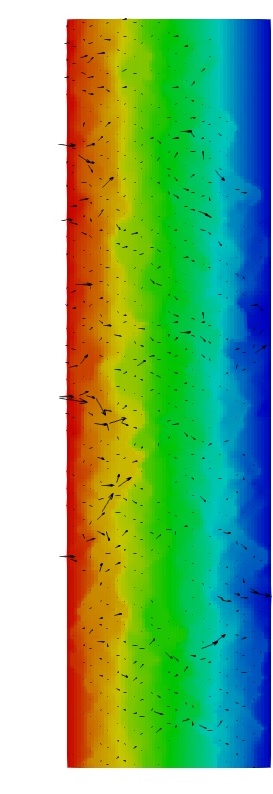}
  \includegraphics[width=0.16\columnwidth]{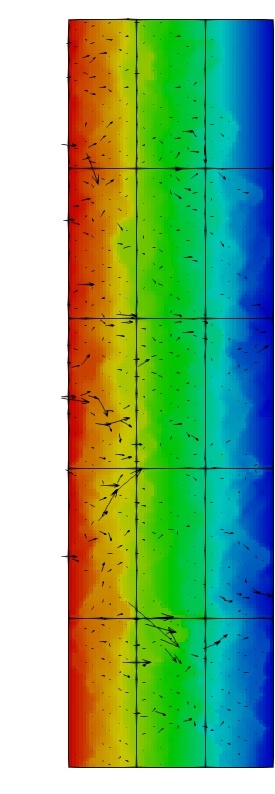}
  \includegraphics[width=0.16\columnwidth]{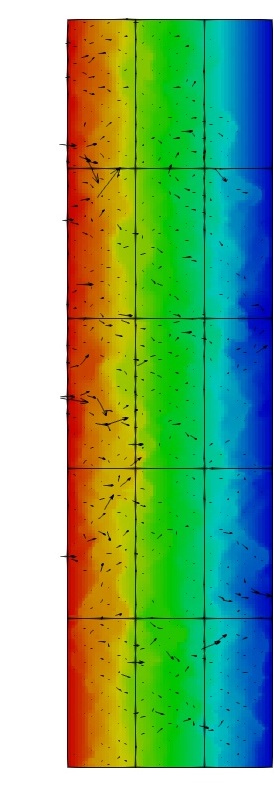}
  \includegraphics[width=0.16\columnwidth]{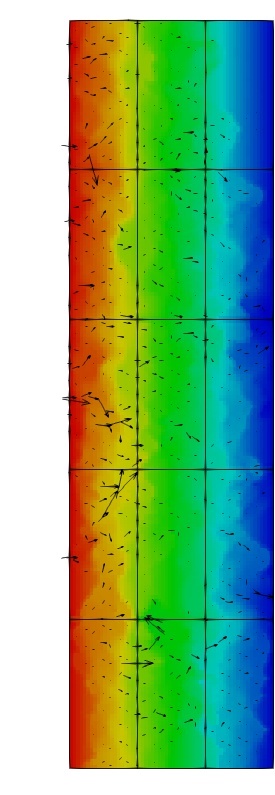}
  \includegraphics[width=0.16\columnwidth]{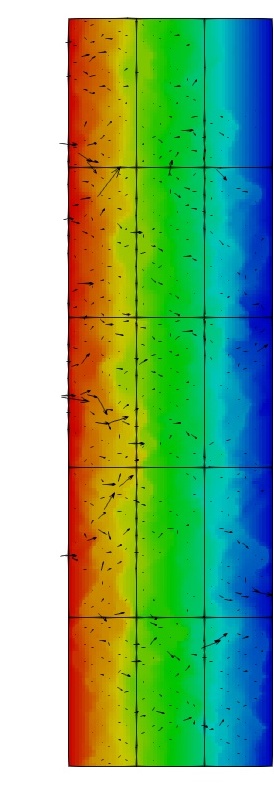}
  \includegraphics[width=0.15\columnwidth]{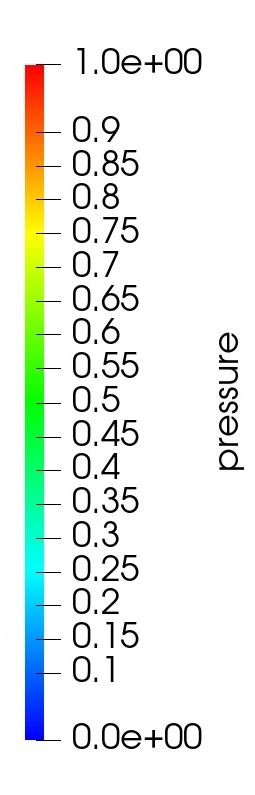}
  \caption{Example 2, pressure (color) and velocity (arrows); from left to right: fine scale, single linear mortar per interface, two linear
    mortars per interface, single quadratic mortar per interface (left), two quadratic mortars per interface.}
\label{fig:Example-2,-solution_plots-1}
\end{figure}

\begin{figure}[h]
    \centering{}
  \includegraphics[width=0.16\columnwidth]{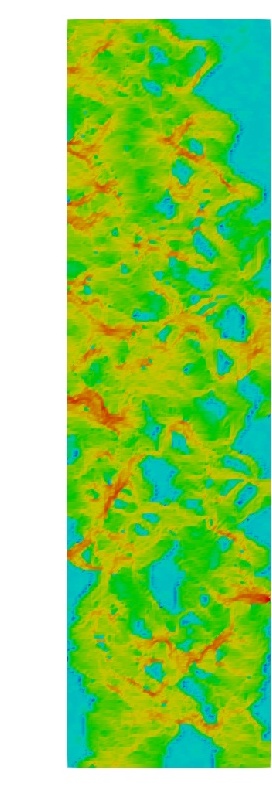}
  \includegraphics[width=0.16\columnwidth]{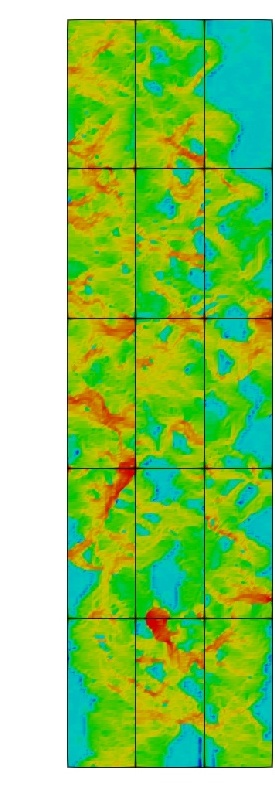}
  \includegraphics[width=0.16\columnwidth]{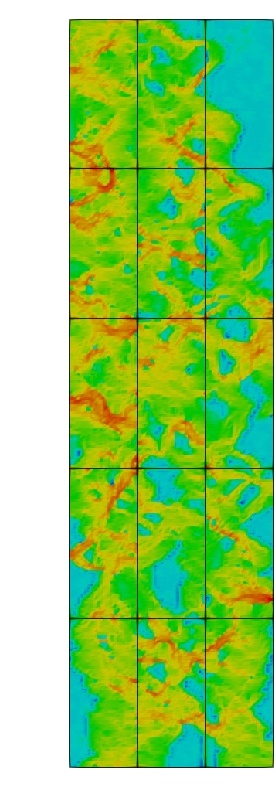}
  \includegraphics[width=0.16\columnwidth]{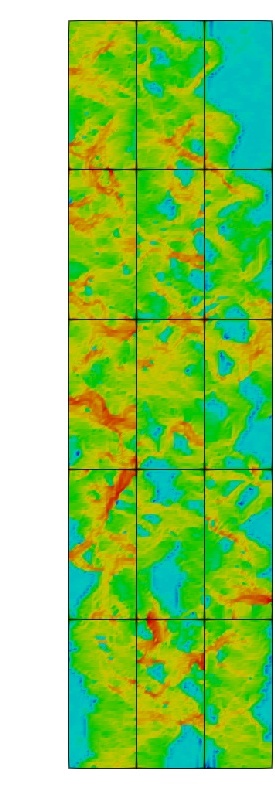}
  \includegraphics[width=0.16\columnwidth]{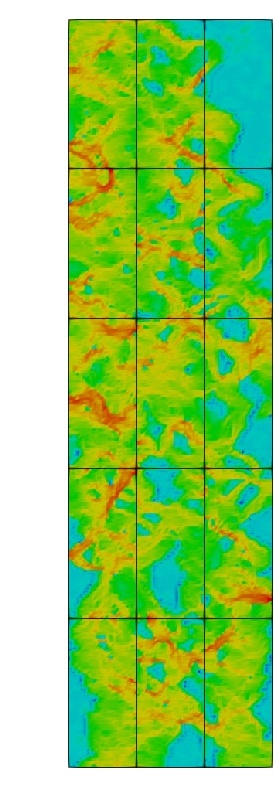}
  \includegraphics[width=0.15\columnwidth]{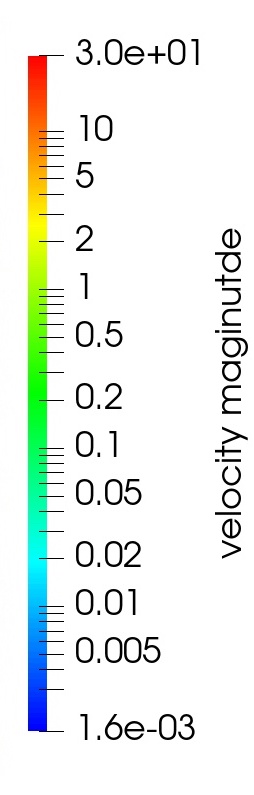}
  \caption{Example 2, velocity magnitude; from left to right: fine scale, single linear mortar per interface, two linear
mortars per interface, single quadratic mortar per interface (left), two quadratic mortars per interface.}
\end{figure}

\begin{figure}[h]
    \centering{}
  \includegraphics[width=0.16\columnwidth]{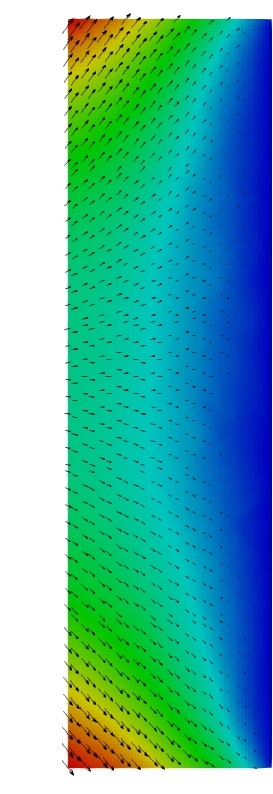}
  \includegraphics[width=0.16\columnwidth]{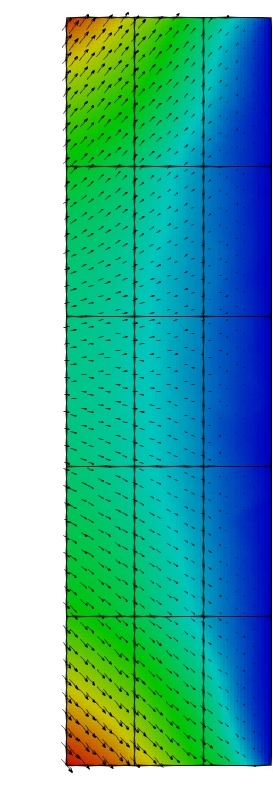}
  \includegraphics[width=0.16\columnwidth]{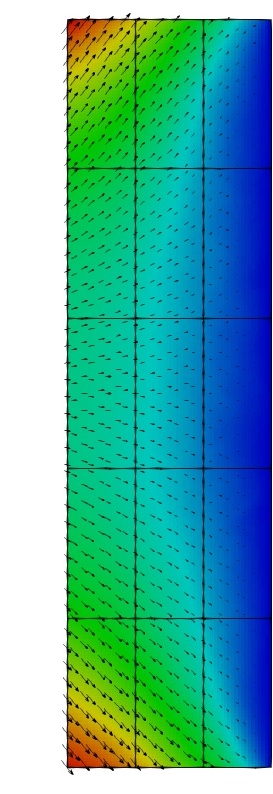}
  \includegraphics[width=0.16\columnwidth]{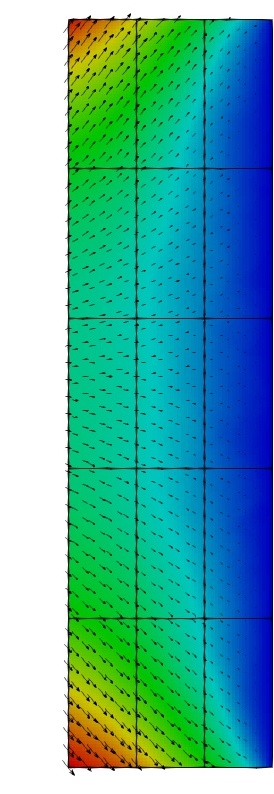}  
  \includegraphics[width=0.16\columnwidth]{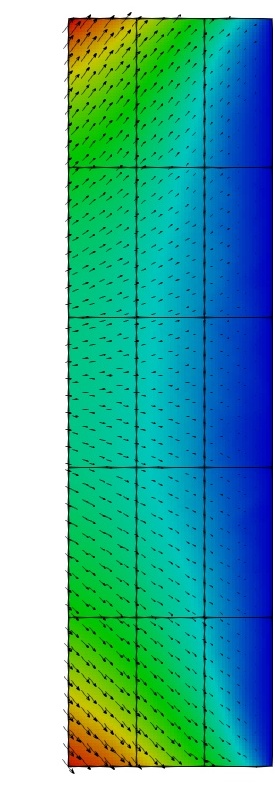}
  \includegraphics[width=0.15\columnwidth]{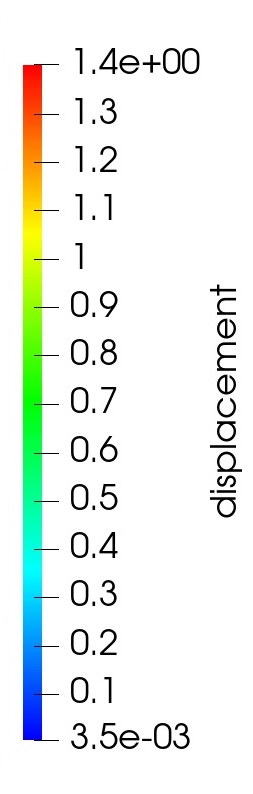}
  \caption{Example 2, displacement vector (arrows) and its magnitude; from left to right: fine scale, single linear mortar per interface, two linear
    mortars per interface, single quadratic mortar per interface (left), two quadratic mortars per interface.}
  \label{fig:Example-2,-solution_plots_last}
\end{figure}

\section{Conclusions}\label{sec:concl}

We presented a multiscale mortar mixed finite element method for the Biot system of poroelasticity in a five-filed fully mixed formulation. The method allows for non-matching subdomain grids at the interfaces, using a composite mortar Lagrange multiplier space that approximates the displacement and pressure on a (possibly coarse) mortar interface grid to impose weakly stress and flux continuity. We established the well-posedness of the method and carried out a multiscale a priori error analysis. The results are robust in the limit of small storativity coefficient. We further presented a non-overlapping domain decomposition algorithm based on a Schur complement reduction of the global system to a (coarse scale) mortar interface problem, which is solved with a Krylov space iterative method. Each iteration requires solving Dirichlet type subdomain problems, which can be performed in parallel. A series of numerical tests illustrates the stability and convergence properties of the method, as well as its computational efficiency. We observed, both theoretically and numerically, that fine scale order convergence can be obtained even for a coarse mortar mesh with a suitable choice of the mortar polynomial degree. An application of the method to a highly heterogeneous benchmark problem illustrates that the multiscale mortar method can achieve comparable accuracy to the fine scale method at a highly reduced computational cost. Moreover, the use of a pre-computed multiscale stress--flux basis further increases the efficiency, making the computational cost independent of the global number of interface degrees of freedom and weakly dependent on the number of time steps.

Several extensions of the presented work are possible. These include combining the multiscale mortar techniques developed here with splitting methods for the Biot system of poroelasticity studied, e.g., in \cite{Ahmed-FS-JCAM,YiSplit,RaduFS1,RaduFS2,RaduFS3,Almani-multirate,dd-biot}, as well as asynchronous and adaptive time stepping using space-time \cite{RaduFS1,space-time,spacetime-SD}, parallel-in-time \cite{RaduFS2}, a posteriori error estimation \cite{Ahmed-FS-JCAM,Ahmed-apost-CMAME}, and multirate \cite{Ahmed-FS-JCAM,Almani-multirate} techniques.

\bibliographystyle{abbrv}
\addcontentsline{toc}{section}{\refname}\bibliography{BiotMortar}

\end{document}